\documentclass[12pt]{amsart}
\usepackage{a4wide}

\usepackage{amssymb,amsmath,stmaryrd}
\usepackage[cmtip,all]{xy}

\newtheorem{theorem}{Theorem}[section]
\newtheorem{lemma}[theorem]{Lemma}
\newtheorem{proposition}[theorem]{Proposition}
\newtheorem{corollary}[theorem]{Corollary}
\newtheorem*{maintheorem}{Main Theorem}

\theoremstyle{remark}
\newtheorem{remark}[theorem]{\it Remark}
\newtheorem{example}[theorem]{\it Example}

\newcommand{\C}{\ensuremath{\mathbb{C}}}
\newcommand{\R}{\ensuremath{\mathbb{R}}}
\newcommand{\tr}{\ensuremath{\mbox{tr}}}
\newcommand{\g}[1]{\ensuremath{\mathfrak{#1}}}
\newcommand{\Exp}{\ensuremath{\mathop{\rm Exp}\nolimits}}
\newcommand{\ad}{\ensuremath{\mathop{\rm ad}\nolimits}}
\newcommand{\Ad}{\ensuremath{\mathop{\rm Ad}\nolimits}}

\newcommand{\spp}{\ensuremath{\g{s}_{\g{p}}^\perp}}

\begin{document}
\title[Hyperpolar Homogeneous Foliations]
{Hyperpolar Homogeneous Foliations\\
on Symmetric Spaces of Noncompact Type}

\author[J. Berndt]{J\"{u}rgen Berndt}
\address{Department of Mathematics,
King's College London, United Kingdom.}
\email{jurgen.berndt@kcl.ac.uk}

\author[J. C. D\'\i{}az-Ramos]{Jos\'{e} Carlos D\'\i{}az-Ramos}
\address{Department of Geometry and Topology,
University of Santiago de Compostela, Spain.}
\email{josecarlos.diaz@usc.es}

\author[H. Tamaru]{Hiroshi Tamaru}
\address{Department of Mathematics, Hiroshima
University, Japan.}
\email{tamaru@math.sci.hiroshima-u.ac.jp}

\thanks{The second author has been supported by a Marie-Curie
European Reintegration Grant (PERG04-GA-2008-239162) and by projects
MTM2009-07756 and INCITE09207151PR (Spain). The third author has been
supported in part by Grant-in-Aid for Young Scientists (B) 20740040,
The Ministry of Education, Culture, Sports, Science and Technology,
Japan.}


%\date{}


\begin{abstract}
A foliation ${\mathcal F}$ on a Riemannian manifold $M$ is
hyperpolar if it admits a flat section, that is, a connected closed
flat submanifold of $M$ that intersects each leaf of ${\mathcal F}$
orthogonally. In this article we classify the hyperpolar homogeneous
foliations on every Riemannian symmetric space $M$ of noncompact
type.

These foliations are constructed as follows. Let $\Phi$ be an
orthogonal subset of a set of simple roots associated with the
symmetric space $M$. Then $\Phi$ determines a horospherical
decomposition $M=F_\Phi^s\times\mathbb{E}^{\mathop{\rm
rank}M-\lvert\Phi\rvert}\times N_\Phi$, where $F_\Phi^s$ is the
Riemannian product of $\lvert\Phi\rvert$ symmetric spaces of rank
one. Every hyperpolar homogeneous foliation on $M$ is isometrically
congruent to the product of the following objects: a particular
homogeneous codimension one foliation on each symmetric space of
rank one in $F_\Phi^s$, a foliation by parallel affine subspaces on
the Euclidean space $\mathbb{E}^{\mathop{\rm
rank}M-\lvert\Phi\rvert}$, and the horocycle subgroup $N_\Phi$ of
the parabolic subgroup of the isometry group of $M$ determined by
$\Phi$.
\end{abstract}

\maketitle

\section{Introduction}

Let $M$ be a connected complete Riemannian manifold and $H$ a
connected closed subgroup of the isometry group $I(M)$ of $M$. Then
each orbit $H \cdot p = \{ h(p) : h \in H\}$, $p \in M$, is a
connected closed submanifold of $M$. A connected complete
submanifold ${\mathcal S}$ of $M$ that meets each orbit of the
$H$-action and intersects the orbit $H \cdot p$ perpendicularly at
each point $p \in {\mathcal S}$ is called a section of the action. A
section ${\mathcal S}$ is always a totally geodesic submanifold of
$M$ (see e.g.\ \cite{HLO}). In general, actions do not admit a
section. The action of $H$ on $M$ is called polar if it has a
section, and it is called hyperpolar if it has a flat section. For
motivation and classification of polar and hyperpolar actions on
Euclidean spaces and symmetric spaces of compact type we refer to
the papers by Dadok \cite{Da85}, Podest\`{a} and Thorbergsson
\cite{PT99}, and Kollross \cite{Ko02}, \cite{Ko07}. If all orbits of
$H$ are principal, then the orbits form a homogeneous foliation
${\mathcal F}$ on $M$. In general, a foliation ${\mathcal F}$ on $M$
is called homogeneous if the subgroup of $I(M)$ consisting of all
isometries preserving ${\mathcal F}$ acts transitively on each leaf
of ${\mathcal F}$. Homogeneous foliations are basic examples of
metric foliations. A homogeneous foliation is called polar resp.\
hyperpolar if its leaves coincide with the orbits of a polar resp.\
hyperpolar action.

An action of the Euclidean space ${\mathbb E}^n$ is polar if and
only if it is hyperpolar. An example of a polar homogeneous
foliation on ${\mathbb E}^n$ is the foliation given by the Euclidean
subspace ${\mathbb E}^k$, $0 < k < n$, and its parallel affine
subspaces. A corresponding section is given by the Euclidean space
${\mathbb E}^{n-k}$ which is perpendicular to ${\mathbb E}^k$ at the
origin $0$. In fact, every polar homogeneous foliation on ${\mathbb
E}^n$ is isometrically congruent to one of these foliations. The
main result of this paper is the classification of all hyperpolar
homogeneous foliations on Riemannian symmetric spaces of noncompact
type. For codimension one foliations this was already achieved by
the first and third author in \cite{BT03}. We mention that on
symmetric spaces of compact type every hyperpolar action has a
singular orbit, and there is no relation between such actions using
duality between symmetric spaces of compact and noncompact type. The
methodology for the classification presented in this paper is
significantly different from the known methodologies in the compact
case. Our methodology is conceptual and based on structure theory of
parabolic subalgebras of real semisimple Lie algebras which is
irrelevant in the compact case.

We will see that these foliations can be constructed from rather
elementary foliations on Euclidean spaces and the hyperbolic spaces
over normed real division algebras. We first describe these
elementary foliations. Each Riemannian symmetric spaces of rank one
is a hyperbolic space ${\mathbb F}H^n$ over a normed real division
algebra ${\mathbb F} \in \{{\mathbb R} , {\mathbb C} , {\mathbb H} ,
{\mathbb O} \}$, where $n \geq 2$, and $n = 2$ if ${\mathbb F} =
{\mathbb O}$. It was proved in \cite{BT04} that on each hyperbolic
space ${\mathbb F}H^n$ there exist exactly two isometric congruency
classes of homogeneous codimension one foliations. One of these two
classes is determined by the horosphere foliation on ${\mathbb
F}H^n$. We denote by ${\mathcal F}_{\mathbb F}^n$ a representative
of the other congruency class, and refer to Section \ref{polhypfol}
for an explicit description.  If $M = {\mathbb F}_1H^{n_1} \times
\ldots \times {\mathbb F}_kH^{n_k}$ is the Riemannian product of $k$
Riemannian symmetric spaces of rank one, then ${\mathcal
F}_{{\mathbb F}_1}^{n_1} \times \ldots \times {\mathcal F}_{{\mathbb
F}_k}^{n_k}$ is a hyperpolar homogeneous foliation on $M$. If $V$ is
a linear subspace of ${\mathbb E}^m$, we denote by ${\mathcal
F}_V^m$ the homogeneous foliation on ${\mathbb E}^m$ whose leaves
are the affine subspaces of ${\mathbb E}^m$ which are parallel to
$V$. We will now explain how these particular foliations lead to the
classification of hyperpolar homogeneous foliations on Riemannian
symmetric spaces of noncompact type.

Let $M = G/K$ be a Riemannian symmetric space of noncompact type,
where $G$ is the connected component of the isometry group of $M$
containing the identity transformation. We denote by $r$ the rank of
$M$. The Lie algebra ${\mathfrak g}$ of $G$ is a semisimple real Lie
algebra. Let ${\mathfrak k}$ be the Lie algebra of $K$, ${\mathfrak
g} = {\mathfrak k} \oplus {\mathfrak p}$ be a Cartan decomposition
of ${\mathfrak g}$, ${\mathfrak a}$ be a maximal abelian subspace of
${\mathfrak p}$, and ${\mathfrak g}_0 \oplus \left(
\bigoplus_{\lambda \in \Sigma} {\mathfrak g}_{\lambda} \right)$ be
the corresponding restricted root space decomposition of ${\mathfrak
g}$. The set $\Sigma$ denotes the corresponding set of restricted
roots. We choose a subset $\Lambda \subset \Sigma$ of simple roots
and denote by $\Sigma^+$ the resulting set of positive restricted
roots in $\Sigma$. It is well known that there is a one-to-one
correspondence between the subsets $\Phi$ of $\Lambda$ and the
conjugacy classes of parabolic subalgebras ${\mathfrak q}_\Phi$ of
${\mathfrak g}$. Let $\Phi$ be a subset of $\Lambda$ and consider
the Langlands decomposition ${\mathfrak q}_\Phi = {\mathfrak m}_\Phi
\oplus {\mathfrak a}_\Phi \oplus {\mathfrak n}_\Phi$ of the
corresponding parabolic subalgebra ${\mathfrak q}_\Phi$ of
${\mathfrak g}$. This determines a corresponding Langlands
decomposition $Q_\Phi = M_\Phi A_\Phi N_\Phi$ of the parabolic
subgroup $Q_\Phi$ of $G$ with Lie algebra ${\mathfrak q}_\Phi$ and a
horospherical decomposition $M = F_\Phi^s \times {\mathbb
E}^{r-r_\Phi} \times N_\Phi$ of the symmetric space $M$. Here,
$r_\Phi$ is equal to the cardinality $|\Phi|$ of the set $\Phi$,
$F_\Phi^s = M_\Phi \cdot o$ is a semisimple Riemannian symmetric
space of noncompact type with rank equal to $r_\Phi$ embedded as a
totally geodesic submanifold in $M$, and ${\mathbb E}^{r-r_\Phi} =
A_\Phi \cdot o$ is an $(r-r_\Phi)$-dimensional Euclidean space
embedded as a totally geodesic submanifold in $M$. Now assume that
$\Phi$ is a subset of $\Lambda$ with the property that any two roots
in $\Phi$ are not connected in the Dynkin diagram of the restricted
root system associated with $\Lambda$. We call such a subset $\Phi$
an orthogonal subset of $\Lambda$. Each simple root $\alpha \in
\Phi$ determines a hyperbolic space ${\mathbb F}_\alpha
H^{n_\alpha}$ embedded in $M$ as a totally geodesic submanifold, and
$F_\Phi^s$ is isometric to the Riemannian product of $r_\Phi$
Riemannian symmetric spaces of rank one,
\begin{equation*}
F_\Phi^s \cong \prod_{\alpha \in \Phi} {\mathbb F}_\alpha
H^{n_\alpha}.
\end{equation*}
We denote by ${\mathcal F}_\Phi$ the hyperpolar homogeneous
foliation on this product of hyperbolic spaces as described above,
that is,
\begin{equation*} {\mathcal F}_\Phi =
\prod_{\alpha \in \Phi} {\mathcal F}_{{\mathbb
F}_\alpha}^{n_\alpha}.
\end{equation*}
We are now in a position to state the main
result of this paper.

\begin{maintheorem}
Let $M$ be a connected Riemannian symmetric space of noncompact
type.
\begin{itemize}
\item[(i)] Let $\Phi$ be an orthogonal subset of $\Lambda$ and $V$ be
a linear subspace of ${\mathbb E}^{r-r_\Phi}$. Then
\begin{equation*}
{\mathcal F}_{\Phi,V} = {\mathcal F}_\Phi \times {\mathcal
F}_V^{r-r_\Phi} \times N_\Phi \subset F_\Phi^s \times {\mathbb
E}^{r-r_\Phi} \times N_\Phi = M
\end{equation*}
is a hyperpolar homogeneous foliation on $M$.

\item[(ii)] Every
hyperpolar homogeneous foliation on $M$ is isometrically congruent
to ${\mathcal F}_{\Phi,V}$ for some orthogonal subset $\Phi$ of
$\Lambda$ and some linear subspace $V$ of ${\mathbb E}^{r-r_\Phi}$.
\end{itemize}
\end{maintheorem}

For $\Phi = \emptyset$ the symmetric space $F_\Phi^s$ consists of a
single point and we need to assume that $\dim V < r$ in this case to
get a proper foliation. The foliation ${\mathcal
F}_{\emptyset,\{0\}}$ is the horocycle foliation on $M$.

We briefly describe how to construct a subgroup of $G$ whose orbits
form the foliation ${\mathcal F}_{\Phi,V}$. Since $A_\Phi$ acts
freely on $M$ and ${\mathbb E}^{r-r_\Phi} = A_\Phi \cdot o$, there
is a canonical identification of ${\mathbb E}^{r-r_\Phi}$ with the
Lie algebra ${\mathfrak a}_\Phi \subset {\mathfrak a}$. We define a
nilpotent subalgebra ${\mathfrak n}$ of ${\mathfrak g}$ by
${\mathfrak n} = {\mathfrak n}_\emptyset$ and put
$\g{a}=\g{a}_\emptyset$. Then the closed subgroup $AN$ of $G$ with
Lie algebra ${\mathfrak a} \oplus {\mathfrak n}$ acts simply
transitively on $M$, and $M$ is isometric to the solvable Lie group
$AN$ equipped with a suitable left-invariant Riemannian metric. Let
$\ell_\Phi$ be an $r_\Phi$-dimensional linear subspace of
${\mathfrak n}$ such that $\dim (\ell_\Phi \cap {\mathfrak
g}_\alpha) = 1$ for all $\alpha \in \Phi$. We denote by ${\mathfrak
a}^\Phi$ the orthogonal complement of ${\mathfrak a}_\Phi$ in
${\mathfrak a}$ and by ${\mathfrak n} \ominus \ell_\Phi$ the
orthogonal complement of $\ell_\Phi$ in ${\mathfrak n}$. Here, the
orthogonal complement is taken with respect to the standard positive
definite inner product on ${\mathfrak g}$ given by the Killing form
on ${\mathfrak g}$ and the Cartan involution on ${\mathfrak g}$
determined by ${\mathfrak k}$. Then
\begin{equation*}
{\mathfrak s}_{\Phi,V} = ({\mathfrak a}^\Phi \oplus V) \oplus
({\mathfrak n} \ominus \ell_\Phi) \subset {\mathfrak a} \oplus
{\mathfrak n}
\end{equation*}
is a subalgebra of ${\mathfrak a} \oplus {\mathfrak n}$. Denote by
$S_{\Phi,V}$ the connected closed subgroup of $AN$ with Lie algebra
${\mathfrak s}_{\Phi,V}$. Then the action of $S_{\Phi,V}$ on $M$ is
hyperpolar and the orbits of this action form the hyperpolar
homogeneous foliation ${\mathcal F}_{\Phi,V}$ on $M$. We will see
later in this paper that for a given set $\Phi$ different choices of
$\ell_\Phi$ lead to isometrically congruent foliations on $M$.

We now describe the contents of this paper in more detail. In
Section \ref{secHomgoeneousFoliations} we show that all homogeneous
foliations on Hadamard manifolds can be produced by isometric
actions of solvable Lie groups all of whose orbits are principal. In
Section \ref{secPreliminaries} we present the aspects of the general
theory of symmetric spaces of noncompact type and of parabolic
subalgebras of real semisimple Lie algebras which are relevant for
our paper. In Section \ref{polhypfol} we prove a necessary and
sufficient Lie algebraic criterion for an isometric Lie group action
inducing a foliation on a symmetric space of noncompact type to be
polar or hyperpolar. Using this criterion we present examples of
polar and of hyperpolar actions on symmetric spaces of noncompact
type. In this section we also prove part (i) of the main theorem,
which is the easiest part of the proof. Section
\ref{secClassification} constitutes the main part of this paper and
contains the proof of part (ii) of the main theorem. Finally, in
Section \ref{secGeometry} we discuss aspects of the geometry of the
leaves of the hyperpolar homogeneous foliations on symmetric spaces
of noncompact type.

\section{Homogeneous foliations on Hadamard manifolds}
\label{secHomgoeneousFoliations}

A simply connected complete Riemannian manifold with nonpositive
sectional curvature is called a {\it Hadamard manifold}.

\begin{proposition}\label{thAllOrbitsPrincipal}
Let $M$ be a Hadamard manifold and $H$ be a connected closed
subgroup of $I(M)$ whose orbits form a homogeneous foliation on $M$.
Then each orbit of $H$ is a principal orbit.
\end{proposition}

\begin{proof}
Assume that there exists an exceptional orbit, that is, a
non-principal orbit whose dimension coincides with the dimension of
the principal orbits. Let $K$ be a maximal compact subgroup of $H$.
By Cartan's Fixed Point Theorem (see e.g.~\cite{E96}, p.~21), $K$
has a fixed point $o\in M$. Since $K$ is maximal compact, the orbit
through $o$ must be exceptional and $K=H_o$. Then $H\cdot o=H/K$ is
diffeomorphic to $\R^k$, where $k$ is the dimension of the foliation
(see for example \cite[p. 148, Theorem 3.4]{OV94}). Since the orbit
$H \cdot o$ is simply connected, the stabilizer $K$ is connected.
The cohomogeneity of the slice representation at $o$ coincides with
the cohomogeneity of  the action of $H$ on $M$, and since all the
orbits of $H$ have the same dimension it follows that the orbits of
the slice representation at $o$ are zero-dimensional. Since $K$ is
connected, it follows that the orbits of the slice representation at
$o$ are points. This means that $K$ acts trivially on the normal
space $\nu_o(H\cdot o)$ of $H \cdot o$ at $o$, which contradicts the
assumption that the orbit $H\cdot o$ is exceptional.
\end{proof}

We will now use the previous result to show that every homogeneous
foliation on a Hadamard manifold can be realized as the orbits of
the action of a closed solvable group  of isometries.

\begin{proposition} \label{ReduceToSolvable}
Let $M$ be a Hadamard manifold and let $H$ be a connected closed
subgroup of $I(M)$ whose orbits form a homogeneous foliation
${\mathcal F}$ on $M$. Then there exists a connected closed solvable
subgroup $S$ of $H$ such that the leaves of ${\mathcal F}$ coincide
with the orbits of $S$.
\end{proposition}

\begin{proof}
Consider a Levi-Malcev decomposition ${\mathfrak h} = {\mathfrak l}
\inplus {\mathfrak r}$ (semidirect sum of Lie algebras) of the Lie
algebra ${\mathfrak h}$ of $H$ into the radical ${\mathfrak r}$ of
${\mathfrak h}$ and a Levi subalgebra ${\mathfrak l}$. The radical
${\mathfrak r}$ is the largest solvable ideal in ${\mathfrak h}$
and ${\mathfrak l}$ is a semisimple subalgebra. Let ${\mathfrak l} =
{\mathfrak k} \oplus {\mathfrak a} \oplus {\mathfrak n}$ (direct sum
of vector spaces) be an Iwasawa decomposition of ${\mathfrak l}$.
Then ${\mathfrak a}$ is an abelian subalgebra of ${\mathfrak l}$,
${\mathfrak n}$ is a nilpotent subalgebra of ${\mathfrak l}$, and
${\mathfrak d} = {\mathfrak a} \inplus {\mathfrak n}$ (semidirect
sum of Lie algebras) is a solvable subalgebra of ${\mathfrak l}$.
Since the semidirect sum of two solvable Lie algebras is again
solvable, the subalgebra ${\mathfrak s} = {\mathfrak d} \inplus
{\mathfrak r}$ (semidirect sum of Lie algebras) is a solvable
subalgebra of ${\mathfrak h}$, and we have ${\mathfrak h} =
{\mathfrak k} \oplus {\mathfrak s}$ (direct sum of vector spaces).
Let $S$ be the connected solvable subgroup of $H$ with Lie algebra
${\mathfrak s}$ and let $K$ be the connected subgroup of $H$ with
Lie algebra $\g{k}$. Since $M$ is a Hadamard manifold, Cartan's
Fixed Point Theorem implies that the
compact group $K$ has a fixed point $o \in M$. Since $H = SK$, it
follows that the orbits $H \cdot o$ and $S \cdot o$ coincide.

By Proposition \ref{thAllOrbitsPrincipal}, the orbit $H \cdot o$ is
a principal orbit of the $H$-action. Let $p$ be a point in $M$ which
does not lie on the principal orbit $H \cdot o$. Since $H \cdot o$
is a closed subset of $M$, there exists a point $q \in H \cdot o$
such that the distance $t$ between $p$ and $q$ minimizes the
distance between $p$ and $H \cdot o$. Since $M$ is complete there
exists a geodesic joining $q$ and $p$, and a standard variational
argument shows that this geodesic intersects the orbit $H \cdot o$
perpendicularly. This proves that every orbit of $H$ is of the form
$H\cdot p$ with $p=\exp_o(\xi)$ and $\xi\in\nu_o(H\cdot o)$. Since
$H \cdot o$ is a principal orbit of the $H$-action on $M$ and
$S\subset H$, the slice representation at $o$ of each of these two
actions is trivial. This fact and $H \cdot o = S \cdot o$ imply
\begin{align*}
S \cdot p & =  \{ s(\exp_o(\xi)):s \in S\} = \{ \exp_{s(o)}(s_*\xi)
:s \in S\} \\ &  = \{ \exp_{h(o)}(h_*\xi):h \in H\} = \{
h(\exp_o(\xi)):h \in H\} = H \cdot p,
\end{align*}
which shows that the actions of $S$ and $H$ are orbit equivalent.

Since $S$ is solvable, its closure $\bar{S}$ in $I(M)$ is a closed
solvable subgroup of $I(M)$ (see e.g.\ \cite{O93}, p.~54, Theorem
5.3). Since the actions of $S$ and $H$ are orbit equivalent, the
orbits of $S$ are closed, and hence by \cite{D}, the actions of $S$
and $\bar{S}$ are orbit equivalent. This finishes the proof of the
proposition.
\end{proof}

\section{Riemannian symmetric spaces of noncompact
type}\label{secPreliminaries}

In this section we present some material about Riemannian symmetric
spaces of noncompact type. We follow \cite{H78} for the theory of
symmetric spaces and \cite{K96} for the theory of semisimple Lie
algebras.

Let $M$ be a connected Riemannian symmetric space of noncompact
type. We denote by $n$ the dimension of $M$ and by $r$ the rank of
$M$. It is well known that $M$ is a Hadamard manifold and therefore
diffeomorphic to $\R^n$. Let $G$ be the connected component of the
isometry group of $M$ containing the identity transformation of $M$.
We fix a point $o\in M$ and denote by $K$ the isotropy subgroup of
$G$ at $o$. We identify $M$ with the homogeneous space $G/K$ in the
usual way and denote by $\g{g}$ and $\g{k}$ the Lie algebra of $G$
and $K$, respectively. Let $B$ be the Killing form of $\g{g}$ and
define $\g{p}$ as the orthogonal complement of $\g{k}$ with respect
to $B$. Then $\g{g}=\g{k}\oplus\g{p}$ is a Cartan decomposition of
$\g{g}$. If $\theta$ is the corresponding Cartan involution, we can
define a positive definite inner product on $\g{g}$ by $\langle
X,Y\rangle=-B(X,\theta Y)$ for all $X,Y\in\g{g}$. We identify
$\g{p}$ with $T_oM$ and we normalize the Riemannian metric on $M$ so
that its restriction to $T_o M\times T_o M=\g{p}\times\g{p}$
coincides with $\langle\,\cdot\,,\,\cdot\,\rangle$.

We now fix a maximal abelian subspace $\g{a}\subset\g{p}$ and denote
by $\g{a}^*$ the dual space of $\g{a}$. For each $\lambda\in\g{a}^*$
we define $\g{g}_\lambda=\{X\in\g{g}:\ad(H)X=\lambda(H)X \mbox{ for
all }H\in\g{a}\}$. We say that $0 \neq \lambda\in\g{a}^*$ is a
restricted root if $\g{g}_\lambda\neq \{0\}$, and we denote by
$\Sigma$ the set of all restricted roots. Since $\g{a}$ is abelian,
$\ad(\g{a})$ is a commuting family of self-adjoint linear
transformations of $\g{g}$. This implies that the subset
$\Sigma\subset\g{a}^*$ of all restricted roots is nonempty, finite
and $\g{g}=\g{g}_0\oplus(\bigoplus_{\lambda\in\Sigma}\g{g}_\lambda)$
is an orthogonal direct sum called the restricted root space
decomposition of $\g{g}$ determined by $\g{a}$. Here,
$\g{g}_0=\g{k}_0\oplus\g{a}$, where $\g{k}_0=Z_{\g{k}}(\g{a})$ is
the centralizer of $\g{a}$ in $\g{k}$. For each $\lambda\in\g{a}^*$
let $H_\lambda\in\g{a}$ denote the dual vector in $\g{a}$ with
respect to the Killing form, that is, $\lambda(H)=\langle
H_\lambda,H\rangle$ for all $H\in\g{a}$. This also defines an inner
product on $\g{a}^*$ by setting $\langle\lambda,\mu\rangle=\langle
H_\lambda,H_\mu\rangle$ for all $\lambda,\mu\in\g{a}^*$.

We now introduce an ordering in $\Sigma$ and denote by $\Sigma^+$
the resulting set of positive roots. We denote by
$\Lambda=\{\alpha_1,\dots,\alpha_r\}$ the set of simple roots of
$\Sigma^+$ in line with the notation used in \cite{K96}. By
$\{H^1,\dots,H^r\} \subset \g{a}$ we denote the dual basis of
$\{\alpha_1,\dots,\alpha_r\}$, that is, $\alpha_i(H^j)=\delta_i^j$,
where $\delta$ is the Kronecker delta. Then each root
$\lambda\in\Sigma$ can be written as $\lambda=\sum_{i=1}^r
c_i\alpha_i$ where all the $c_i$ are integers, and they are all
nonpositive or nonnegative depending on whether the root is negative
or positive. The sum $\sum_{i=1}^r c_i$ is called the level of the
root.

The subspace $\g{n}=\bigoplus_{\lambda\in\Sigma^+}\g{g}_\lambda$ of
$\g{g}$ is a nilpotent subalgebra of $\g{g}$. Moreover,
$\g{a}\oplus\g{n}$ is a solvable subalgebra of $\g{g}$ with
$[\g{a}\oplus\g{n},\g{a}\oplus\g{n}]=\g{n}$. We can write $\g{g}$ as
the direct sum of vector subspaces
$\g{g}=\g{k}\oplus\g{a}\oplus\g{n}$, the so-called Iwasawa
decomposition of $\g{g}$. Let $A$, $N$ and $AN$ be the connected
subgroups of $G$ with Lie algebra $\g{a}$, $\g{n}$ and
$\g{a}\oplus\g{n}$, respectively. All these subgroups are simply
connected and $G$ is diffeomorphic to the product $K\times A\times
N$. Moreover, the simply connected solvable Lie group
$AN$ acts simply transitively on $M$.
%Hence $M$ is isometric to the connected, simply connected
%solvable Lie group $AN$ equipped with the left-invariant Riemannian
%metric that is induced from the inner product
%$\langle\,\cdot\,,\,\cdot\,\rangle$. Consider
%$X,Y,Z\in\g{a}\oplus\g{n}$ as left-invariant vector fields on $M$.
%If $\nabla$ denotes the Levi-Civita covariant derivative of $M =
%AN$, the equality $\langle\ad(X)Y,Z\rangle=-\langle\ad(\theta
%X)Z,Y\rangle$ implies that the Koszul formula can be written as
%\begin{equation*}
%2\langle\nabla_XY,Z\rangle=\langle[X,Y]+(1-\theta)[\theta
%X,Y],Z\rangle.
%\end{equation*}
We can then equip  $AN$ with a
left-invariant Riemannian metric
$\langle\,\cdot\,,\,\cdot\,\rangle_{AN}$ so that $M$ and $AN$
become isometric. This metric is determined by $\langle
H_1+X_1,H_2+X_2\rangle_{AN}=\langle H_1,H_2\rangle+(1/2)\langle
X_1,X_2\rangle$ for all $H_1,H_2\in\g{a}$ and
$X_1,X_2\in\g{n}$ (see e.g.\ the proof of Proposition 4.4
in \cite{Ta07}). Consider $X,Y,Z\in\g{a}\oplus\g{n}$ as
left-invariant vector fields on $AN$. If $\nabla$ denotes the
Levi-Civita covariant derivative of $M = AN$, the equality
$\langle\ad(X)Y,Z\rangle=-\langle\ad(\theta X)Z,Y\rangle$
implies that the Koszul formula can be written as
\begin{equation*}
4\langle\nabla_XY,Z\rangle_{AN}=\langle[X,Y]+(1-\theta)[\theta
X,Y],Z\rangle.
\end{equation*}

We will now associate to each subset $\Phi$ of $\Lambda$ a parabolic
subalgebra ${\mathfrak q}_\Phi$ of ${\mathfrak g}$. Let $\Phi$ be a
subset of $\Lambda$. We denote by $\Sigma_\Phi$ the root subsystem
of $\Sigma$ generated by $\Phi$, that is, $\Sigma_\Phi$ is the
intersection of $\Sigma$ and the linear span of $\Phi$, and put
$\Sigma_\Phi^+ = \Sigma_\Phi \cap \Sigma^+$. Let
\begin{equation*}
{\mathfrak l}_\Phi = {\mathfrak g}_0 \oplus \left(\bigoplus_{\lambda
\in \Sigma_\Phi} {\mathfrak g}_{\lambda}\right) \ {\rm and}\
{\mathfrak n}_\Phi = \bigoplus_{\lambda \in \Sigma^+\setminus
\Sigma_\Phi^+} {\mathfrak g}_{\lambda}.
\end{equation*}
Then ${\mathfrak l}_\Phi$ is a reductive subalgebra of ${\mathfrak
g}$ and ${\mathfrak n}_\Phi$ is a nilpotent subalgebra of
${\mathfrak g}$. Let
\begin{equation*}
{\mathfrak a}_\Phi = \bigcap_{\alpha \in \Phi} {\rm ker}\,\alpha \ \
{\rm and}\ \ {\mathfrak a}^\Phi = {\mathfrak a} \ominus {\mathfrak
a}_\Phi.
\end{equation*}
Then ${\mathfrak a}_\Phi$ is an abelian subalgebra of ${\mathfrak
g}$ and ${\mathfrak l}_\Phi$ is the centralizer and the normalizer
of ${\mathfrak a}_\Phi$ in ${\mathfrak g}$. The abelian subalgebra
${\mathfrak a}_\Phi$ is also known as the split component of the
reductive Lie algebra ${\mathfrak l}_\Phi$. Since $[{\mathfrak
l}_\Phi,{\mathfrak n}_\Phi] \subset {\mathfrak n}_\Phi$,
\begin{equation*}
{\mathfrak q}_\Phi = {\mathfrak l}_\Phi \oplus {\mathfrak n}_\Phi
\end{equation*}
is a subalgebra of ${\mathfrak g}$, the so-called parabolic
subalgebra of ${\mathfrak g}$ associated with the subset $\Phi$ of
$\Lambda$. The subalgebra ${\mathfrak l}_\Phi = {\mathfrak q}_\Phi
\cap \theta({\mathfrak q}_\Phi)$ is a reductive Levi subalgebra of
${\mathfrak q}_\Phi$ and ${\mathfrak n}_\Phi$ is the unipotent
radical of ${\mathfrak q}_\Phi$, and therefore the decomposition
${\mathfrak q}_\Phi = {\mathfrak l}_\Phi \oplus {\mathfrak n}_\Phi$
is a semidirect sum of the Lie algebras ${\mathfrak l}_\Phi$ and
${\mathfrak n}_\Phi$. The decomposition ${\mathfrak q}_\Phi =
{\mathfrak l}_\Phi \oplus {\mathfrak n}_\Phi$ is known as the
Chevalley decomposition of the parabolic subalgebra ${\mathfrak
q}_\Phi$.

We now define a reductive subalgebra ${\mathfrak m}_\Phi$ of
${\mathfrak g}$ by
\begin{equation*}
{\mathfrak m}_\Phi = {\mathfrak l}_\Phi \ominus {\mathfrak a}_\Phi =
{\mathfrak k}_0 \oplus {\mathfrak a}^\Phi \oplus
\left(\bigoplus_{\lambda \in \Sigma_\Phi} {\mathfrak
g}_{\lambda}\right).
\end{equation*}
The subalgebra ${\mathfrak m}_\Phi$ normalizes ${\mathfrak a}_\Phi
\oplus {\mathfrak n}_\Phi$, and
\begin{equation*}
{\mathfrak g}_\Phi = [{\mathfrak m}_\Phi,{\mathfrak m}_\Phi] =
[{\mathfrak l}_\Phi,{\mathfrak l}_\Phi]
\end{equation*}
is a semisimple subalgebra of ${\mathfrak g}$. The center
${\mathfrak z}_\Phi$ of ${\mathfrak m}_\Phi$ is contained in
${\mathfrak k}_0$ and induces the direct sum decomposition
${\mathfrak m}_\Phi = {\mathfrak z}_\Phi \oplus {\mathfrak g}_\Phi$.
The decomposition
\begin{equation*}
{\mathfrak q}_\Phi = {\mathfrak m}_\Phi \oplus {\mathfrak a}_\Phi
\oplus {\mathfrak n}_\Phi
\end{equation*}
is known as the Langlands decomposition of the parabolic subalgebra
${\mathfrak q}_\Phi$.

For $\Phi = \emptyset$ we obtain ${\mathfrak l}_\emptyset =
{\mathfrak g}_0$, ${\mathfrak m}_\emptyset = {\mathfrak k}_0$,
${\mathfrak a}_\emptyset = {\mathfrak a}$ and ${\mathfrak
n}_\emptyset = {\mathfrak n}$. In this case ${\mathfrak q}_\emptyset
= {\mathfrak k}_0 \oplus {\mathfrak a} \oplus {\mathfrak n} =
{\mathfrak g}_0 \oplus {\mathfrak n}$ is a minimal parabolic
subalgebra of ${\mathfrak g}$. For $\Phi = \Lambda$ we obtain
${\mathfrak l}_\Lambda = {\mathfrak m}_\Lambda = {\mathfrak g}$ and
${\mathfrak a}_\Lambda = {\mathfrak n}_\Lambda = \{0\}$. Each
parabolic subalgebra of ${\mathfrak g}$ is conjugate in ${\mathfrak
g}$ to ${\mathfrak q}_\Phi$ for some subset $\Phi$ of $\Lambda$. The
set of conjugacy classes of parabolic subalgebras of ${\mathfrak g}$
therefore has $2^r$ elements. Two parabolic subalgebras ${\mathfrak
q}_{\Phi_1}$ and ${\mathfrak q}_{\Phi_2}$ of ${\mathfrak g}$ are
conjugate in the full automorphism group ${\rm Aut}({\mathfrak g})$
of ${\mathfrak g}$ if and only if there exists an automorphism $F$
of the Dynkin diagram associated to $\Lambda$ with $F(\Phi_1) =
\Phi_2$. Every parabolic subalgebra contains a minimal parabolic
subalgebra.

Each parabolic subalgebra ${\mathfrak q}_\Phi$ determines a
gradation of ${\mathfrak g}$. For this we define $H^\Phi =
\sum_{\alpha_i \in \Lambda \setminus \Phi} H^i$ and put ${\mathfrak
g}_\Phi ^k = \bigoplus_{\lambda \in \Sigma \cup
\{0\},\lambda(H^\Phi) = k} {\mathfrak g}_\lambda$. Then ${\mathfrak
g} = \bigoplus_{k \in {\mathbb Z}} {\mathfrak g}_\Phi^k$ is a
gradation of ${\mathfrak g}$ with ${\mathfrak g}_\Phi^0 = {\mathfrak
l}_\Phi$, $\bigoplus_{k > 0} {\mathfrak g}_\Phi^k = {\mathfrak
n}_\Phi$ and $\bigoplus_{k \geq 0} {\mathfrak g}_\Phi^k = {\mathfrak
q}_\Phi$. The vector $H^\Phi \in {\mathfrak a}$ is called the
characteristic element of the gradation. The Cartan involution
$\theta$ acts grade-reversing on the gradation, that is, we have
$\theta {\mathfrak g}_\Phi^k = {\mathfrak g}_\Phi^{-k}$ for all $k
\in {\mathbb Z}$. Moreover, this gradation is of type $\alpha_0$,
that is, ${\mathfrak g}_\Phi^{k+1} = [{\mathfrak
g}_\Phi^1,{\mathfrak g}_\Phi^k]$  and ${\mathfrak g}_\Phi^{-k-1} =
[{\mathfrak g}_\Phi^{-1},{\mathfrak g}_\Phi^{-k}]$ holds for all $k
> 0$ (see e.g.\ \cite{KA88}). If $\lambda$ is the
highest root in $\Sigma$ and $m_\Phi = \lambda(H^\Phi)$, we have
${\mathfrak g}_\Phi^{m_\Phi} \neq \{0\}$ and ${\mathfrak g}_\Phi^{k}
= \{0\}$ for all $k > m_\Phi$. For $\Phi = \emptyset$ we have
${\mathfrak n}_\emptyset = {\mathfrak n}$, and we also use the
notation ${\mathfrak n}^k = {\mathfrak g}_\emptyset^k$ for all $k >
0$. Thus we have a gradation ${\mathfrak n} = \bigoplus_{k =
1}^{m_\emptyset} {\mathfrak n}^k $ of ${\mathfrak n}$ which is
generated by ${\mathfrak n}^1$. Note that $m_\emptyset$ is the level
of the highest root in $\Sigma^+$. For each $k
> 0$ we define
\begin{equation*} {\mathfrak p}^k = {\mathfrak p} \cap \left(
{\mathfrak g}_\emptyset^k \oplus {\mathfrak g}_\emptyset^{-k}
\right), \end{equation*} which gives a direct sum decomposition
${\mathfrak p} = {\mathfrak a} \oplus \left( \bigoplus_{k =
1}^{m_\emptyset} {\mathfrak p}^k \right)$.

For each $\lambda \in \Sigma$ we define
\begin{equation*}
{\mathfrak k}_\lambda = {\mathfrak k} \cap ({\mathfrak g}_\lambda
\oplus {\mathfrak g}_{-\lambda})\ \ {\rm and}\ \ {\mathfrak
p}_\lambda = {\mathfrak p} \cap ({\mathfrak g}_\lambda \oplus
{\mathfrak g}_{-\lambda}).
\end{equation*}
Then we have ${\mathfrak p}_\lambda = {\mathfrak p}_{-\lambda}$,
${\mathfrak k}_\lambda = {\mathfrak k}_{-\lambda}$ and ${\mathfrak
p}_\lambda \oplus {\mathfrak k}_\lambda = {\mathfrak g}_\lambda
\oplus {\mathfrak g}_{-\lambda}$ for all $\lambda \in \Sigma$. It is
easy to see that the subspaces
\begin{equation*}
{\mathfrak p}_\Phi = {\mathfrak l}_\Phi \cap {\mathfrak p} =
{\mathfrak a} \oplus \left( \bigoplus_{\lambda \in \Sigma_{\Phi}}
{\mathfrak p}_\lambda \right) \ {\rm and}\ {\mathfrak p}_\Phi^s =
{\mathfrak m}_\Phi \cap {\mathfrak p} = {\mathfrak g}_\Phi \cap
{\mathfrak p} = {\mathfrak a}^\Phi \oplus \left( \bigoplus_{\lambda
\in \Sigma_{\Phi}} {\mathfrak p}_\lambda \right)
\end{equation*}
are Lie triple systems in ${\mathfrak p}$. We define a subalgebra
${\mathfrak k}_\Phi$ of ${\mathfrak k}$ by
\begin{equation*}
{\mathfrak k}_\Phi = {\mathfrak q}_\Phi \cap {\mathfrak k} =
{\mathfrak l}_\Phi \cap {\mathfrak k} = {\mathfrak m}_\Phi \cap
{\mathfrak k} = {\mathfrak k}_0 \oplus \left( \bigoplus_{\lambda \in
\Sigma_{\Phi}} {\mathfrak k}_\lambda \right).
\end{equation*}
Then ${\mathfrak g}_\Phi = ({\mathfrak g}_\Phi \cap {\mathfrak
k}_\Phi) \oplus {\mathfrak p}_\Phi^s$ is a Cartan decomposition of
the semisimple subalgebra ${\mathfrak g}_\Phi$ of ${\mathfrak g}$
and ${\mathfrak a}^\Phi$ is a maximal abelian subspace of
${\mathfrak p}_\Phi^s$. If we define $({\mathfrak g}_\Phi)_0 =
({\mathfrak g}_\Phi \cap {\mathfrak k}_0) \oplus {\mathfrak
a}^\Phi$, then ${\mathfrak g}_\Phi = ({\mathfrak g}_\Phi)_0 \oplus
\left(\bigoplus_{\lambda \in \Sigma_\Phi} {\mathfrak
g}_{\lambda}\right)$ is the restricted root space decomposition of
${\mathfrak g}_\Phi$ with respect to ${\mathfrak a}^\Phi$ and $\Phi$
is the corresponding set of simple roots. Since ${\mathfrak m}_\Phi
= {\mathfrak z}_\Phi \oplus {\mathfrak g}_\Phi$ and ${\mathfrak
z}_\Phi \subset {\mathfrak k}_0$, we see that ${\mathfrak g}_\Phi
\cap {\mathfrak k}_0 = {\mathfrak k}_0 \ominus {\mathfrak z}_\Phi$.

We now relate these algebraic constructions to the geometry of the
symmetric space $M$. Let $\Phi$ be a subset of $\Lambda$ and $r_\Phi
= |\Phi|$. We denote by $A_\Phi$ the connected abelian subgroup of
$G$ with Lie algebra ${\mathfrak a}_\Phi$ and by $N_\Phi$ the
connected nilpotent subgroup of $G$ with Lie algebra ${\mathfrak
n}_\Phi$. The centralizer $L_\Phi = Z_G({\mathfrak a}_\Phi)$ of
${\mathfrak a}_\Phi$ in $G$ is a reductive subgroup of $G$ with Lie
algebra ${\mathfrak l}_\Phi$. The subgroup $A_\Phi$ is contained in
the center of $L_\Phi$. The subgroup $L_\Phi$ normalizes $N_\Phi$
and $Q_\Phi = L_\Phi N_\Phi$ is a subgroup of $G$ with Lie algebra
${\mathfrak q}_\Phi$. The subgroup $Q_\Phi$ coincides with the
normalizer $N_G({\mathfrak l}_\Phi \oplus {\mathfrak n}_\Phi)$ of
${\mathfrak l}_\Phi \oplus {\mathfrak n}_\Phi$ in $G$, and hence
$Q_\Phi$ is a closed subgroup of $G$. The subgroup $Q_\Phi$ is the
parabolic subgroup of $G$ associated with the subsystem $\Phi$ of
$\Lambda$.

Let $G_\Phi$ be the connected subgroup of $G$ with Lie algebra
${\mathfrak g}_\Phi$. Since ${\mathfrak g}_\Phi$ is semisimple,
$G_\Phi$ is a semisimple subgroup of $G$. The intersection $K_\Phi$
of $L_\Phi$ and $K$, i.e. $K_\Phi = L_\Phi \cap K$, is a maximal
compact subgroup of $L_\Phi$ and ${\mathfrak k}_\Phi$ is the Lie
algebra of $K_\Phi$. The adjoint group ${\rm Ad}(L_\Phi)$ normalizes
${\mathfrak g}_\Phi$, and consequently $M_\Phi = K_\Phi G_\Phi$ is a
subgroup of $L_\Phi$. One can show that $M_\Phi$ is a closed
reductive subgroup of $L_\Phi$, $K_\Phi$ is a maximal compact
subgroup of $M_\Phi$, and the center $Z_\Phi$ of $M_\Phi$ is a
compact subgroup of $K_\Phi$. The Lie algebra of $M_\Phi$ is
${\mathfrak m}_\Phi$ and $L_\Phi$ is isomorphic to the Lie group
direct product $M_\Phi \times A_\Phi$, i.e. $L_\Phi = M_\Phi \times
A_\Phi$. For this reason $A_\Phi$ is called the split component of
$L_\Phi$. The parabolic subgroup $Q_\Phi$ acts transitively on $M$
and the isotropy subgroup at $o$ is $K_\Phi$, that is, $M =
Q_\Phi/K_\Phi$.

Since ${\mathfrak g}_\Phi = ({\mathfrak g}_\Phi \cap {\mathfrak
k}_\Phi) \oplus {\mathfrak p}_\Phi^s$ is a Cartan decomposition of
the semisimple subalgebra ${\mathfrak g}_\Phi$, we have $[{\mathfrak
p}_\Phi^s,{\mathfrak p}_\Phi^s] = {\mathfrak g}_\Phi \cap {\mathfrak
k}_\Phi$. Thus $G_\Phi$ is the connected closed subgroup of $G$ with
Lie algebra $[{\mathfrak p}_\Phi^s,{\mathfrak p}_\Phi^s] \oplus
{\mathfrak p}_\Phi^s$. Since ${\mathfrak p}_\Phi^s$ is a Lie triple
system in ${\mathfrak p}$, the orbit $F_\Phi^s = G_\Phi \cdot o$ of
the $G_\Phi$-action on $M$ containing $o$ is a connected totally
geodesic submanifold of $M$ with $T_oF_\Phi^s = {\mathfrak
p}_\Phi^s$. If $\Phi = \emptyset$, then $F_\emptyset^s = \{o\}$,
otherwise $F_\Phi^s$ is a Riemannian symmetric space of noncompact
type and ${\rm rank}(F_\Phi^s) = r_\Phi$, and
\begin{equation*}
F_\Phi^s = G_\Phi \cdot o = G_\Phi/(G_\Phi\cap K_\Phi) = M_\Phi
\cdot o = M_\Phi/K_\Phi.
\end{equation*}
The submanifold $F_\Phi^s$ is also known
as a boundary component of $M$ in the context of the maximal Satake
compactification of $M$ (see e.g.\ \cite{BJ}).

Clearly, ${\mathfrak a}_\Phi$ is a Lie triple system as well, and
the corresponding totally geodesic submanifold is a Euclidean space
\begin{equation*}
{\mathbb E}^{r-r_\Phi} = A_\Phi \cdot o.
\end{equation*}
Since the action of $A_\Phi$ on $M$ is free and $A_\Phi$ is simply
connected, we can identify ${\mathbb E}^{r-r_\Phi}$, $A_\Phi$ and
${\mathfrak a}_\Phi$ canonically. This identification will be used
throughout this paper.

Finally, ${\mathfrak p}_\Phi ={\mathfrak p}_\Phi^s \oplus {\mathfrak
a}_\Phi $ is a Lie triple system, and the corresponding totally
geodesic submanifold $F_\Phi$ is the symmetric space
\begin{equation*}
F_\Phi = L_\Phi \cdot o = L_\Phi/K_\Phi = (M_\Phi \times
A_\Phi)/K_\Phi =
 F_\Phi^s \times {\mathbb E}^{r-r_\Phi}.
\end{equation*}

The submanifolds $F_\Phi$ and $F_\Phi^s$ have a natural geometric
interpretation. Denote by $\bar{C}^+(\Lambda) \subset {\mathfrak a}$
the closed positive Weyl chamber which is determined by the simple
roots $\Lambda$. Let $Z$ be nonzero vector in $\bar{C}^+(\Lambda)$
such that $\alpha(Z) = 0$ for all $\alpha \in \Phi$ and $\alpha(Z) >
0$ for all $\alpha \in \Lambda \setminus \Phi$, and consider the
geodesic $\gamma_Z(t) = {\rm Exp}(tZ) \cdot o$ in $M$ with
$\gamma_Z(0) = o$ and $\dot{\gamma}_Z(0) = Z$. The totally geodesic
submanifold $F_\Phi$ is the union of all geodesics in $M$ which are
parallel to $\gamma_Z$, and $F_\Phi^s$ is the semisimple part of
$F_\Phi$ in the de Rham decomposition of $F_\Phi$ (see e.g.\
\cite{E96}, Proposition 2.11.4  and Proposition 2.20.10).

The group $Q_\Phi$ is diffeomorphic to the product $M_\Phi\times
A_\Phi\times N_\Phi$. This analytic diffeomorphism induces an
analytic diffeomorphism between $F_\Phi^s \times {\mathbb
E}^{r-r_\Phi} \times N_\Phi$ and $M$ known as a horospherical
decomposition of the symmetric space $M$.

\section{Polar and hyperpolar foliations}\label{polhypfol}

We first prove an algebraic characterization of polar actions and of
hyperpolar actions on Riemannian symmetric spaces of noncompact type
(see also Proposition 4.1 in \cite{Ko07}), and then present some
examples.

\begin{theorem}\label{thPolar}
Let $M = G/K$ be a Riemannian symmetric space of noncompact type and
$H$ be a connected closed subgroup of $G$ whose orbits form a
homogeneous foliation ${\mathcal F}$ on $M$. Consider the
corresponding Cartan decomposition $\g{g} = \g{k} \oplus \g{p}$ and
define
\begin{equation*}
{\g{h}_{\g{p}}^\perp}=\{\xi\in\g{p}:\langle\xi,Y\rangle=0\mbox{ for
all }Y\in\g{h}\}.
\end{equation*}
Then the following statements hold:
\begin{itemize}
\item[(i)] The action of $H$ on $M$ is polar if and only
if ${\g{h}_{\g{p}}^\perp}$
is a Lie triple system in $\g{p}$ and $\g{h}$ is orthogonal to the
subalgebra
$[{\g{h}_{\g{p}}^\perp},{\g{h}_{\g{p}}^\perp}]
\oplus{\g{h}_{\g{p}}^\perp}$
of $\g{g}$.

\item[(ii)] The action of $H$ on $M$ is hyperpolar if and only if
${\g{h}_{\g{p}}^\perp}$ is an abelian subspace of $\g{p}$.
\end{itemize}
In both cases, let ${H_{\g{p}}^\perp}$ be the connected subgroup of
$G$ with Lie algebra
$[{\g{h}_{\g{p}}^\perp},{\g{h}_{\g{p}}^\perp}]
\oplus{\g{h}_{\g{p}}^\perp}$.
Then the orbit ${\mathcal S} = {H_{\g{p}}^\perp} \cdot o$ is a
section of the $H$-action on $M$.
\end{theorem}

\begin{proof}
Statement (ii) is an obvious consequence of statement (i). So we
proceed with proving (i).

If the action of $H$ on $M$ is polar, then ${\g{h}_{\g{p}}^\perp}$
is a Lie triple system by definition of a polar action. We now
assume that ${\g{h}_{\g{p}}^\perp}$ is a Lie triple system. We have
to show that the action of $H$ on $M$ is polar if and only if
$\g{h}$ is orthogonal to
$[{\g{h}_{\g{p}}^\perp},{\g{h}_{\g{p}}^\perp}]
\oplus{\g{h}_{\g{p}}^\perp}$.
Since ${\g{h}_{\g{p}}^\perp}$ is a Lie triple system, the orbit
${\mathcal S} = {H_{\g{p}}^\perp} \cdot o$ is a connected complete
totally geodesic submanifold of $M$. Let $p$ be a point in $M$ which
does not lie on the orbit $H \cdot o$. Since $H \cdot o$ is a closed
submanifold of $M$, there exists a point $q \in H \cdot o$ such that
the distance between $p$ and $q$ is equal to the distance between
$p$ and $H \cdot o$. Since $M$ is complete, there exists a geodesic
in $M$ from $p$ to $q$ such that the distance from $p$ to $q$ can be
measured along this geodesic. A standard variational argument shows
that this geodesic intersects $H \cdot o$ perpendicularly. It
follows now easily that ${\mathcal S}$ intersects each orbit. Since
$H$ induces a foliation, it therefore remains to show that $T_p(H
\cdot p)$ and $T_p{\mathcal S}$ are orthogonal for each $p \in
{\mathcal S}$ if and only if $\g{h}$ is orthogonal to
$[{\g{h}_{\g{p}}^\perp},{\g{h}_{\g{p}}^\perp}]
\oplus{\g{h}_{\g{p}}^\perp}$.

Let $\gamma$ be the geodesic in ${\mathcal S}$ with $\gamma(0) = o$
and $\dot{\gamma}(0) = \xi \in {\g{h}_{\g{p}}^\perp}$, and assume
that $\xi \neq 0$. For $X \in \g{h}$ and $\eta \in
{\g{h}_{\g{p}}^\perp}$ we denote by $X^*$ and $\eta^*$ the Killing
vector fields on $M$ that are induced from $X$ and $\eta$,
respectively. Then we have
\begin{equation*}
T_{\gamma(t)}(H \cdot \gamma(t)) = \{ X^*_{\gamma(t)}:X \in \g{h}\}\
,\ T_{\gamma(t)}{\mathcal S} = \{\eta^*_{\gamma(t)} : \eta \in
{\g{h}_{\g{p}}^\perp} \}.
\end{equation*}
The restrictions of two
such Killing vector fields $X^*$ and $\eta^*$ to $\gamma$ satisfy
the equation
\begin{equation*}
\left.\frac{d}{dt}\right|_{t=0} \langle
X^*_{\gamma(t)},\eta^*_{\gamma(t)} \rangle = \langle
[\xi^*,X^*]_o,\eta^*_o \rangle + \langle X^*_o , [\xi^*,\eta^*]_o
\rangle = - \langle [ \xi,\eta] , X \rangle,
\end{equation*}
using the facts that $[\xi^*,X^*] = -[\xi,X]^*$, $[\xi^*,\eta^*]
= -[\xi,\eta]^*$, $[\xi,\eta] \in {\mathfrak k}$,
and that $\ad(\xi)$ is a self-adjoint endomorphism
on $\g{g}$. From this it easily follows that $\g{h}$ is orthogonal
to
$[{\g{h}_{\g{p}}^\perp},{\g{h}_{\g{p}}^\perp}]
\oplus{\g{h}_{\g{p}}^\perp}$
if $T_p(H \cdot p)$ and $T_p{\mathcal S}$ are orthogonal for each $p
\in {\mathcal S}$. Conversely, assume that $\g{h}$ is orthogonal to
$[{\g{h}_{\g{p}}^\perp},{\g{h}_{\g{p}}^\perp}]
\oplus{\g{h}_{\g{p}}^\perp}$.
Then, for each $X \in \g{h}$, the restriction $X^*_\gamma$ of the
Killing vector field $X^*$ to $\gamma$ is the Jacobi vector field
along $\gamma$ with initial values $X^*_\gamma(0) = X^*_o =
X_{\g{p}} \in \g{h}_{\g{p}}$ and $(X^*_\gamma)^\prime(0) =
[\xi^*,X^*]_o = -[\xi,X]^*_o = -[\xi,X]_{\g{p}}  \in \g{h}_{\g{p}}$.
Here the subscript indicates orthogonal projection onto $\g{p}$.
Since both initial values are in $\g{h}_{\g{p}} = \nu_o{\mathcal
S}$, it follows that $X^*_\gamma$ takes values in the normal bundle
of ${\mathcal S}$ along $\gamma$. This implies that $T_{\gamma(t)}(H
\cdot \gamma(t))$ and $T_{\gamma(t)}{\mathcal S}$ are orthogonal for
each $t \in {\mathbb R}$. Since this holds for each geodesic
$\gamma$ in ${\mathcal S}$ with $\gamma(0) = o$ and $\dot{\gamma}(0)
= \xi \in {\g{h}_{\g{p}}^\perp}$, $\xi \neq 0$, we conclude that
$T_p(H \cdot p)$ and $T_p{\mathcal S}$ are orthogonal for each $p
\in {\mathcal S}$.
\end{proof}

We will use the previous result to show polarity and hyperpolarity
of certain actions.

\begin{proposition}
Let $M$ be a Riemannian symmetric space of noncompact type and
consider a horospherical decomposition $F_\Phi^s \times {\mathbb
E}^{r-r_\Phi} \times N_\Phi$ of $M$. Let $V$ be a linear subspace of
${\mathbb E}^{r-r_\Phi}$ and assume that $(\Phi,V) \neq
(\emptyset,{\mathbb E}^r)$. Then the action of $V \times N_\Phi
\subset A_\Phi \times N_\Phi$ on $M$ is polar and $F_\Phi^s \times
({\mathbb E}^{r-r_\Phi} \ominus V)$ is a section of this action.
Moreover, the action of $V \times N_\Phi$ on $M$ is hyperpolar if
and only if $\Phi = \emptyset$.
\end{proposition}

\begin{proof}
The subspace $(V \oplus \g{n}_\Phi)^\perp_{\mathfrak p} = \g{p}_\Phi
\ominus V$ of $\g{p}$ is a Lie triple system and $F_\Phi^s \times
({\mathbb E}^{r-r_\Phi} \ominus V)$ is the connected complete
totally geodesic submanifold of $M$ corresponding to $\g{p}_\Phi
\ominus V$. Next, we have $[(V \oplus \g{n}_\Phi)^\perp_{\mathfrak
p} , (V \oplus \g{n}_\Phi)^\perp_{\mathfrak p}] = [\g{p}_\Phi
\ominus V , \g{p}_\Phi \ominus V] \subset \g{k}_\Phi \subset
\g{m}_\Phi$, and since $\g{m}_\Phi$ is orthogonal to ${\mathfrak
a}_\Phi \oplus {\mathfrak n}_\Phi$, we see that $V \oplus
\g{n}_\Phi$ is orthogonal to $[(V \oplus
\g{n}_\Phi)^\perp_{\mathfrak p} , (V \oplus
\g{n}_\Phi)^\perp_{\mathfrak p}] \oplus (V \oplus
\g{n}_\Phi)^\perp_{\mathfrak p}$. Since $A_\Phi \times N_\Phi$ acts
freely on $M$, it is clear that $V \times N_\Phi$ induces a
foliation on $M$. From Theorem \ref{thPolar} we conclude that the
action of $V \times N_\Phi$ on $M$ is polar and that $F_\Phi^s
\times ({\mathbb E}^{r-r_\Phi} \ominus V)$ is a section of the
action. The statement about hyperpolarity follows from the fact that
$F_\Phi^s \times ({\mathbb E}^{r-r_\Phi} \ominus V)$ is flat if and
only if $\Phi = \emptyset$.
\end{proof}

The previous result provides examples of polar actions which are not
hyperpolar on each Riemannian symmetric space of noncompact type
with rank $\geq 2$.  It is worthwhile to compare this with the
results by Kollross \cite{Ko07} that in the compact case polar
actions are in general hyperpolar. Special cases of these actions on
Hermitian symmetric spaces of noncompact type have also been
discussed by Kobayashi \cite{Kob05} in the context of strongly
visible actions on complex manifolds.

\begin{remark}
Let $M$ be a symmetric space of noncompact type with the property
that its restricted root system contains two simple roots of the
same length which are not connected in the Dynkin diagram. The
following example illustrates that the condition in Theorem
\ref{thPolar} (i) that $\g{h}$ is orthogonal to
$[{\g{h}_{\g{p}}^\perp},{\g{h}_{\g{p}}^\perp}]
\oplus{\g{h}_{\g{p}}^\perp}$
is necessary for polarity. Let us consider
$\g{h}=(\g{a}\ominus\R(H_\alpha-H_\beta))
\oplus(\g{n}\ominus\R(X_\alpha+X_\beta))$, with $\alpha$ and $\beta$
two simple roots of the same length which are not connected in the
Dynkin diagram, and $X_\alpha\in\g{g}_\alpha$,
$X_\beta\in\g{g}_\beta$ unit vectors. In order to prove that this is
indeed a subalgebra, by the properties of root systems it suffices
to show that $[H,X_\alpha-X_\beta]\in\g{h}$ for any
$H\in\g{a}\ominus\R(H_\alpha-H_\beta)$ (because $\langle
X_\alpha+X_\beta,X_\alpha-X_\beta\rangle=0$). However, if
$H\in\g{a}\ominus\R(H_\alpha-H_\beta)$ we have $\alpha(H)=\beta(H)$,
which implies
$[H,X_\alpha-X_\beta]=\alpha(H)(X_\alpha-X_\beta)\in\g{h}$ as
desired.

By construction we have
${\g{h}_{\g{p}}^\perp}=\R(H_\alpha-H_\beta)\oplus
\R(1-\theta)(X_\alpha+X_\beta)$. A simple calculation using
$\langle\alpha,\alpha\rangle=\langle\beta,\beta\rangle$ and
$\langle\alpha,\beta\rangle=0$ shows that
\begin{equation*}
[H_\alpha-H_\beta,(1-\theta)(X_\alpha+X_\beta)]=
\langle\alpha,\alpha\rangle(1+\theta)(X_\alpha-X_\beta).
\end{equation*}
This implies in particular that ${\g{h}_{\g{p}}^\perp}$ is not
abelian. Using again
$\langle\alpha,\alpha\rangle=\langle\beta,\beta\rangle$ and
$\langle\alpha,\beta\rangle=0$ we get
\begin{equation*}
[H_\alpha-H_\beta,(1+\theta)(X_\alpha-X_\beta)] =
\langle\alpha,\alpha\rangle(1-\theta)(X_\alpha+X_\beta),
\end{equation*}
and also using $[X_\alpha,X_\beta]=0$ (because $\alpha$ and $\beta$
are not connected in the Dynkin diagram) we obtain
\begin{equation*}
[(1-\theta)(X_\alpha+X_\beta), (1+\theta)(X_\alpha-X_\beta)]  =
-2(H_\alpha-H_\beta).
\end{equation*}
All in all this means that ${\g{h}_{\g{p}}^\perp}$ is a non-abelian
Lie triple system. However, $\g{h}$ cannot give rise to a polar
action because $\g{h}$ is not perpendicular to
$[{\g{h}_{\g{p}}^\perp},{\g{h}_{\g{p}}^\perp}]
=\R(1+\theta)(X_\alpha-X_\beta)$.

This action has the interesting feature that it gives a homogeneous
foliation with the property that the normal bundle consists of Lie
triple systems. It is easy to see that the totally geodesic
submanifold of $M$ generated by any of these Lie triple systems is a
real hyperbolic plane. These real hyperbolic planes have the
property that they do not intersect orthogonally the other orbits.
It is interesting to observe that the normal bundle is not
integrable, as otherwise the integral manifolds would provide
sections and then the action would be polar.
\end{remark}

\begin{remark}\label{exNoFoliation}
The hypothesis in Theorem \ref{thPolar} that $H$ induces a foliation
is necessary. For example, in $\g{sl}_2(\C)$ consider the usual
Cartan decomposition $\g{sl}_2(\C)=\g{su}_2\oplus\g{p}$, where
$\g{p}$ denotes the real vector space of $(2 \times 2)$-Hermitian
matrices with trace zero. Let $\g{a}$ be the subspace of diagonal
matrices in $\g{sl}_2(\C)$ with real coefficients and $\g{t}$ the
subspace of diagonal matrices in $\g{sl}_2(\C)$ with purely
imaginary coefficients. Also, denote by $\g{n}$ the set of strictly
upper triangular matrices in $\g{sl}_2(\C)$. Then,
$\g{su}_2\oplus\g{a}\oplus\g{n}$ is an Iwasawa decomposition of
$\g{sl}_2(\C)$ and $\g{t}\oplus\g{a}$ is a Cartan subalgebra of
$\g{sl}_2(\C)$. Consider the vectors
\begin{equation*}
B=\left(\begin{array}{cc}
1   &   i\\
0   &   -1
\end{array}\right), X=\left(\begin{array}{cc}
-i   &   1\\
0   &   i
\end{array}\right), \xi=\left(\begin{array}{cc}
1   &   -2i\\
-2i   &   -1
\end{array}\right)
\mbox{ and } E=\left(\begin{array}{cc}
0   &   i/2\\
0   &   0
\end{array}\right).
\end{equation*}
Let $\g{h}$ be the Lie subalgebra of $\g{t}\oplus\g{a}\oplus\g{n}$
spanned by $B$ and $X$. Then ${\g{h}_{\g{p}}^\perp}=\R \xi$ is
abelian because it is one-dimensional. The connected closed subgroup
$H$ of $SL_2(\C)$ with Lie algebra $\g{h}$ acts hyperpolarly on the
real hyperbolic space $\R H^3=SL_2(\C)/SU_2$ but does not give rise
to a hyperpolar foliation. To see this let $g=\Exp(E)$. It is easy
to verify that $\Ad(g)B \in {\mathfrak a}$ and $\Ad(g)X \in
{\mathfrak t}$, and hence $\Ad(g)\g{h}=\g{t}\oplus\g{a}$. The
corresponding connected subgroup of $SL_2(\C)$ acts with
cohomogeneity one on $\R H^3$. This action has one singular orbit, a
totally geodesic $\R H^1\subset\R H^3$, and the other orbits are the
tubes around it. Obviously, the action of $H$ is orbit equivalent to
this one.
\end{remark}

We continue with a discussion of some further hyperpolar actions on
Riemannian symmetric spaces of nonpositive curvature.

\begin{example} ({\it Polar and hyperpolar homogeneous foliations
on Euclidean spaces.})
Let $m$ be a positive integer. For each linear subspace $V$ of
the $m$-dimensional Euclidean space ${\mathbb E}^m$ we define a
homogeneous hyperpolar foliation ${\mathcal F}_V^m$ on ${\mathbb
E}^m$ by
\begin{equation*}
({\mathcal F_V^m})_p = p + V = \{ p + v \mid v \in V\}
\end{equation*}
for all $p \in {\mathbb E}^m$. Geometrically, the foliation
${\mathcal F}_V^m$ consists of the union of all affine subspaces of
${\mathbb E}^m$ which are parallel to $V$. It is obvious that
${\mathcal F_V^m}$ is a hyperpolar homogeneous foliation on
${\mathbb E}^m$ whenever $0 < \dim V < m$.

Indeed, any hyperpolar homogeneous foliation on a Euclidean space
$\mathbb{E}^m$ is isometrically congruent to one of these examples.
Assume that $H$ acts isometrically on $\mathbb{E}^m$ and that its
orbits form a hyperpolar homogeneous foliation. Since the action of
$H$ is isometric and gives a foliation on $\mathbb{E}^m$ it suffices
to prove that each orbit of $H$ is totally geodesic.

On the contrary, assume that the orbit of $H$ through $o$ is not
totally geodesic. Then, there exist a nonzero vector $v\in
T_o(H\cdot o)$ and a unit vector $\xi\in\nu_o(H\cdot o)$ such that
$A_\xi v=c v$ with $c\neq 0$, where $A_\xi$ denotes the shape
operator of $H\cdot o$ with respect to $\xi$. Since the orbit
through $o$ is principal, $\xi$ induces an equivariant normal vector
field on $H\cdot o$ which we also denote by $\xi$. This vector field
satisfies $\xi_{h(o)}=h_*\xi_o$ for all $h\in H$. Consider the point
$p=\exp_o(\frac{1}{c}\xi_o)$. Since $\xi$ is equivariant, the orbit
of $H$ through $p$ is $H\cdot
p=\{\exp_{h(o)}(\frac{1}{c}\xi_{h(o)}):h\in H\}$. Hence we can
define the map $F:H\cdot o\to H\cdot p$,
$h(o)\mapsto\exp_{h(o)}(\frac{1}{c}\xi_{h(o)})
=h(o)+\frac{1}{c}\xi_{h(o)}$.  Since the action of $H$ is polar, the
equivariant vector field $\xi$ is parallel with respect to the
normal connection (see e.g.\ \cite{BCO}, p.~44, Corollary 3.2.5),
and thus we get $F_{*o}v=v-\frac{1}{c}A_\xi v=0$, which contradicts
the fact that $H$ gives a foliation.
\end{example}

\begin{example} ({\it Codimension one foliations on Riemannian
manifolds.})
Let $M$ be a connected complete Riemannian manifold and
${\mathcal F}$ be a homogeneous foliation on $M$ with codimension
one. Then ${\mathcal F}$ is hyperpolar. In fact, consider a geodesic
$\gamma : {\mathbb R} \to M$ which is parametrized by arc length and
for which $\dot{\gamma}(0)$ is perpendicular to ${\mathcal
F}_{\gamma(0)}$. Since $M$ is connected and complete, $\gamma$ must
intersect each leaf of ${\mathcal F}$, and since ${\mathcal F}$ is
homogeneous, the geodesic intersects each leaf orthogonally.
Therefore ${\mathcal S} = \gamma({\mathbb R})$ is a section of
${\mathcal F}$. Clearly, ${\mathcal S}$ is a flat totally geodesic
submanifold of $M$, and hence ${\mathcal F}$ is hyperpolar.
\end{example}

\begin{example} ({\it Hyperpolar homogeneous foliations on
hyperbolic spaces.})
Let $M$ be a Riemannian symmetric space of rank one, that is,
$M$ is a hyperbolic space ${\mathbb F}H^n$ over a normed real
division algebra ${\mathbb F} \in \{{\mathbb R}, {\mathbb C},
{\mathbb H}, {\mathbb O}\}$. Here $n \geq 2$, and $n = 2$ if
${\mathbb F} = {\mathbb O}$. Using the notations introduced in the
previous section, we have
\begin{equation*}
{\mathfrak g} =
\begin{cases}
{\mathfrak s}{\mathfrak o}_{1,n} & \text{if ${\mathbb F} = {\mathbb
R}$,} \\
{\mathfrak s}{\mathfrak u}_{1,n} & \text{if ${\mathbb F} = {\mathbb
C}$,} \\
{\mathfrak s}{\mathfrak p}_{1,n} & \text{if ${\mathbb F} = {\mathbb
H}$,} \\
{\mathfrak f}_4^{-20} & \text{if ${\mathbb F} = {\mathbb O}$}.
\end{cases}
\end{equation*}
The restricted root space decomposition of ${\mathfrak g}$ is of the
form
\begin{equation*}
{\mathfrak g} = {\mathfrak g}_{-2\alpha} \oplus {\mathfrak
g}_{-\alpha} \oplus {\mathfrak g}_0 \oplus {\mathfrak g}_\alpha
\oplus {\mathfrak g}_{2\alpha},\end{equation*} where $\dim
{\mathfrak g}_\alpha = \dim {\mathfrak g}_{-\alpha} =
(n-1)\dim_{\mathbb R}{\mathbb F}$, $\dim {\mathfrak g}_{2\alpha} =
\dim {\mathfrak g}_{-2\alpha} = \dim_{\mathbb R}{\mathbb F} - 1$.
Moreover, we have ${\mathfrak g}_0 = {\mathfrak k}_0 \oplus
{\mathfrak a}$ with a one-dimensional subspace ${\mathfrak a}
\subset {\mathfrak p}$ and
\begin{equation*}
{\mathfrak k}_0 \cong
\begin{cases}
{\mathfrak s}{\mathfrak o}_{n-1} & \text{if ${\mathbb F} = {\mathbb
R}$,} \\
{\mathfrak u}_{n-1} & \text{if ${\mathbb F} = {\mathbb
C}$,} \\
{\mathfrak s}{\mathfrak p}_{n-1} \oplus {\mathfrak s}{\mathfrak p}_1
& \text{if ${\mathbb F} = {\mathbb
H}$,} \\
{\mathfrak s}{\mathfrak o}_7 & \text{if ${\mathbb F} = {\mathbb
O}$}.
\end{cases}
\end{equation*}
Let $\ell$ be a one-dimensional linear subspace of ${\mathfrak
g}_\alpha$ and define
\begin{equation*}
{\mathfrak s}_\ell = {\mathfrak a} \oplus
({\mathfrak g}_\alpha \ominus \ell) \oplus {\mathfrak g}_{2\alpha} =
{\mathfrak a} \oplus ({\mathfrak n} \ominus \ell).
\end{equation*}
The subspace ${\mathfrak s}_\ell$ is a subalgebra of ${\mathfrak a}
\oplus {\mathfrak n}$ of codimension one, and the corresponding
connected closed subgroup $S_\ell$ of $AN$ acts freely on ${\mathbb
F}H^n$ with cohomogeneity one. The corresponding homogeneous
foliation ${\mathcal F}_\ell$ on ${\mathbb F}H^n$ is hyperpolar
according to the previous example. Since $K_0$ acts transitively on
the unit sphere in ${\mathfrak g}_\alpha$ by means of the adjoint
representation, ${\mathcal F}_\ell$ and ${\mathcal F}_{\ell^\prime}$
are orbit equivalent for any two one-dimensional linear subspaces
$\ell,\ell^\prime$ of ${\mathfrak g}_\alpha$. We denote by
${\mathcal F}_{\mathbb F}^n$ a representative of the set of
hyperpolar homogeneous foliations of the form ${\mathcal F}_\ell$ on
${\mathbb F}H^n$. We mention that the leaf of ${\mathcal F}_{\mathbb
F}^n$ containing $o \in {\mathbb F}H^n$ is a minimal hypersurface in
${\mathbb F}H^n$. If ${\mathbb F} = {\mathbb R}$, this leaf is a
totally geodesic real hyperbolic hyperplane ${\mathbb R}H^{n-1}
\subset {\mathbb R}H^n$. If ${\mathbb F} = {\mathbb C}$, this leaf
is the minimal ruled real hypersurface in ${\mathbb C}H^n$ which is
determined by a horocycle in a totally geodesic and totally real
${\mathbb R}H^2 \subset {\mathbb C}H^n$. For more details on these
foliations we refer to \cite{Be98} and \cite{BD06}. It was shown in
\cite{BT03} that apart from this foliation and the horosphere
foliation there are no other homogeneous hyperpolar foliations on
Riemannian symmetric spaces of rank one.
\end{example}

\begin{example} ({\it Hyperpolar homogeneous foliations on products
of hyperbolic spa\-ces.}) Let
\begin{equation*}
M = {\mathbb F}_1H^{n_1} \times \ldots \times {\mathbb
F}_kH^{n_k}
\end{equation*}
be the Riemannian product of $k$ Riemannian symmetric
spaces of rank one, where $k$ is a positive integer. Then
\begin{equation*}
{\mathcal F}_{{\mathbb F}_1}^{n_1} \times \ldots \times {\mathcal
F}_{{\mathbb F}_k}^{n_k}
\end{equation*}
is a hyperpolar homogeneous foliation on $M$. This is an elementary
consequence of the previous example.
\end{example}

\begin{example} ({\it Hyperpolar homogeneous foliations on products
of hyperbolic
spaces and Euclidean spaces.}) Let
\begin{equation*}
M = {\mathbb F}_1H^{n_1} \times \ldots \times {\mathbb
F}_kH^{n_k} \times {\mathbb E}^{m}
\end{equation*}
be the Riemannian product of
$k$ Riemannian symmetric spaces of rank one and an $m$-dimensional
Euclidean space, where $k$ and $m$ are positive integers. Moreover,
let $V$ be a linear subspace of ${\mathbb E}^m$. Then
\begin{equation*}
{\mathcal F}_{{\mathbb F}_1}^{n_1} \times \ldots \times {\mathcal
F}_{{\mathbb F}_k}^{n_k} \times {\mathcal F}_V^m
\end{equation*}
is a hyperpolar homogeneous foliation on $M$.
\end{example}

\begin{example}\label{homfolsym} ({\it Homogeneous foliations on
symmetric spaces of noncompact type.}) Let $M$ be a Riemannian
symmetric space of noncompact type and $\Phi$ be a subset of
$\Lambda$ with the property that any two roots in $\Phi$ are not
connected in the Dynkin diagram of the restricted root system
associated with $\Lambda$. We call such a subset $\Phi$ an
orthogonal subset of $\Lambda$. Each simple root $\alpha \in \Phi$
determines a totally geodesic hyperbolic space ${\mathbb F}_\alpha
H^{n_\alpha} \subset M$. In fact, ${\mathbb F}_\alpha H^{n_\alpha}
\subset M$ is the orbit of the connected subgroup of $G$ with Lie
algebra ${\mathfrak g}_{\{\alpha\}}$. Then $F_\Phi$ is isometric to
the Riemannian product of $r_\Phi$ Riemannian symmetric spaces of
rank one and an $(r-r_\Phi)$-dimensional Euclidean space,
\begin{equation*}
F_\Phi = F_\Phi^s \times {\mathbb E}^{r-r_\Phi} \cong \left(
\prod_{\alpha \in \Phi} {\mathbb F}_{\alpha} H^{n_\alpha} \right)
\times {\mathbb E}^{r-r_\Phi}.
\end{equation*}
Note that ${\mathbb F}_\alpha = {\mathbb R}$ if $\alpha$ is reduced
and ${\mathbb F}_{\alpha} \in \{{\mathbb C},{\mathbb H},{\mathbb
O}\}$ if $\alpha$ is non-reduced (i.e., if $2\alpha \in \Sigma$ as
well). Then
\begin{equation*}
{\mathcal F}_\Phi = \prod_{\alpha \in \Phi} {\mathcal F}_{{\mathbb
F}_\alpha}^{n_\alpha}.
\end{equation*}
is a hyperpolar homogeneous foliation on ${\mathcal F}_\Phi^s$. Let
$V$ be a linear subspace of ${\mathbb E}^{r-r_\Phi}$. Then
\begin{equation*}
{\mathcal F}_{\Phi,V} = {\mathcal F}_\Phi \times {\mathcal
F}_V^{r-r_\Phi} \times N_\Phi \subset F_\Phi^s \times {\mathbb
E}^{r-r_\Phi} \times N_\Phi = F_\Phi \times N_\Phi \cong M
\end{equation*}
is a homogeneous foliation on $M$. We will see below that it is
hyperpolar.

Recall that each foliation ${\mathcal F}_{{\mathbb
F}_\alpha}^{n_\alpha}$ on ${\mathbb F}_{\alpha} H^{n_\alpha}$
corresponds to a subalgebra of ${\mathfrak g}_{\{\alpha\}}$ of the
form $\g{a}^{\{\alpha\}} \oplus (\g{g}_\alpha \ominus \ell_\alpha)
\oplus \g{g}_{2\alpha}$ with some one-dimensional linear subspace
$\ell_\alpha$ of ${\mathfrak g}_\alpha$. As a consequence the
foliation ${\mathcal F}_\Phi$ on $F_\Phi^s$ corresponds to the
subalgebra ${\mathfrak a}^\Phi \oplus \left( \bigoplus_{\alpha \in
\Phi} \left( ({\mathfrak g}_\alpha \ominus \ell_\alpha)  \oplus
{\mathfrak g}_{2\alpha} \right) \right) = {\mathfrak a}^\Phi \oplus
({\mathfrak n}_\Phi \ominus \ell_\Phi)$ of ${\mathfrak g}_\Phi$,
where $\ell_\Phi=\bigoplus_{\alpha\in\Phi}\ell_\alpha$. Therefore
the foliation ${\mathcal F}_{\Phi,V}$ on $M$ corresponds to the
subalgebra
\begin{equation*}
\g{s}_{\Phi,V}=(\g{a}^\Phi \oplus V)
\oplus(\g{n}_\Phi\ominus\ell_\Phi)=(\g{a}^\Phi \oplus V
\oplus\g{n}_\Phi)\ominus\ell_\Phi \subset {\mathfrak a} \oplus
{\mathfrak n}_\Phi
\end{equation*}
of ${\mathfrak q}_\Phi$, where we identify canonically $V \subset
{\mathbb E}^{r-r_\Phi} = A_\Phi \cdot o$ with the corresponding
subspace of ${\mathfrak a}_\Phi$.

It is easy to see from the arguments given above that different
choices of $\ell_\alpha$ and $\ell_{\alpha}^\prime$ in ${\mathfrak
g}_\alpha$ lead to isometrically congruent foliations ${\mathcal
F}_\Phi$ and ${\mathcal F}_\Phi^\prime$ on $F_\Phi^s$. However, it
is not obvious that different choices of $\ell_\alpha$ and
$\ell^\prime_\alpha$ in ${\mathfrak g}_\alpha$ lead to isometrically
congruent foliations ${\mathcal F}_{\Phi,V}$ and ${\mathcal
F}_{\Phi,V}^\prime$ on $M$. That this is in fact true follows from
the following two facts. On a semisimple symmetric space the
holonomy algebra is isomorphic to the Lie algebra of the isotropy
subgroup of the isometry group. Moreover, on a simply connected
symmetric space each element of the holonomy group at a point $o$
induces an isometry of the symmetric space with fixed point $o$.
Hence, different choices of $\ell_\alpha$ and $\ell^\prime_\alpha$
in ${\mathfrak g}_\alpha$ lead to isometrically congruent foliations
${\mathcal F}_{\Phi,V}$ and ${\mathcal F}_{\Phi,V}^\prime$ on $M$.

We note that these homogeneous foliations on symmetric spaces of
noncompact type have also been discussed by Koike \cite{Koi082} in
the context of his investigations about ``complex hyperpolar
actions".
\end{example}

\medskip
We are now in the position to formulate the main result of this
paper.

\begin{theorem}\label{maintheorem}
Let $M$ be a connected Riemannian symmetric space of noncompact
type.
\begin{itemize}
\item[(i)] Let $\Phi$ be an orthogonal subset of $\Lambda$ and $V$ be
a linear subspace of ${\mathbb E}^{r-r_\Phi}$. Then
\begin{equation*}
{\mathcal F}_{\Phi,V} = {\mathcal F}_\Phi \times {\mathcal
F}_V^{r-r_\Phi} \times N_\Phi \subset F_\Phi^s \times {\mathbb
E}^{r-r_\Phi} \times N_\Phi = M
\end{equation*}
is a hyperpolar homogeneous foliation on $M$. \item[(ii)] Every
hyperpolar homogeneous foliation on $M$ is isometrically congruent
to ${\mathcal F}_{\Phi,V}$ for some orthogonal subset $\Phi$ of
$\Lambda$ and some linear subspace $V$ of ${\mathbb E}^{r-r_\Phi}$.
\end{itemize}
\end{theorem}

\begin{proof}
We prove part (i) of the theorem here. Section
\ref{secClassification}  is devoted to the proof of part (ii).

According to Theorem \ref{thPolar} we have to prove that
$\g{s}_{\Phi,V}$ is a subalgebra and that
$(\g{s}_{\Phi,V})_{\g{p}}^\perp=\{\xi\in\g{p}:
\langle\xi,Y\rangle=0\mbox{ for all }Y\in\g{s}_{\Phi,V}\}$ is
abelian. Assume that the one-dimensional linear space
$\ell_\alpha=\ell_\Phi\cap\g{g}_\alpha$ is generated by the nonzero
vector $E_\alpha$.

The fact that $\g{s}_{\Phi,V}$ is a subalgebra follows from the
elementary properties of root systems. It is easy to see that
\begin{equation*}
(\g{s}_{\Phi,V})_{\g{p}}^\perp=(\g{a}_\Phi\ominus V)\oplus\left(
\mathop{\bigoplus}_{\alpha\in\Phi}\R((1-\theta)E_{\alpha})\right).
\end{equation*}
We now check that $(\g{s}_{\Phi,V})_{\g{p}}^\perp$ is abelian.

If $H,H'\in\g{a}_\Phi\ominus V$ we obviously have $[H,H']=0$. If
$H\in\g{a}_\Phi\ominus V$ and $\alpha\in\Phi$ we have
$[H,(1-\theta)E_{\alpha}]=\alpha(H)(1+\theta)E_{\alpha}=0$ by
definition of $\g{a}_\Phi$. If $\alpha,\beta\in\Phi$ with
$\alpha\neq\beta$, then $[(1-\theta)E_{\alpha},(1-\theta)
E_{\beta}]=(1+\theta)[E_{\alpha},E_{\beta}]
-(1+\theta)[E_{\alpha},\theta E_{\beta}]$.  Now,
$[E_{\alpha},E_{\beta}]\in\g{g}_{\alpha+\beta}=0$ because
$\alpha+\beta$ is not a root (since $\alpha$ and $\beta$ are not
connected in the Dynkin diagram) and $[E_{\alpha},\theta
E_{\beta}]\in\g{g}_{\alpha-\beta}=0$ as $\alpha-\beta$ is not a root
(because both $\alpha$ and $\beta$ are simple).
\end{proof}

\section{Classification}\label{secClassification}

In this section we prove Theorem \ref{maintheorem} (ii), thus
settling the main result of this paper.

A subalgebra $\g{b}$ of a Lie algebra ${\mathfrak g}$ is called a
\emph{Borel subalgebra} if $\g{b}$ is a maximal solvable subalgebra
of $\g{g}$. Borel subalgebras of real semisimple Lie algebras have
been described in \cite{M61}. Any such Borel subalgebra can be
written as $\g{t}\oplus\g{a}\oplus\g{n}$, where
$\g{h}=\g{t}\oplus\g{a}$ is a Cartan subalgebra of $\g{g}$ and
$\g{n}$ is nilpotent. The subspace $\g{t}$ is called the toroidal
part of $\g{h}$ and consists of all $X\in\g{h}$ for which the
eigenvalues of $\ad(X)$ are purely imaginary. The subspace $\g{a}$
is called the vector part of $\g{h}$ and consists of all $X\in\g{h}$
for which the eigenvalues of $\ad(X)$ are real. There exists a
Cartan decomposition $\g{g}=\g{k}\oplus\g{p}$ such that
$\g{t}\subset\g{k}$ and $\g{a}\subset\g{p}$. We say that $\g{h}$ or
$\g{b}$ is \emph{maximally noncompact} if $\g{a}$ is maximal abelian
in $\g{p}$ and \emph{maximally compact} if $\g{t}$ is maximal
abelian in $\g{k}$. We use this description for the following

\begin{proposition}\label{thBorel}
Let $M=G/K$ be a symmetric space of noncompact type. Let $S$ be a
closed subgroup of $G$ which induces a hyperpolar foliation. Then
the action of $S$ is orbit equivalent to the action of a closed
solvable subgroup whose Lie algebra is contained in a maximally
noncompact Borel subalgebra.
\end{proposition}

\begin{proof}
By means of Proposition \ref{ReduceToSolvable} we may assume that
$S$ is solvable and closed in $I(M)$. The Lie algebra $\g{s}$ of $S$
is contained in a Borel subalgebra $\g{b}$ of $\g{g}$. As we
explained above, there exists a Cartan decomposition
$\g{g}=\g{k}\oplus\g{p}$ such that $\g{b} =
\g{t}\oplus\tilde{\g{a}}\oplus\tilde{\g{n}}$ with
$\g{t}\subset\g{k}$ and $\tilde{\g{a}}\subset\g{p}$. Since
$\tilde{\g{a}}$ is abelian we have the decomposition
$\g{g}=\tilde{\g{g}}_0\oplus\left(
\bigoplus_{\tilde{\lambda}\in\tilde{\Sigma}}
\tilde{\g{g}}_{\tilde{\lambda}}\right)$, where $\tilde{\Sigma}$ is
the set of roots with respect to $\tilde{\g{a}}$ and
$\tilde{\g{g}}_{\tilde{\lambda}}=\{X\in\g{g}:\ad(H)X
=\tilde{\lambda}(H)X \ {\rm for\ all\ }H\in\tilde{\g{a}}\}$. We can
choose an ordering in $\tilde{\g{a}}$ that induces a set of positive
roots $\tilde{\Sigma}^+$ in such a way that
$\tilde{\g{n}}=\bigoplus_{\tilde{\lambda}\in\tilde{\Sigma}^+}
\tilde{\g{g}}_{\tilde{\lambda}}$. It remains to prove that this
Borel subalgebra is maximally noncompact, that is, $\tilde{\g{a}}$
is maximal abelian in $\g{p}$.

On the contrary, assume that $\tilde{\g{a}}$ is not maximal abelian.
Let $\g{a}$ be a maximal abelian subspace of $\g{p}$ containing
$\tilde{\g{a}}$. Then we have the usual restricted root space
decomposition ${\g{g}}=\g{g}_0\oplus
\left(\bigoplus_{\lambda\in\Sigma}\g{g}_{\lambda}\right)$. We choose
an ordering of $\g{a}$ compatible with that of $\tilde{\g{a}}$ and
denote by $\Sigma^+$ the corresponding set of positive roots, and
write $\g{n}=\bigoplus_{\lambda\in\Sigma^+} \g{g}_{\lambda}$. We
have the relations
\begin{equation*}
\tilde{\g{a}}=\mathop{\bigcap_{\lambda\in\Sigma^+}}_{
\lambda_{\vert\tilde{\g{a}}}=0} \mathop{\rm Ker}\lambda,\qquad
\tilde{\g{g}}_0=\g{g}_0\oplus\biggl(\
\mathop{\bigoplus_{\lambda\in\Sigma^+}}_{
\lambda_{\vert\tilde{\g{a}}}=0} \g{g}_{\lambda}\biggr),\qquad
\tilde{\g{g}}_{\tilde{\lambda}}=
\mathop{\bigoplus_{\lambda\in\Sigma^+}}_{
\lambda_{\vert\tilde{\g{a}}}=\tilde{\lambda}} \g{g}_{\lambda}.
\end{equation*}

Recall from Theorem \ref{thPolar} (ii) that $S$ acts hyperpolarly on
$M$ if and only if $\spp=\{\xi\in\g{p}:\xi\perp\g{s}\}$ is abelian.
Obviously, $\g{a}\ominus\tilde{\g{a}}\subset\g{p}$ and
$\g{a}\ominus\tilde{\g{a}}$ is orthogonal to $\tilde{\g{n}}$, and so
$\g{a}\ominus\tilde{\g{a}}\subset\spp$. On the other hand,
$\bigoplus_{\lambda\in\Sigma^+,
\lambda_{\vert\tilde{\g{a}}}=0}\g{g}_{\lambda}\subset
\tilde{\g{g}}_0\subset\g{n}\ominus\tilde{\g{n}}$, and so
$\bigoplus_{\lambda\in\Sigma^+,
\lambda_{\vert\tilde{\g{a}}}=0}\g{p}_{\lambda} \subset\spp$.
Altogether this implies $(\g{a}\ominus\tilde{\g{a}})\oplus\left(
\bigoplus_{\lambda\in\Sigma^+,
\lambda_{\vert\tilde{\g{a}}}=0}\g{p}_{\lambda}\right) \subset\spp$.
Now choose $\lambda\in\Sigma^+$ with
$\lambda_{\vert\tilde{\g{a}}}=0$. By the first relation above, we
can choose $H\in\g{a}\ominus\tilde{\g{a}}$ with $\lambda(H)\neq 0$.
If $X_{\lambda}\in\g{g}_{\lambda}$ is a nonzero vector then
$[H,(1-\theta)X_{\lambda}]= (1+\theta)\lambda(H)X_{\lambda}\neq 0$,
which contradicts the fact that $\spp$ is abelian. Hence,
$\tilde{\g{a}}$ must be maximal abelian in $\g{p}$ and the theorem
follows.
\end{proof}

We now prove that the foliations in Example \ref{homfolsym} exhaust
all the possibilities for hyperpolar homogeneous foliations up to
orbit equivalence. Let $S$ be a connected closed subgroup of the
isometry group inducing a hyperpolar homogeneous foliation on $M$.
From now on we fix a Cartan decomposition $\g{g}=\g{k}\oplus\g{p}$
and a maximally noncompact Borel subalgebra
$\g{t}\oplus\g{a}\oplus\g{n}$ with $\g{t}\subset\g{k}$ and
$\g{a}\subset\g{p}$ maximal abelian. According to Proposition
\ref{thBorel} we may assume that the Lie algebra $\g{s}$ of $S$ is
solvable and that $\g{s}\subset\g{t}\oplus\g{a}\oplus\g{n}$. The
proof goes as follows. First we classify the abelian subspaces of
$\g{a}\oplus\g{p}^1$. A bit more work leads to a description of all
subalgebras $\g{s}$ of $\g{t}\oplus\g{a}\oplus\g{n}$ for which
$\spp$ is abelian and contained in $\g{a}\oplus\g{p}^1$. Hence, the
problem reduces to prove that, if $\g{s}$ is a subalgebra of
$\g{t}\oplus\g{a}\oplus\g{n}$ for which $\spp$ is abelian and the
corresponding connected subgroup of $G$ with Lie algebra $\g{s}$
induces a foliation on $M$, then $\spp\subset\g{a}\oplus\g{p}^1$. We
will consider an auxiliary subalgebra
$\tilde{\g{s}}=\g{s}+\left(\g{n}\ominus\g{n}^1\right)$. This
subalgebra satisfies
$\tilde{\g{s}}_{\g{p}}^\perp\subset\g{a}\oplus\g{p}^1$ and hence its
projection onto $\g{a}\oplus\g{n}$ is one of the known examples.
Then $\spp$ is contained in the centralizer of
$\tilde{\g{s}}_{\g{p}}^\perp$ in $\g{p}$. A bit more work allows us
to calculate $\spp$ explicitly using the fact that $\g{s}$ is a
subalgebra. Then we will conclude that the projection of $\g{s}$
onto $\g{a}\oplus\g{n}$ is one of the known examples. The final step
is to prove that $\g{s}$ induces the same orbits as its projection
onto $\g{a}\oplus\g{n}$.

In what follows (until Lemma~\ref{thBracketsn} inclusive) we will
work in a context slightly more general than that described above.
Let $\g{s}$ be a subalgebra of $\g{t}\oplus\g{a}\oplus\g{n}$ such
that $\spp$ is abelian. Hence, it is not assumed that the orbits of
the connected subgroup of $G$ whose Lie algebra is $\g{s}$ form a
foliation. Example \ref{exNoFoliation} shows that this can happen.
We first state a few basic lemmas.

From now on, if $\g{v}$ is a vector subspace of $\g{g}$, we denote
by $\pi_{\g{v}}$ the orthogonal projection of $\g{g}$ onto $\g{v}$.
Also, we denote by $\g{s}_n=\pi_{\g{a}\oplus\g{n}}(\g{s})$ the
projection of $\g{s}$ onto $\g{a}\oplus\g{n}$, the noncompact part
of $\g{t}\oplus\g{a}\oplus\g{n}$.

We will first derive some elementary results.

\begin{lemma}\label{thPolarization}
Let $\lambda\in\Sigma$ and $X,Y\in\g{g}_\lambda$. Then
$(1-\theta)[\theta X,Y]=2\langle X,Y\rangle H_\lambda$.
\end{lemma}

\begin{proof}
It follows from polarization of the identity
$[\theta(X+Y),X+Y]=\langle X+Y,X+Y\rangle H_\lambda$.
\end{proof}

\begin{lemma}\label{th1dim}
Let $\alpha$ be a simple root and $\g{v}\subset\g{g}_\alpha$ be a
linear subspace such that $[\g{v},\g{v}]=\{0\}$. Then
$[\g{v},\theta\g{v}]\subset\g{a}$ if and only if $\g{v}$ is
one-dimensional.
\end{lemma}

\begin{proof}
If $\g{v}=\R X$ with nonzero $X\in\g{g}_\alpha$, then $[\theta X,
X]=\langle X,X\rangle H_\alpha\in\g{a}$. For the converse, assume
that $\g{v}$ has dimension greater than $1$ and that
$[\g{v},\theta\g{v}]\subset\g{a}$. Let $X,Y\in\g{v}$ be two nonzero
orthogonal vectors. By Lemma \ref{thPolarization} and orthogonality
of $X$ and $Y$, $(1-\theta)[\theta X,Y]=2\langle X,Y\rangle
H_\alpha=0$, so $[\theta X,Y]\in\g{k}_0\cap\g{a}=\{0\}$. Now,
$\langle[\theta X,Y],[\theta X,Y]\rangle =-\langle[X,[\theta
X,Y]],Y\rangle$ and using the Jacobi identity and the fact that
$[\g{v},\g{v}]=\{0\}$ we get $[X,[\theta X,Y]]=-[Y,[X,\theta
X]]=\langle
X,X\rangle[Y,H_\alpha]=-\langle\alpha,\alpha\rangle\langle
X,X\rangle Y$. Altogether this implies $\langle[\theta X,Y],[\theta
X,Y]\rangle=\langle\alpha,\alpha\rangle\langle X,X\rangle\langle
Y,Y\rangle>0$, which gives a contradiction.
\end{proof}

\begin{lemma}\label{thNonzeroBracket}
Let $\lambda,\mu\in\Sigma$ such that $\lambda-\mu\not\in\Sigma$. Let
$X\in\g{g}_\lambda$ and $Y\in\g{g}_\mu$ be nonzero vectors. If
$[X,Y]=0$ then $\lambda+\mu$ is not a root. In particular, if
$\alpha,\beta\in\Lambda$ and $X\in\g{g}_\alpha$ and
$Y\in\g{g}_\beta$ are nonzero vectors, then $[X,Y]=0$ implies that
$\alpha$ and $\beta$ are not connected in the Dynkin diagram.
\end{lemma}

\begin{proof}
Assume that $[X,Y]=0$. Since $[\theta Y,X]\in\g{g}_{\lambda-\mu}=0$
we have, using the Jacobi identity, that
\begin{equation*}
0=[[X,Y],\theta Y]=-[[Y,\theta Y],X]=\langle Y,Y\rangle[H_\mu,X]
=\langle Y,Y\rangle\langle\lambda,\mu\rangle X.
\end{equation*}
Hence, $\langle\lambda,\mu\rangle=0$. Since $\lambda
-\mu\not\in\Sigma$ and the corresponding Cartan integer satisfies
$A_{\mu\lambda}=0$, we get that $\lambda\pm\mu\not\in\Sigma$. For
the second part, just note that $\alpha-\beta$ is not a root.
\end{proof}

\begin{lemma}\label{thTspp}
Let $\Psi\subset\Sigma^+$. For each $\lambda\in\Psi$ let
$v_\lambda\subset\g{g}_\lambda$ be a one-dimensional linear
subspace. Then the linear subspace $[\g{k}_0,\g{a}\oplus
\left(\bigoplus_{\lambda\in\Psi}v_\lambda\right)]$ is orthogonal to
$\g{a}\oplus
\left(\bigoplus_{\lambda\in\Psi}(1-\theta)v_\lambda\right)
\oplus\left(\bigoplus_{\lambda\in\Sigma^+\setminus\Psi}
\g{p}_\lambda\right)$.
\end{lemma}

\begin{proof}
Obviously, $[\g{k}_0,\g{a}]=0$, so there is nothing to prove in this
case. Assume that each $v_\lambda$ is spanned by a corresponding
vector $E_\lambda$. Let $T\in\g{k}_0$. If $H\in\g{a}$ then
$\langle[T,E_\lambda],H\rangle=-\langle E_\lambda,[T,H]\rangle=0$.
Since $[\g{k}_0,\g{g}_\lambda]\subset\g{g}_\lambda$, for any
$\mu\in\Sigma^+$ with $\mu\neq\lambda$ and any $\xi\in\g{p}_\mu$ we
obviously have $\langle[T,E_\lambda],\xi\rangle=0$. Finally, since
$\ad(T)$ is skewsymmetric,
$\langle[T,E_\lambda],(1-\theta)E_\lambda\rangle=
\langle[T,E_\lambda],E_\lambda\rangle=0$, from where the result
follows.
\end{proof}

We now proceed with the first step of the proof, which is describing
abelian subspaces of $\g{a}\oplus\g{p}^1$. Recall that ${\mathfrak
p}^1 = {\mathfrak p} \cap \left( {\mathfrak g}_\emptyset^1 \oplus
{\mathfrak g}_\emptyset^{-1} \right)$. First, we need the following
lemma.

\begin{lemma}\label{thq1dim}
Let $\g{q}\subset\g{a}\oplus\g{p}^1$ be an abelian subspace and
define $\Psi=\{\alpha\in\Lambda:\pi_{\g{g}_\alpha}(\g{q})\neq 0\}$.
Then $\dim\pi_{\g{g}_\alpha}(\g{q})=1$ for all $\alpha\in\Psi$.
\end{lemma}

\begin{proof}
Assume the statement is not true. Any two vectors $\xi,\eta\in\g{q}$
can be written as
$\xi=\xi_0+\sum_{\alpha\in\Psi}(1-\theta)\xi_\alpha$ and
$\eta=\eta_0+\sum_{\alpha\in\Psi}(1-\theta)\eta_\alpha$, with
$\xi_\alpha,\eta_\alpha\in\g{g}_\alpha$. We denote by
$\Psi'\subset\Psi\subset\Lambda$ the subset of roots
$\alpha\in\Lambda$ such that $\xi_\alpha$ and $\eta_\alpha$ are
linearly independent. If the statement of this lemma is not true, we
can find $\xi$ and $\eta$ such that the corresponding $\Psi'$ is
nonempty. An easy calculation taking into account that
$\alpha-\beta$ is not a root if $\alpha,\beta\in\Lambda$ yields
\begin{equation*}
0=[\xi,\eta]= \sum_\alpha(1+\theta)(\alpha(\xi_0)\eta_\alpha
-\alpha(\eta_0)\xi_\alpha)
-\sum_\alpha(1+\theta)[\xi_\alpha,\theta\eta_\alpha]
+\sum_{\alpha,\beta}(1+\theta)[\xi_\alpha,\eta_\beta].
\end{equation*}
Then it follows in particular that
$\sum_\alpha(1+\theta)[\xi_\alpha,\theta\eta_\alpha]=0$,
$[\xi_\alpha,\eta_\alpha]=0$ for all $\alpha\in\Psi$ and
$[\xi_\alpha,\eta_\beta]+[\xi_\beta,\eta_\alpha]=0$ for all
$\alpha,\beta\in\Psi$ with $\alpha\neq\beta$.

By Lemma \ref{thPolarization},
$(1-\theta)[\theta\xi_\alpha,\eta_\alpha]
=2\langle\xi_\alpha,\eta_\alpha\rangle H_\alpha$, which implies
%\begin{equation*}
%0=\sum_{\alpha\in\Psi} (1+\theta)[\xi_\alpha,\theta\eta_\alpha]
%=2\sum_{\alpha\in\Psi}([\xi_\alpha,\theta\eta_\alpha]
%+\langle\xi_\alpha,\eta_\alpha\rangle H_\alpha)
%=2\sum_{\alpha\in\Psi'}([\xi_\alpha,\theta\eta_\alpha]
%+\langle\xi_\alpha,\eta_\alpha\rangle H_\alpha),
%\end{equation*}
\begin{equation*}
0=\frac{1}{2}\sum_{\alpha\in\Psi}
(1+\theta)[\xi_\alpha,\theta\eta_\alpha]
=\sum_{\alpha\in\Psi}([\xi_\alpha,\theta\eta_\alpha]
+\langle\xi_\alpha,\eta_\alpha\rangle H_\alpha)
=\sum_{\alpha\in\Psi'}([\xi_\alpha,\theta\eta_\alpha]
+\langle\xi_\alpha,\eta_\alpha\rangle H_\alpha)
\end{equation*}
the last equality following from the fact that
$[\theta\xi_\alpha,\xi_\alpha]=\langle\xi_\alpha,\xi_\alpha\rangle
H_\alpha$.

For $\alpha\in\Psi'$, using $[\xi_\alpha,\eta_\alpha]=0$ and
$[\theta\xi_\alpha,\xi_\alpha]=\langle\xi_\alpha,\xi_\alpha\rangle
H_\alpha$, we get
\begin{equation*}
[[\xi_\alpha,\theta\eta_\alpha],\theta\xi_\alpha]=
-[[\theta\eta_\alpha,\theta\xi_\alpha],\xi_\alpha]
-[[\theta\xi_\alpha,\xi_\alpha],\theta\eta_\alpha]
=\langle\alpha,\alpha\rangle\langle\xi_\alpha,\xi_\alpha\rangle
\theta\eta_\alpha.
\end{equation*}

Now choose $\alpha,\beta\in\Psi'$ with $\beta\neq\alpha$. Since
$\beta-\alpha$ is not a root and
$[\theta\eta_\beta,\theta\xi_\alpha]=\theta[\eta_\beta,\xi_\alpha]
=\theta[\xi_\beta,\eta_\alpha]=[\theta\xi_\beta,\theta\eta_\alpha]$
we obtain
\begin{align*}
[[\xi_\beta,\theta\eta_\beta],\theta\xi_\alpha]
&={}-[[\theta\eta_\beta,\theta\xi_\alpha],\xi_\beta]
-[[\theta\xi_\alpha,\xi_\beta],\theta\eta_\beta]]\
={} -[[\theta\xi_\beta,\theta\eta_\alpha],\xi_\beta]\\
&=[[\theta\eta_\alpha,\xi_\beta],\theta\xi_\beta]
+[[\xi_\beta,\theta\xi_\beta],\theta\eta_\alpha]\
={} \langle\alpha,\beta\rangle\langle\xi_\beta,
\xi_\beta\rangle\theta\eta_\alpha.
\end{align*}
Taking into account the last two displayed equations we conclude
\begin{equation*}
[[\xi_\beta,\theta\eta_\beta],\theta\xi_\alpha]
=\langle\alpha,\beta\rangle\langle\xi_\beta,
\xi_\beta\rangle\theta\eta_\alpha,
\mbox{ for all }\alpha,\beta\in\Psi'.
\end{equation*}

Therefore, for arbitrary $\alpha\in\Psi'$, the identity
$\sum_{\beta\in\Psi'}([\xi_\beta,\theta\eta_\beta]
+\langle\xi_\beta,\eta_\beta\rangle H_\beta)=0$ yields
\begin{align*}
0 &= \left[\sum_{\beta\in\Psi'}([\xi_\beta,\theta\eta_\beta]
    +\langle\xi_\beta,\eta_\beta\rangle H_\beta),
    \theta\xi_\alpha\right]
= \sum_{\beta\in\Psi'}
    \left([[\xi_\beta,\theta\eta_\beta],\theta\xi_\alpha]
    +\langle\xi_\beta,\eta_\beta\rangle [H_\beta,
    \theta\xi_\alpha]\right)\\
&= \sum_{\beta\in\Psi'}
    \left(\langle\alpha,\beta\rangle\langle\xi_\beta,\xi_\beta\rangle
    \theta\eta_\alpha-\langle\alpha,\beta\rangle
    \langle\xi_\beta,\xi_\beta\rangle\theta\xi_\alpha\right)\\
&= \left(\sum_{\beta\in\Psi'}
    \langle\alpha,\beta\rangle\langle\xi_\beta,\xi_\beta\rangle
    \right)\theta\eta_\alpha
    -\left(\sum_{\beta\in\Psi'}
    \langle\alpha,\beta\rangle
    \langle\xi_\beta,\xi_\beta\rangle\right)\theta\xi_\alpha.
\end{align*}
Since $\alpha\in\Psi'$, $\theta\eta_\alpha$ and
$\theta\xi_\alpha$ are linearly independent, the only way the above
equality can hold is when the coefficients of $\theta\eta_\alpha$
and $\theta\xi_\alpha$ are simultaneously zero. In particular,
$\sum_{\beta\in\Psi'}
\langle\alpha,\beta\rangle\langle\xi_\beta,\xi_\beta\rangle=0$. Hence,
$\sum_{\beta\in\Psi'}\langle\xi_\beta,\xi_\beta\rangle\beta$ is
orthogonal to $\mathop{\rm span}\Psi'$. Since it is also a
vector in $\mathop{\rm span}\Psi'$ and the simple roots are
linearly independent, it follows that
$\langle\xi_\beta,\xi_\beta\rangle=0$ for all $\beta\in\Psi'$,
contradiction.
\end{proof}

We say that a subset $\Phi\subset\Lambda$ is \emph{connected} if the
subdiagram of the Dynkin diagram determined by the roots of $\Phi$
is connected. We say that two subsets $\Phi,\Phi^\prime
\subset\Lambda$ are \emph{disconnected} or \emph{orthogonal} if for
any $\alpha\in\Phi$ and any $\beta\in\Phi^\prime$, $\alpha$ and
$\beta$ are not connected in the Dynkin diagram (that is,
$\alpha+\beta$ is not a root).

\begin{proposition}\label{thAbelian}
Let $\g{q}\subset\g{a}\oplus\g{p}^1$ be an abelian subspace and
$\Psi=\{\alpha\in\Lambda:\pi_{\g{g}_\alpha}(\g{q})\neq 0\}$. This
set can be decomposed as $\Psi=\bigcup_{i=1}^k\Psi_i$ with
$\Psi_i\subset\Lambda $ connected, and $\Psi_i$ and $\Psi_j$
disconnected whenever $i\neq j$. Then there exists a map
$c:\Psi\to\R$, $\alpha\mapsto c_\alpha$, and vectors
$E_i\in\bigoplus_{\alpha\in\Psi_i}\g{g}_\alpha$,
$i\in\{1,\dots,k\}$, with $\pi_{\g{g}_\alpha}(E_i)\neq 0$ for all
$\alpha\in\Psi_i$, such that $\g{q}$ is a linear subspace of
\begin{equation*}
\g{v}_{\g{q}}=\g{a}_\Psi
\oplus\left(\bigoplus_{i=1}^k\R\left(\sum_{\alpha\in\Psi_i}
{c_\alpha}H_\alpha +(1-\theta)E_i \right)\right).
\end{equation*}
\end{proposition}

\begin{proof}
Using the fact that the sets $\Psi_i$ are disconnected, it is easy
to see that a subalgebra $\g{v}_\g{q}$ as considered above is
abelian. Hence, in order to prove the proposition, it suffices to
take an abelian subalgebra $\g{q}\subset\g{a}\oplus\g{p}^1$ and
prove that it can be realized as a subspace of one of the Lie
subalgebras $\g{v}_\g{q}$ as defined above. For that, consider
$\Psi=\{\alpha\in\Lambda:\pi_{\g{g}_\alpha}(\g{q})\neq 0\}$ and
write $\Psi=\bigcup_{i=1}^k\Psi_i$ with $\Psi_i\subset\Lambda$
connected, and $\Psi_i$ and $\Psi_j$ disconnected whenever $i\neq
j$. Our first assertion is that for each $i\in\{1,\dots,k\}$ there
exists a nonzero vector
$E_i\in\bigoplus_{\alpha\in\Psi_i}\g{g}_\alpha$ such that any vector
$\xi\in\g{q}$ can be written as $\xi=\xi_0+\sum_{i=1}^k x_i
(1-\theta)E_i$ for certain $\xi_0\in\g{a}$ and $x_i\in\R$.

We fix $i\in\{1,\dots,k\}$. From Lemma \ref{thq1dim} it follows that
$\dim\pi_{\g{g}_\alpha}(\g{q})=1$ for all $\alpha\in\Psi$. Hence,
for each $\alpha\in\Psi$ we can choose a nonzero vector
$E_\alpha\in\g{g}_\alpha$ such that any vector $\xi\in\g{q}$ can be
written as $\xi=\xi_0+\sum_{\alpha\in\Psi}a_\alpha(1-\theta)
E_\alpha$ for certain $\xi_0\in\g{a}$ and $a_\alpha\in\R$. If, on
the contrary, the previous assertion is not true, we can find
$\xi=\xi_0+\sum_{\alpha}a_\alpha(1-\theta) E_\alpha\in\g{q}$ and
$\eta=\eta_0+\sum_{\alpha}b_\alpha(1-\theta) E_\alpha\in\g{q}$ such
that for some $\alpha,\beta\in\Psi_i$ connected in the Dynkin
diagram the vectors $(a_\alpha,a_\beta),(b_\alpha,b_\beta)\in\R^2$
are linearly independent. Since $\g{q}$ is abelian, we have
\begin{equation*}
0=[\xi,\eta]=\sum_\alpha(\alpha(\xi_0)b_\alpha
-\alpha(\eta_0)a_\alpha)(1+\theta)E_\alpha
+\sum_{\alpha,\beta}a_\alpha b_\beta(1+\theta)[E_\alpha,E_\beta].
\end{equation*}
In particular, taking the $\g{g}_{\alpha+\beta}$ component we get
$a_\alpha b_\beta-a_\beta b_\alpha=0$ by Lemma
\ref{thNonzeroBracket}, which contradicts the fact that
$(a_\alpha,a_\beta)$ and $(b_\alpha,b_\beta)$ are linearly
independent.

Therefore, we have proved our assertion, that is, any vector
$\xi\in\g{q}$ can be written as $\xi=\xi_0+\sum_{i=1}^k x_i
(1-\theta)E_i$ for certain $\xi_0\in\g{a}$, $x_i\in\R$ and
$E_i\in\bigoplus_{\alpha\in\Psi_i}\g{g}_\alpha$. By the definition
of $\Psi$, it is obvious that $\pi_{\g{g}_\alpha}(E_i)\neq 0$ for
all $\alpha\in\Psi_i$, and indeed we can write
$E_i=\sum_{\alpha\in\Psi_i}E_\alpha$ with suitable
$E_\alpha\in\g{g}_\alpha$ (note that these might be different from
the above $E_\alpha$'s). Since $\Psi_i$ and $\Psi_j$ are
disconnected if $i\neq j$, it is clear that $[E_i,E_j]=0$ for all
$i,j\in\{1,\dots,k\}$.

We choose such a vector $\xi=\xi_0+\sum_{j=1}^k x_j (1-\theta)E_j$
and assume $x_i=0$. By definition of $\Psi$ there certainly exists
$\eta=\eta_0+\sum_{j=1}^k y_j (1-\theta)E_j\in\g{q}$ with $y_i\neq
0$. Hence,
\begin{equation*}
0=[\xi,\eta]=\sum_{j=1}^k(1+\theta) \left(\sum_{\alpha\in\Psi_j}
\left(y_j\alpha(\xi_0)-x_j\alpha(\eta_0)\right)E_\alpha\right),
\end{equation*}
and taking the $\g{g}_\alpha$-component for any $\alpha\in\Psi_i$ we
get $\alpha(\xi_0)=0$ because $x_i=0$ and $y_i\neq 0$. This implies
that for each $\xi\in\g{q}$ and each $\alpha\in\Psi_i$ we can write
$\langle\xi,H_\alpha\rangle
=c_\alpha(\xi)\langle\alpha,\alpha\rangle x_i$ (if $x_i=0$ any
$c_\alpha(\xi)$ will do). The next step is to prove that we can
choose the same $c_\alpha$ for all $\xi\in\g{q}$.

Assume that we cannot choose the $c_\alpha$'s in such a way. Then
there would be $\xi=\xi_0+\sum_{j=1}^k x_j (1-\theta)E_j\in\g{q}$
and $\eta=\eta_0+\sum_{j=1}^k y_j (1-\theta)E_j\in\g{q}$ such that
$c_\alpha(\xi)\neq c_\alpha(\eta)$ for some $\alpha\in\Psi_i$. This
of course implies that $x_i,y_i\neq 0$. Taking the corresponding
$\g{g}_\alpha$-component of $[\xi,\eta]$ as in the previous
displayed formula we get
\begin{equation*}
0=y_i\alpha(\xi_0)-x_i\alpha(\eta_0)=
y_i\langle\xi,H_\alpha\rangle-x_i\langle\eta,H_\alpha\rangle =x_i
y_i\langle\alpha,\alpha\rangle(c_\alpha(\xi)- c_\alpha(\eta)),
\end{equation*}
which leads to a contradiction. Hence we can choose the $c_\alpha$
independently of $\xi\in\g{q}$, which allows us to define a function
$c:\Psi\to\R$ by $\langle\xi,H_\alpha\rangle=c_\alpha
\langle\alpha,\alpha\rangle x_i$, with the notation as above.

Finally, if $\xi=\xi_0+\sum_{i=1}^k x_i (1-\theta)E_i\in\g{q}$ we
can write $\xi_0=\xi_0'+\sum_{\alpha\in\Psi}a_\alpha H_\alpha$ with
$\xi_0'\in\g{a}_\Psi=\g{a}\ominus\left( \bigoplus_{\alpha\in\Psi}\R
H_\alpha\right)$ and $a_\alpha\in\R$. Here, the $a_\alpha$ must
satisfy $c_\alpha \langle\alpha,\alpha\rangle
x_i=\langle\xi,H_\alpha\rangle
=a_\alpha\langle\alpha,\alpha\rangle$, so $\g{q}$ is contained in
one of the model spaces in the statement of the proposition.
\end{proof}

The following lemma is useful to understand how $\g{s}$ and $\spp$
are related.

\begin{lemma}\label{thsspp}
If $\g{s}\subset\g{t}\oplus\g{a}\oplus\g{n}$ is a subalgebra and
$\spp=\{\xi\in\g{p}:\xi\perp\g{s}\}$, then
$\g{s}_n=\pi_{\g{a}\oplus\g{n}}(\g{s}) =
\{X\in\g{a}\oplus\g{n}:X\perp\spp\}$.
\end{lemma}

\begin{proof}
If $X\in\g{s}$, then for all $\xi\in\spp$ we have $\langle
X,\xi\rangle=0$ by definition. Since $\g{k}$ and $\g{p}$ are
orthogonal, $\pi_{\g{a}\oplus\g{n}}(X)\perp\spp$.

Conversely, let $X\in\g{a}\oplus\g{n}$ such that $X\perp\spp$, and
choose $Y\in(\g{a}\oplus\g{n})\ominus\g{s}_n$. We may write
$Y=H+\sum_{\lambda\in\Sigma^+}Y_\lambda$ with $H\in\g{a}$ and
$Y_\lambda\in\g{g}_\lambda$. Clearly,
$Y-\sum_{\lambda\in\Sigma^+}\theta
Y_\lambda=H+\sum_{\lambda\in\Sigma^+}(1-\theta)Y_\lambda\in\g{p}$,
and if $Z\in\g{s}$ we have $\langle
Y-\sum_{\lambda\in\Sigma^+}\theta Y_\lambda,Z\rangle=\langle
Y,Z\rangle-\sum_{\lambda\in\Sigma^+}\langle\theta
Y_\lambda,Z\rangle=0$, because $Y$ and $Z$ are perpendicular and so
are $\g{g}_{-\lambda}$ and $\g{t}\oplus\g{a}\oplus\g{n}$. This
proves that $Y-\sum_{\lambda\in\Sigma^+}\theta Y_\lambda\in\spp$. By
assumption we have $X\perp\spp$, and so $0=\langle
X,Y-\sum_{\lambda\in\Sigma^+}\theta Y_\lambda\rangle=\langle
X,Y\rangle-\sum_{\lambda\in\Sigma^+}\langle X,\theta
Y_\lambda\rangle=\langle X,Y\rangle$, again because
$\g{g}_{-\lambda}$ and $\g{a}\oplus\g{n}$ are perpendicular. Since
$Y\in(\g{a}\oplus\g{n})\ominus\g{s}$ is arbitrary, we conclude that
$X\in(\g{a}\oplus\g{n})\ominus((\g{a}\oplus\g{n})\ominus\g{s}_n)
=\g{s}_n$.
\end{proof}

Proposition \ref{thAbelian} and Lemma \ref{thsspp} allow us to
conclude the first step of the proof of our classification.

\begin{theorem}\label{thsppLevel1}
Let $\g{s}\subset\g{t}\oplus\g{a}\oplus\g{n}$ be a subalgebra such
that $\spp\subset\g{a}\oplus\g{p}^1$ is abelian. Then there are an
orthogonal subset $\Phi\subset\Lambda$, numbers $a_\alpha\in\R$ and
nonzero vectors $E_\alpha\in\g{g}_\alpha$ for each $\alpha\in\Phi$,
and a linear subspace $V\subset\g{a}_\Phi$ such that
\begin{equation*}
\g{s}_n=(V\oplus\g{a}^\Phi\oplus\g{n})\ominus
\left(\bigoplus_{\alpha\in\Phi}\R (a_\alpha
H_\alpha+E_\alpha)\right).
\end{equation*}
\end{theorem}

\begin{proof}
Since $\spp\subset\g{a}\oplus\g{p}^1$ is abelian, by Proposition
\ref{thAbelian} we have that $\spp$ is a linear subspace of
\begin{equation*}
\g{q}=\g{a}_\Phi
\oplus\left(\bigoplus_{i=1}^k\R\left(\sum_{\alpha\in\Phi_i}
{a_\alpha}H_\alpha +(1-\theta)E_i \right)\right).
\end{equation*}
Here, as usual,
$\Phi=\{\alpha\in\Lambda:\pi_{\g{g}_\alpha}(\spp)\neq
0\}=\bigcup_{i=1}^k\Phi_i$, with $\Phi_i$ connected, and $\Phi_i$
disconnected to $\Phi_j$ whenever $i\neq j$,
$E_i\in\bigoplus_{\alpha\in\Phi_i}\g{g}_\alpha$ with
$\pi_{\g{g}_\alpha}(E_i)\neq 0$ for all $\alpha\in\Phi_i$, and
$a:\Phi\to\R$ a real-valued function. Our first step is to prove
that each $\Phi_i$ consists of exactly one root.

Fix $i\in\{1,\dots,k\}$ and assume that $\Phi_i$ has more than one
root. We may write $E_i=\sum_{\alpha\in\Phi_i}E_\alpha$ with
$E_\alpha\in\g{g}_\alpha$. Also, take $\xi=\xi_0+\sum_j
y_j(\sum_{\alpha\in\Phi_j} {a_\alpha}H_\alpha
+(1-\theta)E_j)\in\spp$ with $y_i\neq 0$. Since
$\dim\left(\bigoplus_{\alpha\in\Phi_i}\R E_\alpha\right)>1$, there
exists a nonzero vector $X\in\bigoplus_{\alpha\in\Phi_i}\R E_\alpha$
such that $X$ is orthogonal to $E_i$ (or equivalently, to
$(1-\theta)E_i$, or to $\xi$, or to $\spp$). By Lemma \ref{thsspp},
there exists $S\in\g{t}$ such that $S+X\in\g{s}$. We write
$X=\sum_{\alpha\in\Phi_i}x_\alpha E_\alpha$. Now, for each
$\alpha\in\Phi_i$, and again for dimension reasons we can find a
vector
$Z_\alpha=H_\alpha+\sum_{\beta\in\Phi_i}z_{\alpha\beta}E_\beta$ such
that $Z_\alpha$ is orthogonal to $\xi$ (and hence to $\spp$). Lemma
\ref{thsspp} ensures that there exists $T_\alpha\in\g{t}$ such that
$T_\alpha+Z_\alpha\in\g{s}$ for each $\alpha\in\Phi_i$. By Lemma
\ref{thTspp} we get
$\langle[T_\alpha,X],\xi\rangle=\langle[S,Z_\alpha],\xi\rangle=0$.
Since $\g{s}$ is a subalgebra and $[E_\beta,E_\gamma]\in\g{n}^2$, we
have
\begin{equation*}
0=\langle[T_\alpha+Z_\alpha,S+X],\xi\rangle
%=\langle[H_\alpha,X],\xi\rangle
=\langle\sum_{\beta\in\Phi_i}x_\beta\langle\alpha,\beta\rangle
E_\beta,\xi\rangle=y_i\sum_{\beta\in\Phi_i}
\langle\alpha,\beta\rangle x_\beta\langle E_\beta,E_\beta\rangle
\end{equation*}
for each $\alpha\in\Phi_i$. Putting $A_{\beta\alpha}=
\frac{2\langle\alpha,\beta\rangle}{\langle\alpha,\alpha\rangle}$ and
taking into account that $y_i\neq 0$, the previous equation is
equivalent to $\sum_{\beta\in\Phi_i}A_{\beta\alpha} x_\beta\langle
E_\beta,E_\beta\rangle=0$. Of course, $(A_{\beta\alpha})$ is the
Cartan matrix of the Dynkin subdiagram associated with $\Phi_i$.
Cartan matrices are known to be nonsingular (see for example
\cite[Proposition 2.52 (e)]{K96}) which means that $x_\beta\langle
E_\beta,E_\beta\rangle=0$ for all $\beta\in\Phi_i$. This contradicts
the fact that $X$ is a nonzero vector and proves that $\Phi_i$ has
exactly one root.

Thus from now one we can assume that $\spp$ is a linear subspace of
\begin{equation*}
\g{q}=\g{a}_\Phi \oplus\left(\bigoplus_{\alpha\in\Phi} \R\left(
{a_\alpha}H_\alpha +(1-\theta)E_\alpha \right)\right),
\end{equation*}
with $E_\alpha\in\g{g}_\alpha$ and $\Phi$ an orthogonal subset of
$\Lambda$. Thus, we have to show that
$\bigoplus_{\alpha\in\Phi}\R({a_\alpha}H_\alpha
+(1-\theta)E_\alpha)\subset\spp$. Consider a vector $\xi=\xi_0+
\sum_{\alpha\in\Phi}x_\alpha({a_\alpha}H_\alpha
+(1-\theta)E_\alpha)\in\g{q}$ orthogonal to $\spp$, with
$\xi_0\in\g{a}_\Phi$ and $x_\alpha\in\R$. Since $\xi$ is orthogonal
to $\spp$, by Lemma \ref{thsspp} we can find $S\in\g{t}$ such that
the vector $X=S+\xi_0+
\sum_{\alpha\in\Phi}x_\alpha({a_\alpha}H_\alpha +E_\alpha)$ is in
$\g{s}$. On the other hand, it is clear, using again Lemma
\ref{thsspp}, that there exists $T_\alpha\in\g{t}$ such that
$Z_\alpha=T_\alpha+\langle E_\alpha,E_\alpha\rangle
H_\alpha-c_\alpha\langle\alpha,\alpha\rangle E_\alpha$ is a vector in
$\g{s}$. Since $\g{s}$ is a subalgebra, $[Z_\alpha,X]\in\g{s}$ for
each $\alpha\in\Phi$. A calculation using the facts that
$\g{t}\oplus\g{a}$ is abelian, $\xi_0\in\g{a}_\Phi$ and
$\alpha+\beta\not\in\Sigma$ if $\alpha,\beta\in\Phi$ and
$\alpha\neq\beta$, gives
\begin{equation*}
[Z_\alpha,X]=x_\alpha\langle\alpha,\alpha\rangle (\langle
E_\alpha,E_\alpha\rangle+a_\alpha^2\langle\alpha,\alpha\rangle)
E_\alpha+[T_\alpha,\sum_{\beta\in\Phi}x_\beta
E_\beta]+[S,a_\alpha\langle\alpha,\alpha\rangle E_\alpha].
\end{equation*}
Lemma \ref{thTspp} implies that the last two addends above are
orthogonal to $\spp$. Since $[Z_\alpha,X]$ is orthogonal to $\spp$,
the first addend must be orthogonal to $\spp$ as well. By definition
of $\spp$, the only way this can happen is when $x_\alpha=0$ for all
$\alpha\in\Phi$. This implies $\xi=\xi_0\in\g{a}_\Phi$ and proves
$\bigoplus_{\alpha\in\Phi}\R({a_\alpha}H_\alpha
+(1-\theta)E_\alpha)\subset\spp$. Now the theorem follows after a
straightforward application of Lemma \ref{thsspp}.
\end{proof}

Motivated by Theorem \ref{thsppLevel1} we introduce the following
notation. Let $a : \Phi \to {\mathbb R}$ be a map and define
$a_\alpha = a(\alpha)$ for all $\alpha \in \Phi$. Furthermore, for
each $\alpha \in \Phi$ we choose a nonzero vector $E_\alpha \in
\ell_\Phi \cap {\mathfrak g}_\alpha$. Consider
\begin{equation*}
\g{s}_{\Phi,V,a}=(V\oplus\g{a}^\Phi \oplus\g{n})
\ominus\left(\bigoplus_{\alpha\in\Phi}\R
(a_{\alpha}H_{\alpha}+E_{\alpha})\right).
\end{equation*}
As above we will see in Proposition \ref{thCongruency} that this
does not depend on the particular choice of $\ell_\Phi$. We
obviously have $\g{s}_{\Phi,V,0}=\g{s}_{\Phi,V}$ for the zero map $0
: \Phi \to {\mathbb R}$.

\begin{proposition}\label{thCongruency}
We have $\Ad(g)\g{s}_{\Phi,V,a}={\g{s}}_{\Phi,V}$ with
$g=\Exp(-\sum_{\alpha\in\Phi}a_\alpha E_\alpha) \in N$ if $\Phi \neq
\emptyset$ and $g = {\rm id}_G$ if $\Phi = \emptyset$. In
particular, $\g{s}_{\Phi,V,a}$ is a subalgebra of
$\g{a}\oplus\g{n}$. Moreover, the corresponding connected subgroup
$S_{\Phi,V,a}$ is conjugate to $S_{\Phi,V}$ and induces a hyperpolar
homogeneous foliation. We also have
\begin{equation*}
(\g{s}_{\Phi,V,a})_{\g{p}}^\perp=(\g{a}_\Phi\ominus V)\oplus
\left(\bigoplus_{\alpha\in\Phi}\R (a_{\alpha}H_{\alpha}
+(1-\theta)E_{\alpha})\right).
\end{equation*}
\end{proposition}

\begin{proof}
We define $\xi_{\alpha}=a_{\alpha}H_{\alpha}+E_{\alpha}$ for
$\alpha\in\Phi$. Then the subalgebra $\g{s}_{\Phi,V,a}$ can
equivalently be written as
$\g{s}_{\Phi,V,a}=(V\oplus\g{a}^\Phi\oplus\g{n})
\ominus\left(\bigoplus_{\alpha\in\Phi}\R\xi_{\alpha}\right)$. Let
$g_\alpha=\Exp(-a_{\alpha}E_{\alpha})$ and
$g=\prod_{\alpha\in\Phi}g_\alpha$. Since $\alpha$ and $\beta$ are
not connected in the Dynkin diagram, we have $[E_\alpha,E_\beta]=0$,
and so $g=\Exp(-\sum_{\alpha\in\Phi}a_\alpha E_\alpha)$. Our aim is
to prove that $\Ad(g)\g{s}_{\Phi,V,a}={\g{s}}_{\Phi,V}$.

We introduce the following notation:
\begin{equation*}
\g{s}_\alpha=(V\oplus\g{a}^\Phi\oplus\g{n})\ominus\R
E_{\alpha},\quad
\hat{\g{s}}_\alpha=(V\oplus\g{a}^\Phi\oplus\g{n})
\ominus\R\xi_{\alpha}.
\end{equation*}
First we prove that $\Ad(g_\alpha)\hat{\g{s}}_\alpha=\g{s}_\alpha$
for each $\alpha\in\Phi$. Note that, since $-a_{\alpha}
E_{\alpha}\in\g{a}\oplus\g{n}$, it follows that
$\Ad(g_\alpha)(\g{a}\oplus\g{n})=\g{a}\oplus\g{n}$. Now let
$X\in\hat{\g{s}}_\alpha$. Since $E_{\alpha}$ is a unit vector and
$X\in\g{a}\oplus\g{n}$, we have
\begin{align*}
\langle\Ad(g_\alpha)X,E_{\alpha}\rangle &=\langle
X,\Ad(\Exp(a_{\alpha}\theta E_{\alpha}))E_{\alpha}\rangle  = \langle
X,e^{a_{\alpha}\ad(\theta E_{\alpha})}
        E_{\alpha}\rangle\\
&=\langle X,E_{\alpha}+a_{\alpha} H_{\alpha}
    +\frac{a_{\alpha}^2}{2}
    \lvert\alpha\rvert^2\theta E_{\alpha}\rangle
=\langle X,\xi_{\alpha}\rangle =0.
\end{align*}
Also, if $H\in\g{a}_\Phi\ominus V$ then $\alpha(H)=0$ for each
$\alpha\in\Phi$, so
\begin{align*}
\langle\Ad(g_\alpha)X,H\rangle &=\langle X,\Ad(\Exp(a_{\alpha}\theta
E_{\alpha}))H\rangle
=\langle X,e^{a_{\alpha}\ad(\theta E_{\alpha})}H\rangle\\
&=\langle X,H+a_{\alpha}\alpha(H)\theta E_{\alpha}\rangle= \langle
X,H\rangle = 0.
\end{align*}
Altogether this proves that $\Ad(g_\alpha)\hat{\g{s}}_\alpha=
\g{s}_\alpha$.

Now let $\alpha,\beta\in\Phi$ with $\alpha\neq\beta$. We prove that
$\Ad(g_\alpha)\hat{\g{s}}_\beta=\hat{\g{s}}_\beta$. Since $\alpha$
and $\beta$ are simple roots, $\beta-m\alpha$ is not a root for
$m\geq 1$. Hence,
\begin{equation*}
\Ad(\Exp(t\theta E_{\alpha}))E_{\beta} =e^{t\ad(\theta
E_{\alpha})}E_{\beta} =\sum_{m=0}^\infty\frac{t^m}{m!}\ad(\theta
E_{\alpha})^m E_{\beta} =E_{\beta}.
\end{equation*}
If $H\in\g{a}$ we also get
\begin{equation*}
\Ad(\Exp(t\theta E_{\alpha}))H =e^{t\ad(\theta E_{\alpha})}H
=H+t\alpha(H)\theta E_{\alpha}.
\end{equation*}
Now, let $X\in\hat{\g{s}}_\beta$. Using the previous equations we
obtain
\begin{equation*}
\langle\Ad(g_\alpha)X,\xi_{\beta}\rangle =\langle X,
\Ad(\Exp(a_{\alpha}\theta E_{\alpha}))\xi_{\beta}\rangle =\langle
X,\xi_{\beta}\rangle +a_{\alpha}\langle {\alpha},{\beta}\rangle
\langle X,\theta E_{\alpha}\rangle =0.
\end{equation*}
Also, if $H\in\g{a}_\Phi\ominus V$ we get
\begin{equation*}
\langle\Ad(g_\alpha)X,H\rangle =\langle X,\Ad(\Exp(a_{\alpha}\theta
E_{\alpha}))H\rangle =\langle X,H\rangle
+a_{\alpha}{\alpha}(H)\langle X,\theta E_{\alpha}\rangle=0.
\end{equation*}
Altogether this proves $\Ad(g_\alpha)\hat{\g{s}}_\beta
=\hat{\g{s}}_\beta$. A similar argument shows that
$\Ad(g_\alpha){\g{s}}_\beta={\g{s}}_\beta$. However, since
$\g{s}_{\Phi,V,a}=\bigcap_{\alpha\in\Phi}\hat{\g{s}}_\alpha$ and
${\g{s}}_{\Phi,V}=\bigcap_{\alpha\in\Phi}\g{s}_\alpha$, using the
previous two equalities and a simple induction argument, we obtain
$\Ad(g)\g{s}_{\Phi,V,a}=(\prod_{\alpha\in\Phi}\Ad(g_\alpha))
\g{s}_{\Phi,V,a}={\g{s}}_{\Phi,V}$.
\end{proof}

This is a good point to recall the contents of Theorem
\ref{thsppLevel1}, which says that if
$\g{s}\subset\g{t}\oplus\g{a}\oplus\g{n}$ is a subalgebra such that
$\spp\subset\g{a}\oplus\g{p}^1$ is abelian, then there are an
orthogonal subset $\Phi\subset\Lambda$, numbers $a_\alpha\in\R$ and
nonzero vectors $E_\alpha\in\g{g}_\alpha$ for each $\alpha\in\Phi$,
and a linear subspace $V\subset\g{a}_\Phi$ such that
$\g{s}_n=\g{s}_{\Phi,V,a}$.

\begin{proposition}\label{thProperties}
Let $\g{s}\subset\g{t}\oplus\g{a}\oplus\g{n}$ be a subalgebra such
that
\begin{equation*}
\g{s}_n=  \g{s}_{\Phi,V,a} = (V\oplus\g{a}^\Phi\oplus\g{n})\ominus
\left(\bigoplus_{\alpha\in\Phi}\R(a_\alpha H_\alpha+E_\alpha)
\right)
\end{equation*}
with $\Phi$ a subset of orthogonal simple roots,
$E_\alpha\in\g{g}_\alpha$ nonzero vectors, $V$ a linear subspace of
$\g{a}_\Phi$ and $a_\alpha\in\R$, and define
$E=-\sum_{\alpha\in\Phi}a_\alpha E_\alpha$ and $g=\Exp(E)$. Then the
following statements hold:
\begin{itemize}
\item[(i)] $\Ad(g)\g{s}$ is a subalgebra of
$\g{t}\oplus\g{a}\oplus\g{n}$ and
$(\Ad(g)\g{s})_n\subset\Ad(g)\g{s}_n=\g{s}_{\Phi,V}$.

\item[(ii)] For each $\alpha\in\Phi$ the projection of
$(\g{t}\oplus\g{a}\oplus(\g{n}\ominus\g{g}_\alpha))\cap\g{s}$ onto
$\g{t}$ centralizes $E_\alpha$.

\item[(iii)] $V\subset(\Ad(g)\g{s})_n$.

\item[(iv)] Assume that $\lambda\in\Sigma^+\setminus\Phi$ satisfies
$\lambda+\alpha\not\in\Sigma$ for each $\alpha\in\Phi$. Then
$\g{g}_\lambda\subset(\Ad(g)\g{s})_n$.
\end{itemize}
In addition, assume that the orbits of the connected subgroup $S$ of
$G$ whose Lie algebra is $\g{s}$ form a homogeneous foliation. Then
the following further statements hold:
\begin{itemize}
\item[(v)] $(\Ad(g)\g{s})_n=\Ad(g)\g{s}_n=\g{s}_{\Phi,V}$.

\item[(vi)] Denote by $\g{s}_c$ the projection of $\g{s}$ onto
$\g{t}$. Then $\g{s}_c$ is an abelian subalgebra that centralizes
each $E_\alpha$. In particular, $[\g{s}_c,\spp]=0$.

\item[(vii)] With the notation as in \emph{(vi)}, let $S_c$ be the
connected subgroup of $G$ whose Lie algebra is $\g{s}_c$. Then $S_c$
acts trivially on $\nu_o(S\cdot o)$.
\end{itemize}
\end{proposition}

\begin{remark}
Remark \ref{exNoFoliation} shows that the hypothesis that the orbits
of $S$ form a homogeneous foliation is necessary in Proposition
\ref{thProperties} (v), (vi) and (vii).
\end{remark}

\begin{proof} (i) First note that since $\g{n}$ is an ideal of
$\g{t}\oplus\g{a}\oplus\g{n}$ we have
$\Ad(g)(\g{t}\oplus\g{a}\oplus\g{n})\subset
\g{t}\oplus\g{a}\oplus\g{n}$ and
$\Ad(g)\g{s}\subset\g{t}\oplus\g{a}\oplus\g{n}$. Proposition
\ref{thCongruency} implies that $\Ad(g){\g{s}}_n=\g{s}_{\Phi,V}$. We
have to prove that $(\Ad(g){\g{s}})_n\subset\g{s}_{\Phi,V}$. Let
$T+H+X\in{\g{s}}$ with $T\in\g{t}$, $H\in\g{a}$ and $X\in\g{n}$. We
already have $\Ad(g)(H+X)\in\Ad(g)\g{s}_n=\g{s}_{\Phi,V}$, so it
suffices to prove that the projection of $\Ad(g)T$ onto
$\g{a}\oplus\g{n}$ is in $\g{s}_{\Phi,V}$. Since $\g{n}$ is an ideal
of $\g{t}\oplus\g{a}\oplus\g{n}$ this projection is
$\Ad(g)T-T=\sum_{k=1}^\infty\frac{1}{k!}\ad(E)^k T\in\g{n}$. We have
to prove that $\langle\Ad(g)T-T,E_\alpha\rangle=0$ for all
$\alpha\in\Phi$. Since $[E,T]\in\g{n}$ and $\g{n}$ is nilpotent it
follows that $\ad(E)^k T\in\g{n}\ominus\g{n}^1$ for all $k\geq 2$,
and for $k=1$, we have $\langle[E,T],E_\alpha\rangle=0$ for all
$\alpha\in\Phi$ by Lemma \ref{thTspp}. Hence,
$(\Ad(g)\g{s})_n\subset\g{s}_{\Phi,V}=\Ad(g)\g{s}_n$.

(ii) Let $\alpha\in\Phi$ and $T$ be in the image of the projection
of $(\g{t}\oplus\g{a}\oplus(\g{n}\ominus\g{g}_\alpha))\cap\g{s}$
onto $\g{t}$. Note that $[T,E_\alpha]\in\g{g}_\alpha\ominus\R
E_\alpha$ since $\ad(T)$ preserves each root space and is
skewsymmetric. Hence, we only have to show that
$\langle\g{g}_\alpha\ominus\R E_\alpha,[T,E_\alpha]\rangle=0$. Let
$X\in\g{g}_\alpha\ominus\R E_\alpha$ be arbitrary. Since
$X\in\g{s}_n$, there exists $S_X\in\g{t}$ such that $S_X+X\in\g{s}$.
By definition of $T$, there exist $H\in\g{a}$ and
$Y\in\g{n}\ominus\g{g}_\alpha$ such that $T+H+Y\in\g{s}$. As $\g{s}$
is a subalgebra we have
$[S_X,Y]+[X,T]+[X,H]+[X,Y]=[S_X+X,T+H+Y]\in\g{s}$. Since
$[S_X,Y]\in\g{n}\ominus\g{g}_\alpha$,
$[X,H]=-\alpha(H)X\in\g{g}_\alpha\ominus\R E_\alpha$ and
$[X,Y]\in\g{n}\ominus\g{n}^1$, the definition of $\g{s}_n$ yields
\begin{equation*}
0=\langle [S_X+X,T+H+Y],a_\alpha H_\alpha+(1-\theta)E_\alpha\rangle
=\langle[X,T],E_\alpha\rangle=-\langle X,[T,E_\alpha]\rangle,
\end{equation*}
which completes the proof of (ii).

(iii) Let $H\in V\subset\g{a}_\Phi$.
Since $H\in\g{s}_n$ there is $T_H\in\g{t}$ such that
$T_H+H\in\g{s}$. By (ii) we have $[T_H,E]=0$ and by definition
$\alpha(H)=0$ for all $\alpha\in\Phi$, so
$\Ad(g^{-1})(T_H+H)=e^{-\ad(E)}(T_H+H)=T_H+H\in\g{s}$. Hence
$T_H+H\in\Ad(g)\g{s}$ and $H\in(\Ad(g)\g{s})_n$.

(iv) Assume $\lambda\in\Sigma^+\setminus\Phi$ and
$\lambda+\alpha\not\in\Sigma$ for any $\alpha\in\Phi$. Take
$X\in\g{g}_\lambda$ and $T_X\in\g{t}$ such that $T_X+X\in\g{s}$. By
(ii) we have $[T_X,E]=0$. Since $\lambda+\alpha\not\in\Sigma$ for
any $\alpha\in\Phi$, we also have $[X,E]=0$. Hence,
$\Ad(g^{-1})(T_X+X)=e^{-\ad(E)}(T_X+X)=T_X+X\in\g{s}$. This implies
$T_X+X\in\Ad(g)\g{s}$ and $X\in(\Ad(g)\g{s})_n$, so (iv) follows.

(v) By (i) we already know
$(\Ad(g)\g{s})_n\subset\Ad(g)\g{s}_n=\g{s}_{\Phi,V}$. We prove the
equality by showing that $\dim (\Ad(g)\g{s})_n=\dim \Ad(g)\g{s}_n$.

By hypothesis and Proposition \ref{thAllOrbitsPrincipal}, all the
orbits of $S$ are principal and the same is true of $I_g(S)$. Hence
the isotropy groups $S_o=S\cap K$ and $I_g(S)_o=I_g(S)\cap K$ are
conjugate. Their Lie algebras are $\g{s}\cap\g{t}$ and
$(\Ad(g)\g{s})\cap\g{t}$, respectively. By (ii) we have
$[\g{s}\cap\g{t},E]=0$ so $\Ad(g)=e^{\ad(E)}$ acts as the identity
on $\g{s}\cap\g{t}$. Hence
$\g{s}\cap\g{t}\subset(\Ad(g)\g{s})\cap\g{t}$ and thus equality
follows by hypothesis. This implies $\dim (\Ad(g)\g{s})_n=\dim
\Ad(g)\g{s}-\dim(\Ad(g)\g{s})\cap\g{t}=\dim \g{s}-\dim
\g{s}\cap\g{t}= \dim \g{s}_n=\dim \Ad(g)\g{s}_n$.

(vi) Obviously, $\g{s}_c$ is an abelian subalgebra because $\g{t}$
is abelian. For the second part, we assume first that
$\g{s}_n=\g{s}_{\Phi,V}$. Fix $\alpha\in\Phi$ and let
$X\in\g{g}_\alpha\ominus\R E_\alpha$ be arbitrary. Then there exists
$S_X\in\g{t}$ such that $S_X+X\in\g{s}$. Since
$H_\alpha\in\g{a}^\Phi\subset\g{s}_n$, there exists
$T_{H_\alpha}\in\g{t}$ such that $T_{H_\alpha}+H_\alpha\in\g{s}$. As
$\g{s}$ is a subalgebra we have
$[T_{H_\alpha}+H_\alpha,S_X+X]=(\ad(T_{H_\alpha})
+\langle\alpha,\alpha\rangle
1_{\g{g}_\alpha})X\in\g{g}_\alpha\cap\g{s}\subset\g{g}_\alpha\ominus\R
E_\alpha$, where $1_{\g{g}_\alpha}$ is the identity transformation
of $\g{g}_\alpha$. Since $\langle\alpha,\alpha\rangle\neq 0$,
$\ad(T_{H_\alpha}) +\langle\alpha,\alpha\rangle 1_{\g{g}_\alpha}$ is
an isomorphism. Thus, the previous equality implies
$\g{g}_\alpha\ominus\R E_\alpha\in\g{s}$. This implies
$S_X\in\g{s}\cap\g{t}$ and thus $[S_X,E_\alpha]=0$ by (ii). Also
according to (ii), and since $\alpha\in\Phi$ is arbitrary, (vi)
follows when $\g{s}_n=\g{s}_{\Phi,V}$.

Now we finish the proof of (vi). Let
$\g{s}\subset\g{t}\oplus\g{a}\oplus\g{n}$ be a subalgebra such that
$\g{s}_n=\g{s}_{\Phi,V,a}$ and assume that all the orbits of the
corresponding connected subgroup $S$ of $G$ whose Lie algebra is
$\g{s}$ are principal. By (v) we get
$(\Ad(g)\g{s})_n=\g{s}_{\Phi,V}$. Take an element $T+H+X\in\g{s}$
with $T\in\g{t}$, $H\in\g{a}$ and $X\in\g{n}$. Since
$H+X\in\g{s}_n$, it follows that $\Ad(g)(H+X)\in\Ad(g)\g{s}_n=
\g{s}_{\Phi,V}$ by Proposition \ref{thCongruency}. Hence, the
projection of $\Ad(g)(T+H+X)$ onto $\g{t}$ is the same as the
projection of $\Ad(g)T$ onto $\g{t}$, and as $g\in N$, that
projection is $T$. Now, since $(\Ad(g)\g{s})_n=\g{s}_{\Phi,V}$,
applying the argument in the previous paragraph to the subalgebra
$\Ad(g)\g{s}$ we get
$[T,E_\alpha]=[\pi_{\g{t}}(\Ad(g)(T+H+X)),E_\alpha]=0$ for all
$\alpha\in\Phi$. This already implies $[E,T]=0$ and thus
$\Ad(g)T=T$, so $\Ad(g)(T+H+X)=T+\Ad(g)(H+X)$ and
$\Ad(g)(H+X)\in\g{a}\oplus\g{n}$. Since $[\g{t},\g{a}]=0$, we obtain
$[T,(\g{a}_\Phi\ominus
V)\oplus\left(\bigoplus_{\alpha\in\Phi}\R(a_\alpha
H_\alpha+(1-\theta)E_\alpha)\right)]=0$ and the result follows.

(vii) Let $t\in S_c$ and $\xi\in\nu_o(S\cdot o)$. Since
$\spp\subset\g{p}$ we may identify $\spp$ and $\nu_o(S\cdot o)$. By
(vi), $\g{s}_c$ centralizes $\spp$, so with the above identification
we get $t_*\xi=\Ad(t)\xi=\xi$.
\end{proof}

We will need the following result:

\begin{lemma}\label{thBracketsn}
Let $\g{s}$ be a subalgebra of $\g{t}\oplus\g{a}\oplus\g{n}$ and
$\g{s}_n$ its projection onto $\g{a}\oplus\g{n}$. Let $\lambda$ and
$\mu$ be two positive roots (not necessarily different). If
$\g{g}_\lambda+\g{g}_\mu\subset\g{s}_n$, then
$\g{g}_{\lambda+\mu}\subset\g{s}_n$.
\end{lemma}

\begin{proof}
We may assume that $\lambda+\mu$ is a root; otherwise there is
nothing to prove. Let $X\in\g{g}_\lambda$ and $Y\in\g{g}_\mu$. By
definition there exist $S,T\in\g{t}$ such that $S+X,T+Y\in\g{s}$.
Then, $[S+X,T+Y]=[S,Y]-[T,X]+[X,Y]\in\g{s}$. Recall that
$[\g{k}_0,\g{g}_\nu]\subset\g{g}_\nu$ for any $\nu\in\Sigma$. The
vector $[S,Y]-[T,X]+[X,Y]$ is in $\g{n}$ and hence in $\g{s}_n$. On
the other hand,
$[S,Y]-[T,X]\in\g{g}_\mu+\g{g}_\lambda\subset\g{s}_n$, so
$[X,Y]\in\g{s}_n$. Since, $X$ and $Y$ are arbitrary,
$\g{g}_{\lambda+\mu}=[\g{g}_\lambda,\g{g}_\mu]\subset\g{s}_n$.
\end{proof}

We now drop the assumption $\spp\subset\g{a}\oplus\g{p}^1$.

Let ${\g{s}}$ be a subalgebra of $\g{t}\oplus\g{a}\oplus\g{n}$ such
that ${\g{s}}_{\g{p}}^\perp$ is abelian. From now on we assume that
the orbits of the connected closed subgroup ${S}$ of $G$ whose Lie
algebra is ${\g{s}}$ form a homogeneous foliation on $M$. As usual,
we denote by ${\g{s}}_n=\pi_{\g{a}\oplus\g{n}}(\g{s})$ the
projection of $\g{s}$ onto the noncompact part of
$\g{t}\oplus\g{a}\oplus\g{n}$.

We define $\tilde{\g{s}}={\g{s}}+(\g{n}^2\oplus\dots\oplus\g{n}^m) =
{\g{s}}+({\mathfrak n} \ominus {\mathfrak n}^1)$ where
$m=m_\emptyset$ is the level of the highest root of $\Sigma$. Since
${\mathfrak n} \ominus {\mathfrak n}^1$ is an ideal of
$\g{t}\oplus\g{a}\oplus\g{n}$ it follows that $\tilde{\g{s}}$ is a
subalgebra of $\g{t}\oplus\g{a}\oplus\g{n}$. Also,
${\g{s}}\subset\tilde{\g{s}}$ and thus
$\tilde{\g{s}}_{\g{p}}^\perp\subset{\g{s}}_{\g{p}}^\perp$, which
means that $\tilde{\g{s}}_{\g{p}}^\perp$ is also an abelian subspace
of $\g{p}$.

It is obvious by definition that
$\tilde{\g{s}}_{\g{p}}^\perp\subset\g{a}\oplus\g{p}^1$. Hence
Theorem \ref{thsppLevel1} implies that
$\tilde{\g{s}}_n:=\pi_{\g{a}\oplus\g{n}}(\tilde{\g{s}})
=\g{s}_{\Phi,V,a}$ with $\Phi\subset\Lambda$ a subset of orthogonal
simple roots, $a_\alpha\in\R$, $0\neq E_\alpha\in\g{g}_\alpha$ and
$V\subset\g{a}_\Phi$ as usual. By Proposition \ref{thCongruency},
there exists $g\in N$ such that
$\Ad(g)\tilde{\g{s}}_n=\g{s}_{\Phi,V}$. This element can be taken to
be $g=\Exp(E)$ with $E=-\sum_{\alpha\in\Phi}a_\alpha E_\alpha$.

We define $\hat{\g{s}}=\Ad(g)\g{s}$. The subgroup of $G$ whose Lie
algebra is $\hat{\g{s}}$ is $\hat{S}=I_g({S})$. Obviously, $\hat{S}$
induces a hyperpolar homogeneous foliation on $M$. By Proposition
\ref{thProperties} (i) we get
$\hat{\g{s}}_n:=\pi_{\g{a}\oplus\g{n}}(\hat{\g{s}})
\subset(\Ad(g)\tilde{\g{s}})_n\subset\Ad(g)\tilde{\g{s}}_n
=\g{s}_{\Phi,V}$. Then it follows that $(\g{a}_\Phi\ominus
V)\oplus\left(\bigoplus_{\alpha\in\Phi}\R
(1-\theta)E_\alpha\right)\subset\hat{\g{s}}_{\g{p}}^\perp$. Since
$\hat{\g{s}}_{\g{p}}^\perp$ is abelian, we have that
$\hat{\g{s}}_{\g{p}}^\perp$ must be contained in the centralizer
$Z_{\g{p}}((\g{a}_\Phi\ominus
V)\oplus\left(\bigoplus_{\alpha\in\Phi}\R
(1-\theta)E_\alpha)\right)$ of $(\g{a}_\Phi\ominus
V)\oplus\left(\bigoplus_{\alpha\in\Phi}\R (1-\theta)E_\alpha\right)$
in $\g{p}$. Our first aim is essentially to calculate this
centralizer. Eventually, this will allow us to determine
$\hat{\g{s}}_n$ and later ${\g{s}}_n$.

We start with $\bigoplus_{\alpha\in\Phi}\R (1-\theta)E_\alpha$ where
the situation is a bit more involved. We deal with this in a series
of lemmas.

\begin{lemma}\label{thCentre}
Let $\alpha\in\Phi$ and let $\xi\in\g{p}$ be written as
$\xi=\xi_0+\sum_{\lambda\in\Sigma^+} (1-\theta)\xi_\lambda$ with
$\xi_0\in\g{a}$ and $\xi_\lambda\in\g{g}_\lambda$ for each
$\lambda\in\Sigma^+$. Then $\xi$ is in the centralizer
$Z_{\g{p}}(\R(1-\theta)E_\alpha)$ of $\R(1-\theta)E_\alpha$ in
$\g{p}$ if and only if $\xi_0\in\g{a}_{\{\alpha\}}$,
$\xi_\alpha\in\R E_\alpha$, $\xi_{2\alpha}=0$ and
$[\xi_{\lambda-\alpha},E_\alpha]=[\xi_{\lambda+\alpha},\theta
E_\alpha]$ for all $\lambda\in\Sigma^+\setminus\{\alpha,2\alpha\}$.
\end{lemma}

\begin{proof}
If the vector $\xi$ commutes with $(1-\theta)E_\alpha$ a simple
calculation yields
\begin{equation*}
0=[\xi,(1-\theta)E_\alpha] =(1+\theta)\left(\alpha(\xi_0)E_\alpha
+\sum_{\lambda\in\Sigma^+}[\xi_\lambda,E_\alpha]
-\sum_{\lambda\in\Sigma^+\setminus\Lambda}
[\theta\xi_\lambda,E_\alpha] -[\theta\xi_\alpha,E_\alpha]\right).
\end{equation*} The above vector is zero if and only if each of
its components in $\g{k}_\lambda$, $\lambda\in\Sigma^+\cup\{0\}$, is
zero.

The $\g{k}_0$-component is zero if and only if
$[\theta\xi_\alpha,E_\alpha]\in\g{a}$, and the
$\g{k}_{2\alpha}$-component is zero if and only if
$[\xi_\alpha,E_\alpha]=0$. Denote by $\g{v}$ the vector subspace of
$\g{g}_\alpha$ spanned by $E_\alpha$ and $\xi_\alpha$. The above two
conditions imply $[\g{v},\theta\g{v}]\subset\g{a}$ and
$[\g{v},\g{v}]=0$. Since $\alpha$ is a simple root, Lemma
\ref{th1dim} implies that $\g{v}$ is 1-dimensional and hence
$\xi_\alpha\in\R E_\alpha$.

The $\g{k}_\alpha$-component vanishes if and only if
$\alpha(\xi_0)E_\alpha-[\xi_{2\alpha},\theta E_\alpha]=0$. Taking
inner product with $E_\alpha$ yields
\begin{equation*}
0=\langle\alpha(\xi_0)E_\alpha-[\xi_{2\alpha},\theta E_\alpha],
E_\alpha\rangle =\alpha(\xi_0)\langle E_\alpha,E_\alpha\rangle
-\langle\xi_{2\alpha},[E_\alpha,E_\alpha]\rangle
=\alpha(\xi_0)\langle E_\alpha,E_\alpha\rangle.
\end{equation*}
Hence, $\xi_0\in\g{a}_{\{\alpha\}}=\g{a}\ominus\R H_\alpha$. Taking
into account the above equation, this also implies
$[\xi_{2\alpha},\theta E_\alpha]=0$. Using the Jacobi identity we
get
\begin{equation*}
0=[[\xi_{2\alpha},\theta E_\alpha],E_\alpha] =-[[\theta
E_\alpha,E_\alpha],\xi_{2\alpha}] =-\langle
E_\alpha,E_\alpha\rangle[H_\alpha,\xi_{2\alpha}]
=-2\lvert\alpha\rvert^2\langle E_\alpha,
E_\alpha\rangle\xi_{2\alpha}
\end{equation*}
which implies $\xi_{2\alpha}=0$.

Finally, if $\lambda\in\Sigma^+\setminus\{\alpha,2\alpha\}$, the
$\g{k}_\lambda$-component is
$(1+\theta)([\xi_{\lambda-\alpha},E_\alpha]
-[\xi_{\lambda+\alpha},\theta E_\alpha])$. This vanishes if and only
if $[\xi_{\lambda-\alpha},E_\alpha] -[\xi_{\lambda+\alpha},\theta
E_\alpha]=0$ because $\g{g}_\lambda$ and $\g{g}_{-\lambda}$ are
linearly independent. Since the ``only if '' part is elementary, the
result follows.
\end{proof}

\begin{lemma}\label{thString}
Let $\alpha\in\Phi$ and
$\lambda\in\Sigma^+\setminus\{\alpha,2\alpha\}$, and assume that the
$\alpha$-string of $\lambda$ has length greater than one. Then
$\bigoplus_{m\in\mathbb{Z}}
\g{g}_{\lambda+m\alpha}\subset\hat{\g{s}}_n$.
\end{lemma}

\begin{proof}
Since $(1-\theta)E_\alpha\in\hat{\g{s}}_{\g{p}}^\perp$ and
$\hat{\g{s}}_{\g{p}}^\perp$ is abelian, we have
${\hat{\g{s}}_{\g{p}}^\perp}\subset Z_{\g{p}}
(\R(1-\theta)E_\alpha)$. Let $\xi\in{\hat{\g{s}}_{\g{p}}^\perp}$ and
write as usual
$\xi=\xi_0+\sum_{\lambda\in\Sigma^+}(1-\theta)\xi_\lambda$ with
$\xi_0\in\g{a}$ and $\xi_\lambda\in\g{g}_\lambda$ for each
$\lambda\in\Sigma^+$. Lemma \ref{thCentre} already implies that
$\xi_0\in\g{a}_{\{\alpha\}}$, $\xi_\alpha\in\R E_\alpha$ and
$\xi_{2\alpha}=0$. We have to prove that $\xi_{\lambda+m\alpha}=0$
for all $m\in\mathbb{Z}$. We prove this assertion depending on the
whether the length of the $\alpha$-string of $\lambda$ is 2, 3 or 4.
Note that $\lambda\not\in\Phi$.

Assume that the length of the $\alpha$-string of $\lambda$ is 2. In
this case we may assume
$\lambda-\alpha,\lambda+2\alpha\not\in\Sigma^+$ and
$\lambda,\lambda+\alpha\in\Sigma^+$ (switch to $\lambda-\alpha$ if
necessary). Then $\alpha$ and $\lambda$ span a root system of type
$A_2$. Since $\lambda+2\alpha\not\in\Sigma^+$, Lemma \ref{thCentre}
implies $[\xi_\lambda,E_\alpha]=[\xi_{\lambda+2\alpha},\theta
E_\alpha]=0$. Since $\lambda-\alpha\not\in\Sigma$ and
$\lambda+\alpha\in\Sigma$ we get from Lemma \ref{thNonzeroBracket}
that $\xi_\lambda=0$. Similarly, by Lemma \ref{thCentre} we have
$[\xi_{\lambda+\alpha},\theta
E_\alpha]=[\xi_{\lambda-\alpha},E_\alpha]=0$. Since
$\lambda+\alpha-(-\alpha)=\lambda+2\alpha\not\in\Sigma$ and
$\lambda+\alpha+(-\alpha)=\lambda\in\Sigma$, Lemma
\ref{thNonzeroBracket} yields $\xi_{\lambda+\alpha}=0$ which
finishes the proof in this case.

Assume that the $\alpha$-string of $\lambda$ has length 3. In this
case we may assume $\lambda-\alpha,\lambda+3\alpha\not\in\Sigma^+$
and $\lambda,\lambda+\alpha,\lambda+2\alpha\in\Sigma^+$. Then,
$\lambda$ and $\alpha$ span a root system of type $B_2$ or $BC_2$.
First we claim that $\xi_{\lambda+\alpha}\in[\g{g}_{\alpha}\ominus\R
E_\alpha,\g{g}_\lambda]$. Since the root system spanned by $\lambda$
and $\alpha$ is of type $B_2$ or $BC_2$, $\lambda+\alpha$ and
$\alpha$ are orthogonal and have the same length. This implies that
there exists an element of the Weyl group that maps $\alpha$ to
$\lambda+\alpha$. Hence $\g{g}_{\lambda+\alpha}$ and $\g{g}_\alpha$
have the same dimension. By Lemma \ref{thNonzeroBracket}, for any
nonzero $Z_\lambda\in\g{g}_\lambda$,
$\ad(Z_\lambda):\g{g}_\alpha\to\g{g}_{\lambda+\alpha}$ is injective,
hence bijective. We write $\xi_{\lambda+\alpha}=[Z_\lambda,c
E_\alpha+Z_\alpha]$ with $c\in\R$ and
$Z_\alpha\in\g{g}_\alpha\ominus\R E_\alpha$. Since $\langle
Z_\alpha,E_\alpha\rangle=0$, Lemma \ref{thPolarization} yields
$(1-\theta)[\theta E_\alpha,Z_\alpha]=0$ and thus $[\theta
E_\alpha,Z_\alpha]\in\g{k}_0$. Hence $\ad([\theta
E_\alpha,Z_\alpha])$ is skewsymmetric. Lemma \ref{thCentre} implies
$[\xi_{\lambda+\alpha},\theta
E_\alpha]=[\xi_{\lambda-\alpha},E_\alpha]=0$. Then, using the Jacobi
identity, we get
\begin{align*}
0&=\langle[\xi_{\lambda+\alpha},\theta E_\alpha],Z_\lambda\rangle =
\langle[[Z_\lambda,c E_\alpha+Z_\alpha],\theta E_\alpha],
Z_\lambda\rangle\\
&=-c\langle[[E_\alpha,\theta E_\alpha],Z_\lambda],Z_\lambda\rangle
-\langle[[Z_\alpha,\theta E_\alpha],Z_\lambda],Z_\lambda\rangle\\
&=c\langle\lambda,\alpha\rangle\langle E_\alpha,E_\alpha\rangle
\langle Z_\lambda,Z_\lambda\rangle -\langle\ad([Z_\alpha,\theta
E_\alpha])Z_\lambda,Z_\lambda\rangle =
c\langle\lambda,\alpha\rangle\langle E_\alpha,E_\alpha\rangle
\langle Z_\lambda,Z_\lambda\rangle.
\end{align*}
Hence $c=0$ and our assertion follows.

Now we claim that
$\g{g}_\lambda\oplus\g{g}_{\lambda+2\alpha}\subset\hat{\g{s}}_n$. By
Lemma \ref{thCentre} we have
\begin{equation*}
Z_{\g{p}_\lambda\oplus\g{p}_{\lambda+2\alpha}}(\R(1-\theta)E_\alpha)=
\{(1-\theta)(\eta_\lambda+\eta_{\lambda+2\alpha})
\in\g{p}_\lambda\oplus\g{p}_{\lambda+2\alpha}:
[\eta_\lambda,E_\alpha]=[\eta_{\lambda+2\alpha},\theta E_\alpha]\}
\end{equation*}
whose dimension coincides with $\dim\g{p}_\lambda$. Note that
$\eta_{\lambda+2\alpha}$ is uniquely determined by $\eta_\lambda$
since
\begin{equation*}
[[\eta_\lambda,E_\alpha],E_\alpha]=[[\eta_{\lambda+2\alpha},\theta
E_\alpha],E_\alpha]=-[[\theta
E_\alpha,E_\alpha],\eta_{\lambda+2\alpha}]=-\lvert
E_\alpha\rvert^2\langle\alpha,\lambda+2\alpha\rangle
\eta_{\lambda+2\alpha}.
\end{equation*}

Since $H_\alpha\in\hat{\g{s}}_n$, by Lemma \ref{thCentre} there
exists $S\in\g{t}$ such that $S+H_\alpha\in\hat{\g{s}}$. We prove
that
$Z_{\g{p}_\lambda\oplus\g{p}_{\lambda+2\alpha}}(\R(1-\theta)E_\alpha)$
is invariant under $\ad(S)$. Let $X\in\g{g}_\alpha\ominus\R
E_\alpha$. Then there exists $T\in\g{t}$ such that
$T+X\in\hat{\g{s}}$ by Lemmas \ref{thsspp} and \ref{thCentre}.
Hence, $[S+H_\alpha,T+X]=(\ad(S)+\langle\alpha,\alpha\rangle
1_{\g{g}_\alpha})X\in\hat{\g{s}}\cap\g{g}_\alpha\subset\hat{\g{s}}_n$.
Using again Lemma \ref{thCentre} we get
$0=\langle(\ad(S)+\langle\alpha,\alpha\rangle
1_{\g{g}_\alpha})X,(1-\theta)E_\alpha\rangle=-\langle
X,\ad(S)E_\alpha\rangle$. Since $X\in\g{g}_\alpha\ominus\R E_\alpha$
and $\ad(S)E_\alpha\in\g{g}_\alpha\ominus\R E_\alpha$ (because
$\ad(S)$ is skewsymmetric), the above equation implies
$[S,E_\alpha]=0$. Note that $[S,\theta
E_\alpha]=\theta[S,E_\alpha]=0$. Now assume that
$(1-\theta)(\eta_\lambda+\eta_{\lambda+2\alpha})\in
Z_{\g{p}_\lambda\oplus\g{p}_{\lambda+2\alpha}}(\R(1-\theta)E_\alpha)$.
We have to show that
$(1-\theta)([S,\eta_\lambda]+[S,\eta_{\lambda+2\alpha}])\in
Z_{\g{p}_\lambda\oplus\g{p}_{\lambda+2\alpha}}(\R(1-\theta)E_\alpha)$.
Indeed, using the Jacobi identity and $[S,E_\alpha]=0$ we get
\begin{equation*}
[[S,\eta_\lambda],E_\alpha]=-[[\eta_\lambda,E_\alpha],S]
=-[[\eta_{\lambda+2\alpha},\theta
E_\alpha],S]=[[S,\eta_{\lambda+2\alpha}],E_\alpha].
\end{equation*}
This proves that
$Z_{\g{p}_\lambda\oplus\g{p}_{\lambda+2\alpha}}(\R(1-\theta)E_\alpha)$
is invariant under $\ad(S)$.

Let $Z_\lambda\in\g{g}_\lambda$. Then there exists
$Z_{\lambda+2\alpha}\in\g{g}_{\lambda+2\alpha}$ such that
$Z_\lambda+Z_{\lambda+2\alpha}$ is perpendicular to
$Z_{\g{p}_\lambda\oplus\g{p}_{\lambda+2\alpha}}
(\R(1-\theta)E_\alpha)$. Thus Lemma \ref{thsspp} yields that
$Z_\lambda+Z_{\lambda+2\alpha}\in\hat{\g{s}}_n$, and so there exists
$T\in\g{t}$ such that
$T+Z_\lambda+Z_{\lambda+2\alpha}\in\hat{\g{s}}$. Hence
$[S+H_\alpha,T+Z_\lambda+Z_{\lambda+2\alpha}]=[S,Z_\alpha
+Z_{\lambda+2\alpha}] +\langle\lambda, \alpha\rangle
Z_\lambda+\langle\lambda+2\alpha,\alpha\rangle
Z_{\lambda+2\alpha}\in\hat{\g{s}}\cap(\g{g}_\lambda
\oplus\g{g}_{\lambda+2\alpha})$. As $Z_\lambda+Z_{\lambda+2\alpha}$
is perpendicular to
$Z_{\g{p}_\lambda\oplus\g{p}_{\lambda+2\alpha}}(\R(1-\theta)E_\alpha)$
and
$Z_{\g{p}_\lambda\oplus\g{p}_{\lambda+2\alpha}}(\R(1-\theta)E_\alpha)$
is $\ad(S)$-invariant, it follows that
$[S,Z_\lambda+Z_{\lambda+2\alpha}]$ is also perpendicular to
$Z_{\g{p}_\lambda\oplus\g{p}_{\lambda+2\alpha}}
(\R(1-\theta)E_\alpha)$,
and so $[S,Z_\lambda+Z_{\lambda+2\alpha}]\in\hat{\g{s}}_n$ by Lemma
\ref{thTspp}. Hence, $\langle\lambda, \alpha\rangle
Z_\lambda+\langle\lambda+2\alpha,\alpha\rangle
Z_{\lambda+2\alpha}\in\hat{\g{s}}_n$. Since $\lambda$ and $\alpha$
span a root system of type $B_2$ or $BC_2$ we know that
$\langle\lambda,\alpha\rangle<0$ and
$\langle\lambda+2\alpha,\alpha\rangle>0$. Therefore we have
$Z_\lambda\in\hat{\g{s}}_n$, which implies
$\g{g}_{\lambda}\subset\hat{\g{s}}_n$. Similarly, one can show
$\g{g}_{\lambda+2\alpha}\subset\hat{\g{s}}_n$, which proves our
claim. Hence, $\xi_{\lambda},\xi_{\lambda+2\alpha}=0$.

We have that $(\g{g}_{\alpha}\ominus\R
E_\alpha)\oplus\g{g}_\lambda\subset\hat{\g{s}}_n$. Let
$X\in\g{g}_\alpha\ominus\R E_\alpha$ and $Y\in\g{g}_\lambda$. There
exist $S,T\in\g{t}$ such that $S+X,T+Y\in\hat{\g{s}}$. Hence,
$[S+X,T+Y]=[S,Y]-[T,X]+[X,Y]\in\hat{\g{s}}\cap(\g{g}_\alpha\oplus
\g{g}_\lambda\oplus\g{g}_{\alpha+\lambda})\subset\hat{\g{s}}_n$.
Then, by Lemma \ref{thsspp}, the right-hand side of the previous
equation is orthogonal to $(1-\theta)E_\alpha$ so
$[T,X]\in\g{g}_\alpha\ominus\R E_\alpha\subset\hat{\g{s}}_n$. Since
$[S,Y]\in\g{g}_\lambda\subset\hat{\g{s}}_n$ we conclude
$[X,Y]\in\hat{\g{s}}_n$. As $X$ and $Y$ are arbitrary,
$\xi_{\lambda+\alpha}\in[\g{g}_\alpha\ominus\R
E_\alpha,\g{g}_\lambda]\subset\hat{\g{s}}_n$. This implies
$\xi_{\lambda+\alpha}=0$ and finishes the proof for $\alpha$-strings
of $\lambda$ of length 3.

Finally, assume that the length of the $\alpha$-string of $\lambda$
is 4. In this case we may assume
$\lambda-\alpha,\lambda+4\alpha\not\in\Sigma^+$ and
$\lambda,\lambda+\alpha,\lambda+2\alpha,\lambda+3\alpha\in\Sigma^+$.
Then $\alpha$ and $\lambda$ span a root system of type $G_2$. A
consequence of this fact is that all four restricted root spaces
have the same dimension. Now, by Lemma \ref{thNonzeroBracket}, the
linear map $\ad(E_\alpha):\g{g}_\lambda\to\g{g}_{\lambda+\alpha}$ is
injective, and hence bijective. Thus, we can write
$\xi_{\lambda+\alpha}=[E_\alpha,X_\lambda]$ with
$X_\lambda\in\g{g}_\lambda$.  We get from Lemma \ref{thCentre} that
\[
0=[\xi_{\lambda-\alpha},E_\alpha]=[\xi_{\lambda+\alpha}, \theta
E_\alpha] =[[E_\alpha,X_\lambda],\theta E_\alpha]
=-[[\theta E_\alpha,E_\alpha],X_\lambda]
=-\langle\lambda,\alpha\rangle\langle E_\alpha,E_\alpha\rangle
X_\lambda.
\]
Since $\langle\lambda,\alpha\rangle\neq 0$, this implies
$X_\lambda=0$, and hence $\xi_{\lambda+\alpha}=0$. Since
$\lambda+3\alpha-(-\alpha)\not\in\Sigma$ and
$\lambda+3\alpha+(-\alpha)\in\Sigma$, Lemma \ref{thNonzeroBracket}
and $[\xi_{\lambda+3\alpha},\theta
E_\alpha]=[\xi_{\lambda+\alpha},E_\alpha]=0$ imply
$\xi_{\lambda+3\alpha}=0$.

Another application of Lemma \ref{thNonzeroBracket} implies that
$\ad(\theta
E_\alpha):\g{g}_{\lambda+3\alpha}\to\g{g}_{\lambda+2\alpha}$ is
injective and hence bijective. A similar argument as before writing
$\xi_{\lambda+2\alpha}=[\theta E_\alpha,X_{\lambda+3\alpha}]$ with
$X_{\lambda+3\alpha}\in\g{g}_{\lambda+3\alpha}$ yields, using Lemma
\ref{thCentre},
\begin{align*}
0&=[\xi_{\lambda+4\alpha},\theta E_\alpha]=[\xi_{\lambda+2\alpha},
E_\alpha]
=[[\theta E_\alpha,X_{\lambda+3\alpha}],E_\alpha]\\
&=-[[E_\alpha,\theta E_\alpha],X_{\lambda+3\alpha}]=
\langle\lambda+3\alpha,\alpha\rangle\langle E_\alpha,
E_\alpha\rangle X_{\lambda+3\alpha}.
\end{align*}
Hence, $X_{\lambda+3\alpha}=0$ and $\xi_{\lambda+2\alpha}=0$ because
$\langle\lambda+3\alpha,\alpha\rangle\neq 0$. Since
$\lambda-\alpha\not\in\Sigma$ and $\lambda+\alpha\in\Sigma$, Lemma
\ref{thNonzeroBracket} and
$[\xi_{\lambda},E_\alpha]=[\xi_{\lambda+2\alpha},\theta E_\alpha]=0$
imply $\xi_{\lambda}=0$.
\end{proof}

We define
$\Psi=\{\gamma\in\Lambda:\langle\gamma,\alpha\rangle=0\mbox{ for all
}\alpha\in\Phi\}$. The root subsystem of $\Sigma$ generated by
$\Psi$ is denoted by $\Sigma_\Psi$. We also denote by $2\Phi$ the
set of roots of the form $2\alpha$ with $\alpha\in\Phi$. Of course,
the number of elements of $2\Phi$ is at most the number of
irreducible components of $\Sigma$. The root subsystem generated by
$\Psi\cup\Phi$ is $\Sigma_\Psi\cup\Phi\cup 2\Phi$.

\begin{lemma}\label{thspp1}
We have
\begin{equation*}
{\hat{\g{s}}_{\g{p}}^\perp}\subset\g{a}_\Phi
\oplus\left(\bigoplus_{\alpha\in\Phi}
\R(1-\theta)E_\alpha\right)\oplus \left( {\mathfrak p}_\Psi \ominus
{\mathfrak a} \right) .
\end{equation*}
\end{lemma}

\begin{proof}
If $\Phi=\emptyset$ then $\Psi=\Lambda$ and $\Sigma_\Psi=\Sigma$,
and so the assertion is that
${\hat{\g{s}}_{\g{p}}^\perp}\subset\g{p}$ and there is nothing to
prove in that case. Hence, we may assume that $\Phi\neq\emptyset$.
It can happen that $\Phi\cup\Psi=\Lambda$. By definition of $\Psi$
this implies that $\Sigma$ is reducible, and in fact it is the
direct sum of two root systems, one generated by $\Psi$ and the
other one generated by $\Phi$. Moreover, each element of $\Phi$ is
in an irreducible component of rank one of $\Sigma$. In that case,
$\Sigma_\Psi=\Sigma\setminus(\Phi\cup 2\Phi)$ and the result follows
readily from Lemma \ref{thCentre}. Hence, we may also assume that
$\Phi\cup\Psi\neq\Lambda$.

Let $Z = H^{\Phi \cup \Psi}$ be the characteristic element in
$\g{a}$ of the gradation ${\mathfrak g} = \bigoplus_{k \in {\mathbb
Z}} {\mathfrak g}^k_{\Phi \cup \Psi}$ of ${\mathfrak g}$
corresponding to the parabolic subalgebra ${\mathfrak q}_{\Phi \cup
\Psi}$. We claim that if $\lambda\in\Sigma^+$ and $\lambda(Z)=1$
then there exists $\alpha\in\Phi$ such that
$\lambda+\alpha\in\Sigma^+$ or $\lambda-\alpha\in\Sigma^+$.

In order to prove this, let $\lambda\in\Sigma^+$ be such that
$\lambda(Z)=1$. If $\langle\lambda,\alpha\rangle\neq 0$ for some
$\alpha\in\Phi$, then $A_{\lambda\alpha}\neq 0$, where
$A_{\lambda\alpha}$ is the corresponding Cartan integer. This
clearly implies our claim. Thus we may assume
$\langle\lambda,\alpha\rangle=0$ for all $\alpha\in\Phi$. Write
$\lambda=\sum_{\gamma\in\Lambda}n_\gamma\gamma$ with $n_\gamma\geq
0$. Then by hypothesis,
$1=\lambda(Z)=\sum_{\gamma\in\Lambda\setminus(\Phi\cup\Psi)}
n_\gamma$. Since $n_\gamma\geq 0$ we can then write
$\lambda=\sum_{\alpha\in\Phi}n_\alpha\alpha+\beta+\mu$, where
$\beta\in\Lambda\setminus(\Phi\cup\Psi)$ and
$\mu\in\mathop{\mbox{span}}\Psi$. Now, since $\Phi$ consists of
orthogonal roots, for each $\alpha\in\Phi$ we have
$0=\langle\lambda,\alpha\rangle=n_\alpha\langle\alpha,\alpha\rangle
+\langle\beta,\alpha\rangle=(n_\alpha+A_{\beta\alpha}/2)
\langle\alpha,\alpha\rangle$ so the Cartan integer satisfies
$A_{\beta\alpha}=-2n_\alpha$. It cannot happen that $n_\alpha=0$ for
all $\alpha\in\Phi$ because in that case the previous equality
implies $\langle\beta,\alpha\rangle=0$ for all $\alpha\in\Phi$ and
hence $\beta\in\Psi$, contradiction. Therefore we can find
$\alpha\in\Phi$ such that $n_\alpha>0$. We will see that
$\lambda\pm\alpha\in\Sigma^+$, from where the claim will follow.

By the properties of Cartan integers, the equation
$A_{\beta\alpha}=-2n_\alpha$ can only hold when $n_\alpha=1$ and
$\{\alpha,\beta\}$ spans a root system of type $B_2$ (or $BC_2$). It
is also obvious that $\lambda$ must be a root of the root subsystem
of $\Sigma$ determined by the irreducible component where both
$\alpha$ and $\beta$ lie. Hence, the connected component of $\alpha$
and $\beta$ in the Dynkin diagram of the original root system
$\Sigma$ has a double arrow pointing to $\alpha$. Therefore, this
connected component is one of $B_r$, $C_r$, $BC_r$ or $F_4$. Relabel
the corresponding reduced Dynkin diagram as indicated in the
following figure:

\bigskip

$$
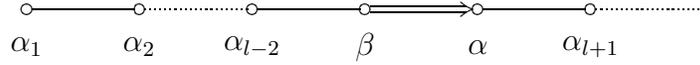
\xy \POS
(0,0) *\cir<2pt>{} ="a",
(15,0) *\cir<2pt>{}="b",
(30,0) *\cir<2pt>{}="c",
(45,0) *\cir<2pt>{}="d",
(60,0) *\cir<2pt>{}="e",
(75,0) *\cir<2pt>{}="f",
(90,0) ="g",
(0,-5) *{\alpha_1},
(15,-5) *{\alpha_2},
(30,-5) *{\alpha_{l-2}},
(45,-5) *{\beta},
(60,-5) *{\alpha},
(75,-5) *{\alpha_{l+1}},
\ar @{-} "a";"b",
\ar @{.} "b";"c",
\ar @{-} "c";"d",
\ar @2{->} "d";"e",
\ar @{-} "e";"f",
\ar @{.} "f";"g",
\endxy
$$

\bigskip

According to this labeling
$\mu\in\mathop{\mbox{span}}\{\alpha_1,\dots,\alpha_{l-2},
\alpha_{l+2},\dots\}$ (whenever the corresponding simple roots
exist). However, since $\alpha_{i}$ is orthogonal to $\alpha$ and
$\beta$ for $i\geq l+2$, it follows that $\lambda$ can be a root
only if $\mu\in\mathop{\mbox{span}}\{\alpha_1,\dots,\alpha_{l-2}\}$,
so the problem reduces to studying a root of the form
$\lambda=\sum_{i=1}^l n_i\alpha_i$ in a root system of type $B_l$ or
$BC_l$ with the labeling as above and with $n_{l-1}=1$, $n_l\geq 1$.
By the description of all roots for these root systems, $\lambda$
must be of the form $\lambda=\alpha_i+\dots+\alpha_l$ or
$\lambda=\alpha_i+\dots+\alpha_{l-1}+2\alpha_l$. Only the first of
these two possibilities can be orthogonal to $\alpha$, and in that
case it follows that $\lambda\pm\alpha\in\Sigma^+$.

Therefore, if $\lambda\in\Sigma^+$ and $\lambda(Z)=1$ there exists
$\alpha\in\Phi$ such that $\lambda+\alpha\in\Sigma^+$ or
$\lambda-\alpha\in\Sigma^+$. Since $\lambda\neq\alpha,2\alpha$, we
can now apply Lemma \ref{thString} to obtain that $\g{g}^1_{\Phi
\cup \Psi}\subset\hat{\g{s}}_n$. Since the gradation is of type
$\alpha_0$, it follows from Lemma \ref{thBracketsn} that
$\bigoplus_{k\geq 1}\g{g}^k_{\Phi \cup \Psi} \subset\hat{\g{s}}_n$.
This implies that ${\hat{\g{s}}_{\g{p}}^\perp}$ is contained in the
projection of $\g{g}^0$ onto $\g{p}$. Obviously $\lambda(Z)=0$ if
and only if $\lambda\in\Sigma_\Psi\cup\Phi\cup 2\Phi$. Combining
this with Lemma \ref{thCentre} we get the result.
\end{proof}

Now we turn our attention to the $\g{a}$-part of
${\hat{\g{s}}_{\g{p}}^\perp}$ and define
$\bar\Sigma=\{\lambda\in\Sigma:\lambda(\g{a}_\Phi\ominus
V)=0\}=\{\lambda\in\Sigma:H_\lambda\in V\oplus\g{a}^\Phi\}$. It is
obvious that $\bar\Sigma$ is a (possibly empty) root subsystem of
$\Sigma$. We denote by $\bar\Sigma^+$ a set of positive roots with
respect to an ordering consistent with that of $\Sigma$. We define
$\Pi=\Lambda\cap\bar\Sigma$ and denote by $\Sigma_\Pi$ the
corresponding root subsystem of $\Sigma$ generated by $\Pi$.
Consider in $\Sigma_\Pi$ an ordering compatible with that of
$\Sigma$ so that $\Sigma_\Pi^+=\Sigma_\Pi\cap\Sigma^+$.

\begin{lemma}\label{thsppImproved}
We have
\begin{equation*}
{\hat{\g{s}}_{\g{p}}^\perp}\subset\g{a}_\Phi
\oplus\left(\bigoplus_{\alpha\in\Phi}
\R(1-\theta)E_\alpha\right)\oplus \left( {\mathfrak p}_\Pi \ominus
{\mathfrak a} \right) .
\end{equation*}
\end{lemma}

\begin{proof}
Write $\xi\in\g{p}$ as
$\xi=\xi_0+\sum_{\lambda\in\Sigma^+}(1-\theta)\xi_\lambda$ with
$\xi_0\in\g{a}$ and $\xi_\lambda\in\g{g}_\lambda$ for each
$\lambda\in\Sigma^+$. Given $H\in\g{a}_\Phi\ominus V$ we get
$[H,\xi]=(1+\theta)\sum_{\lambda\in\Sigma^+}\lambda(H)\xi_\lambda$,
which implies that the centralizer of $\g{a}_\Phi\ominus V$ in
$\g{p}$ is $Z_{\g{p}}(\g{a}_\Phi\ominus V)=
\g{a}\oplus\left(\bigoplus_{\lambda\in\bar\Sigma^+}
\g{p}_\lambda\right)$.

For any $\alpha\in\Phi$ it is obvious that $\alpha(\g{a}_\Phi\ominus
V)=0$, and so $\Phi\subset\bar\Sigma$.  Using Lemma \ref{thspp1} and
$Z_{\g{p}}(\g{a}_\Phi\ominus V)=
\g{a}\oplus\left(\bigoplus_{\lambda\in\bar\Sigma^+}
\g{p}_\lambda\right)$ we easily get
\begin{equation*}
{\hat{\g{s}}_{\g{p}}^\perp}\subset\g{a}_\Phi
\oplus\left(\bigoplus_{\alpha\in\Phi}
\R(1-\theta)E_\alpha\right)\oplus
\left(\bigoplus_{\lambda\in\Sigma_\Psi^+\cap\bar\Sigma}
\g{p}_\lambda\right).
\end{equation*}
This implies that for any
$\lambda\in\Sigma^+\setminus(\Phi\cup(\Sigma_\Psi\cap\bar\Sigma))$
we have $\g{g}_\lambda\in\hat{\g{s}}_n$. We will use this fact
several times during this proof.

Let $Z = H^{\Lambda \setminus (\Psi \setminus \Pi)}$ be the
characteristic element in $\g{a}$ of the gradation ${\mathfrak g} =
\bigoplus_{k \in {\mathbb Z}} {\mathfrak g}^k_{\Lambda \setminus
(\Psi \setminus \Pi)}$ of ${\mathfrak g}$ corresponding to the
parabolic subalgebra ${\mathfrak q}_{\Lambda \setminus (\Psi
\setminus \Pi)}$. Let $\lambda\in\Sigma_\Psi^+$ be written as
$\lambda=\sum_{\gamma\in\Psi}n_\gamma \gamma$ and assume
$\lambda(Z)=1$. Then $1=\lambda(Z)=
\sum_{\gamma\in\Psi\setminus\Pi}n_\gamma$, so we can write
$\lambda=\alpha+\mu$, where $\alpha\in\Psi\setminus\Pi$ and
$\mu\in\mathop{\mbox{span}}\Pi$. By definition of $\mu$ it is
obvious that $\mu(\g{a}_\Phi\ominus V)=0$, and by definition of
$\alpha$ it is clear that $\lambda(\g{a}_\Phi\ominus
V)=\alpha(\g{a}_\Phi\ominus V)\neq 0$, that is,
$\lambda\not\in\bar\Sigma$. Thus we have
$\g{g}_\lambda\subset\hat{\g{s}}_n$. On the other hand, assume
$\lambda\in\Sigma^+\setminus\Sigma_\Psi^+$ satisfies $\lambda(Z)=1$.
Then we conclude $\g{g}_\lambda\subset\hat{\g{s}}_n$ unless
$\lambda\in\Phi$. The latter case is not possible since
$\Phi\subset\Lambda\setminus\Psi\subset\Lambda
\setminus(\Psi\setminus\Pi)$, which would imply $\lambda(Z)=0$.
Hence, the conclusion is that for any $\lambda\in\Sigma^+$
satisfying $\lambda(Z)=1$ we have
$\g{g}_\lambda\subset\hat{\g{s}}_n$. This implies that
$\g{g}^1_{\Lambda \setminus (\Psi \setminus \Pi)}
\subset\hat{\g{s}}_n$, and hence, by Lemma 4.13 we have
$\bigoplus_{k\geq 1}\g{g}^k_{\Lambda \setminus (\Psi \setminus \Pi)}
\subset\hat{\g{s}}_n$. Combining this with the above inclusion for
${\hat{\g{s}}_{\g{p}}^\perp}$ implies the result.
\end{proof}

We are now ready to determine $\g{s}_n$.

\begin{lemma}\label{thhatsn}
We have $\g{s}_n=\g{s}_{\Phi,V,a}$.
\end{lemma}

\begin{proof}
Fix $\alpha\in\Pi$. Since $\alpha\in\bar\Sigma$ we  have
$\alpha(\g{a}_\Phi\ominus V)=0$ and hence $H_\alpha\in
V\oplus\g{a}^\Phi$. On the other hand, $\alpha\in\Sigma_\Psi^+$ so
$H_\alpha\in\g{a}_\Phi$. Since $\g{a}_\Phi\cap\g{a}^\Phi=\{0\}$,
Proposition \ref{thProperties} (iii) applied to $\tilde{{\g{s}}}$
implies $H_\alpha\in V\subset(\Ad(g)\tilde{{\g{s}}})_n$. Thus there
exists $S\in\g{t}$ such that $S+H_\alpha\in\Ad(g)\tilde{{\g{s}}}$.
Let $X\in\g{g}_\alpha$. By definition of $\Pi$ and Proposition
\ref{thProperties} (iv) we get
$\g{g}_\alpha\subset(\Ad(g)\tilde{{\g{s}}})_n$. Then there exists
$T\in\g{t}$ such that $T+X\in\Ad(g)\tilde{{\g{s}}}$. Since
$\Ad(g)\tilde{{\g{s}}}$ is a Lie algebra and
$[\g{t},\g{g}_\alpha]\subset\g{g}_\alpha$ we have
$[S+H_\alpha,T+X]=[S,X]+[H_\alpha,X]
=(\ad(S)+\langle\alpha,\alpha\rangle 1_{\g{g}_{\alpha}})X
\in(\Ad(g)\tilde{{\g{s}}})\cap\g{g}_\alpha$, where
$1_{\g{g}_{\alpha}}$ is the identity of ${\g{g}_{\alpha}}$. Since
the linear map $\ad(S)$ is skewsymmetric, it follows that
$\ad(S)+\langle\alpha,\alpha\rangle
1_{\g{g}_{\alpha}}:\g{g}_\alpha\to\g{g}_\alpha$ is an isomorphism.
However, $X\in\g{g}_\alpha$ is arbitrary, and so
$\g{g}_\alpha\subset\Ad(g)\tilde{{\g{s}}}$. Since $\alpha\in\Pi$ is
also arbitrary it follows that for all $\alpha\in\Pi$ we have
$\g{g}_\alpha\subset\Ad(g)\tilde{{\g{s}}}$.

Let us denote by $\bar{\g{n}}^s = {\mathfrak n}^s \cap {\mathfrak
n}_\Pi$ the direct sum of root spaces associated with roots of
$\Sigma_\Pi^+$ of level $s$ (note that the level of a root in
$\Sigma_\Pi^+$ coincides with the level of this root as a root of
$\Sigma$). The previous argument shows that
$\bar{\g{n}}^1\subset\Ad(g)\tilde{{\g{s}}}$. Choose a basis
$\{E_1,\dots,E_k\}$ of $\bar{\g{n}}^1$. Since $g\in N$ it follows
that $\Ad(g)(\g{n} \ominus \g{n}^1) \subset \g{n} \ominus \g{n}^1$,
and as $\Ad(g)$ is an automorphism equality holds. Hence, by
definition of $\tilde{{\g{s}}}$ and $\hat{\g{s}}$ it is obvious that
$\Ad(g)\tilde{{\g{s}}}=\hat{\g{s}}+(\g{n} \ominus \g{n}^1)$.
Therefore, for each $i$, there exists $X_i\in\g{n} \ominus \g{n}^1$
such that $E_i+X_i\in\hat{\g{s}}$.

We introduce the following notation. Define
$[Y_1,Y_2,\dots,Y_l]=[Y_1,[Y_2,\dots,Y_l]]$ inductively, being
$[Y_1,Y_2]$ the usual Lie bracket. Denote by $k$ the level of the
highest root of $\Sigma_\Pi^+$. Let $s$ be the smallest integer for
which $\bar{\g{n}}^s\oplus\dots
\oplus\bar{\g{n}}^k\subset\hat{\g{s}}_n$. Our aim is to prove $s=1$.

First we prove $s\leq k$, that is,
$\bar{\g{n}}^k\subset\hat{\g{s}}_n$. Since
$[\g{n}^a,\g{n}^b]\subset\g{n}^{a+b}$ and
$\g{n}\ominus(\g{n}^1\oplus\bar{\g{n}}^2\oplus\dots
\oplus\bar{\g{n}}^k)\subset\hat{\g{s}}_n$ by Lemma
\ref{thsppImproved}, we have
$[E_{i_1}+X_{i_1},\dots,E_{i_k}+X_{i_k}]\equiv[E_{i_1},\dots,E_{i_m}]
\mod\hat{\g{s}}_n$. Here we have used the fact that the level of a
root in $\Sigma_\Pi^+$ is the same as the level of that root as a
root of $\Sigma^+$. The brackets of $k$ vectors in the right-hand
side of the previous formula span $\bar{\g{n}}^k$ whereas the
brackets on the left-hand side belong to $\hat{\g{s}}\cap\g{n}$
because $\hat{\g{s}}\cap\g{n}$ is a subalgebra. Since
$\hat{\g{s}}\cap\g{n}\subset\hat{\g{s}}_n$, this implies
$\bar{\g{n}}^k\subset\hat{\g{s}}_n$.

Now assume $s>1$. Hence
$\bar{\g{n}}^s\oplus\dots\oplus\bar{\g{n}}^k\subset\hat{\g{s}}_n$
but $\bar{\g{n}}^{s-1}\not\subset\hat{\g{s}}_n$. We use again
$[\g{n}^a,\g{n}^b]\subset\g{n}^{a+b}$ and
$\g{n}\ominus(\g{n}^1\oplus\bar{\g{n}}^2\oplus\dots
\oplus\bar{\g{n}}^{s-1})\subset\hat{\g{s}}_n$, which follows from
Lemma \ref{thsppImproved} and the definition of $s$. Thus we get
$[E_{i_1}+X_{i_1},\dots,E_{i_{s-1}}+X_{i_{s-1}}]\equiv
[E_{i_1},\dots,E_{i_{s-1}}]\mod\hat{\g{s}}_n$. Again, the brackets
in the right-hand side of the congruency span $\bar{\g{n}}^{s-1}$
whereas the brackets in the left-hand side belong to
$\hat{\g{s}}\cap\g{n}\subset\hat{\g{s}}_n$. Then we get
$\bar{\g{n}}^{s-1}\subset\hat{\g{s}}_n$, which is a contradiction.
Therefore $s=1$ and thus
$\bar{\g{n}}^1\oplus\dots\oplus\bar{\g{n}}^k\subset\hat{\g{s}}_n$.
Altogether this implies
\begin{equation*}
(\g{a}_\Phi\ominus V) \oplus\left(\bigoplus_{\alpha\in\Phi}
\R(1-\theta)E_\alpha\right)\subset
{\hat{\g{s}}_{\g{p}}^\perp}\subset\g{a}_\Phi
\oplus\left(\bigoplus_{\alpha\in\Phi} \R(1-\theta)E_\alpha\right).
\end{equation*}

Since $V\subset(\Ad(g)\tilde{{\g{s}}})_n$ by Proposition
\ref{thProperties} (iii) and $\g{n} \ominus \g{n}^1
\subset\hat{\g{s}}_n$ by the above equation, it follows from
$\Ad(g)\tilde{{\g{s}}}=\hat{\g{s}}+(\g{n} \ominus \g{n}^1)$ that
$\hat{\g{s}}_n=\g{s}_{\Phi,V}$.

Let $T+H+X\in\hat{\g{s}}$ with $T\in\g{t}$, $H\in\g{a}$ and
$X\in\g{n}$. By hypothesis, the connected subgroup $\hat{S}$ of $G$
with Lie algebra ${\hat{\g{s}}}$ induces a hyperpolar foliation. Let
$E=-\sum_{\alpha\in\Phi}a_\alpha E_\alpha$ and $g=\Exp(E)$. It
follows from Proposition \ref{thProperties} (vi) that
$[\hat{\g{s}}_c,E]=0$, and hence
$\Ad(g^{-1})(T+H+X)=T+\Ad(g^{-1})(H+X)$. Proposition
\ref{thCongruency} shows that
$\Ad(g^{-1})\g{s}_{\Phi,V}=\g{s}_{\Phi,V,a}$. Since
$\g{s}=\Ad(g^{-1})\hat{\g{s}}$, the result follows.
\end{proof}

To conclude our proof we need to prove the following result:

\begin{proposition}\label{thCongruencyGeneral}
Let $\g{t}\oplus\g{a}\oplus\g{n}$ be a maximally noncompact Borel
subalgebra of $\g{g}$. Let ${\g{s}}$ be a subalgebra of
$\g{t}\oplus\g{a}\oplus\g{n}$ such that
${\g{s}}_n=\pi_{\g{a}\oplus\g{n}}(\g{s}) =\g{s}_{\Phi,V,a}$ with
some orthogonal subset $\Phi$ of $\Lambda$. Assume that the orbits
of the connected subgroup $S$ of $G$ whose Lie algebra is $\g{s}$
form a homogeneous foliation on $M$. Then the actions of $S$ and
$S_{\Phi,V}$ are orbit equivalent.
\end{proposition}

\begin{proof}
First, assume $\g{s}$ is a subalgebra of
$\g{t}\oplus\g{a}\oplus\g{n}$ such that $\g{s}_n=\g{s}_{\Phi,V}$.
Certainly, $\g{s}_n$ is a subalgebra of $\g{a}\oplus\g{n}$. Denote
by $S$ and $S_n$ the corresponding connected subgroups of $G$. Also,
denote by $\bar{N}$ the connected subgroup of $G$ whose subalgebra
is $\g{n}\ominus(\bigoplus_{\alpha\in\Phi}\R E_\alpha)$. We prove
that $S$ and $S_n$ have the same orbits.

Assume that $T+H\in\g{s}$ with $H\in\g{a}$ and $t\in\g{t}$. Let
$X\in\g{n}\ominus\left(\bigoplus_{\alpha\in\Phi}\R E_\alpha\right)$.
By definition, there exists $R\in\g{s}$ such that $R+X\in\g{s}$. As
$\g{t}\oplus\g{a}$ is abelian, $[T+H,X]=[T+H,R+X]\in\g{s}\cap\g{n}$.
Hence, if $tan\in S$, there exists $n'\in\bar{N}$ such that
$tan=n'ta$. Since $\g{t}\oplus\g{a}$ is abelian we have $ta=at$, and
since $\g{a}$ normalizes
$\g{n}\ominus\left(\bigoplus_{\alpha\in\Phi}\R E_\alpha\right)$,
there exists $n''\in\bar{N}$ such that $n'a=an''$. Altogether this
implies $tan=n'ta=n'at=an''t$. Thus, $tan(o)=an''t(o)=an''(o)$ and
hence $S\cdot o\subset S_n\cdot o$. Since both orbits $S\cdot o$ and
$S_n\cdot o$ have the same dimension and are connected and complete
we conclude $S\cdot o=S_n\cdot o$. Now, let $p=\exp_o(\xi)$ with
$\xi\in\nu_o(S\cdot o)$. Using the fact that $S$ acts isometrically
on $M$ and that $t_*\xi=\xi$ by Proposition \ref{thProperties} (vii)
we get
\begin{equation*}
tan(p)=\exp_{tan(o)}((tan)_*\xi)
=\exp_{an''(o)}((an'')_*\xi)=an''(p).
\end{equation*}
Hence, $S\cdot p\subset S_n\cdot p$, and thus equality holds. Since
the action of $S$ is hyperpolar, all the orbits can be obtained in
this way, and so $S$ and $S_n$ have the same orbits as claimed
above.

Now we deal with the general case, that is, ${\g{s}}_n
=\g{s}_{\Phi,V,a}$. Let ${S}$ be the connected subgroup of $G$ with
Lie algebra ${\g{s}}$. By Proposition \ref{thProperties} (v), there
exists $g\in N$ such that $(\Ad(g){\g{s}})_n=\Ad(g){\g{s}}_n=
\g{s}_{\Phi,V}$. The subgroup $I_g(S_n)$ whose Lie algebra is
$\Ad(g){\g{s}}_n=(\Ad(g){\g{s}})_n$ has the same orbits as $I_g(S)$
by the previous argument. Then $I_g({S}_n)$ and $S$ have the same
orbits and hence the theorem follows.
\end{proof}

Now we finish the proof of Theorem \ref{maintheorem} (ii). Let $H$
be a closed subgroup of the isometry group of $M$ inducing a
hyperpolar homogeneous foliation on $M$. By Proposition
\ref{thBorel}, the action of $H$ is orbit equivalent to the action
of a closed solvable subgroup $S$ whose Lie algebra $\g{s}$ is
contained in a maximally noncompact Borel subalgebra. Then, there
exists a Cartan decomposition of $\g{g}=\g{k}\oplus\g{p}$ and a root
space decomposition $\g{g}=\g{g}_0\oplus\left(
\bigoplus_{\lambda\in\Sigma}\g{g}_\lambda\right)$ with respect to a
maximal abelian subspace $\g{a}$ of $\g{p}$ such that the projection
of $\g{s}$ onto $\g{a}\oplus\g{n}$ is given by
$\g{s}_n=\g{s}_{\Phi,V,a}$ by Lemma \ref{thhatsn}. Proposition
\ref{thCongruencyGeneral} then implies that the actions of the
connected subgroups of $G$ with Lie algebras $\g{s}$ and
$\g{s}_{\Phi,V,a}$ are orbit equivalent. Hence, the action of $H$ is
orbit equivalent to the action of $S_{\Phi,V}$ on $M$, which
concludes the proof of Theorem \ref{maintheorem}.

\section{Geometry of the leaves of hyperpolar homogeneous
foliations}\label{secGeometry}

In this section we study the extrinsic geometry of the leaves
of hyperpolar homogeneous foliations on noncompact symmetric
spaces. We identify $M$ with $AN$ endowed with the inner
product $\langle\,\cdot\,,\,\cdot\,\rangle_{AN}$ as described
in Section \ref{secPreliminaries}. However, we still denote by
$\ominus$ orthogonal complement in $\g{g}$ with respect to
$\langle\,\cdot\,,\,\cdot\,\rangle$.

\begin{proposition}\label{thOrbits}
The orbit $S_{\Phi,V}\cdot p$ is isometrically congruent to
$S_{\Phi,V,a}\cdot o$ for some {$a:\Phi\to\R$}.
\end{proposition}

\begin{proof}
Let $\mathcal{D}$ be the left-invariant distribution on $M$
determined by $(\g{a}\ominus V)\oplus\ell_\Phi$. Obviously,
$(\g{a}\ominus V)\oplus\ell_\Phi$ is a subalgebra of
$\g{a}\oplus\g{n}$, and since $AN$ is simply connected, the
leaf of $\mathcal{D}$ through $o$ is
$\mathcal{D}_o=\Exp((\g{a}\ominus V)\oplus\ell_\Phi)\cdot o$.
We prove that $\mathcal{D}$ is autoparallel. Using the formula
for the Levi-Civita connection it is easy to obtain $\nabla_H
H'=\nabla_H E_\alpha=0$,
$\nabla_{E_\alpha}H=-\alpha(H)E_\alpha$ and
$2\nabla_{E_\alpha}E_\beta=\langle E_\alpha,E_\beta\rangle
H_\alpha$ for any $H,H'\in\g{a}\ominus V$ and
$\alpha,\beta\in\Phi$.

In particular, the leaf $\mathcal{D}_o$ is totally geodesic in $M$
and since $\nu_o(S_{\Phi,V}\cdot o)=(\g{a}_\Phi\ominus
V)\oplus\ell_\Phi$, it contains the section of $S_{\Phi,V}$ through
$o$. Since a section of $S_{\Phi,V}$ intersects all the orbits, we
may assume that $p$ lies in that section and so, the point $p$ can
be written as $p=g(o)$ with $g=\Exp(X)$ and $X\in(\g{a}\ominus
V)\oplus\ell_\Phi$. Hence, $S_{\Phi,V}\cdot
p=g(g^{-1}S_{\Phi,V}g)\cdot o=g I_g(S_{\Phi,V})\cdot o$. Since $g$
is an isometry of $M$, the orbit $S_{\Phi,V}\cdot p$ is
isometrically congruent to $I_g(S_{\Phi,V})\cdot o$. We will now
prove that $I_g(S_{\Phi,V})=S_{\Phi,V,a}$ for some $a:\Phi\to\R$,
and for that we will show that
$\Ad(g)\g{s}_{\Phi,V}=\g{s}_{\Phi,V,a}$.

Let $H\in\g{a}$. We show that $\Ad(\Exp
H)\g{s}_{\Phi,V}=\g{s}_{\Phi,V}$. Since $\Ad(\Exp
H)\g{s}_{\Phi,V}\subset\g{a}\oplus\g{n}$, it suffices to prove that
$\Ad(\Exp H)\g{s}_{\Phi,V}$ is orthogonal to $(\g{a}_\Phi\ominus
V)\oplus\ell_\Phi$. Let $X\in\g{s}_{\Phi,V}$,
$H'\in\g{a}_\Phi\ominus V$ and $\alpha\in\Phi$. Then our assertion
follows from $\langle\Ad(\Exp H)X,H'\rangle=\langle
X,\Ad(\Exp(-\theta H))H'\rangle=\langle
X,e^{\ad(H)}H'\rangle=\langle X,H'\rangle=0$, and
\begin{align*}
\langle\Ad(\Exp H)X,E_\alpha\rangle &= \langle
X,e^{\ad(H)}E_\alpha\rangle =\sum_{k=0}^\infty
\frac{\alpha(H)^k}{k!}\langle X,E_\alpha\rangle=
e^{\alpha(H)}\langle X,E_\alpha\rangle=0.
\end{align*}

Write $X=H+\sum_{\alpha\in\Phi}(x_\alpha H_\alpha+y_\alpha
E_\alpha)$ with $H\in\g{a}_\Phi\ominus V$ and
$x_\alpha,y_\alpha\in\R$. Since $\R H_\alpha\oplus\R E_\alpha$ is a
subalgebra, there exist constants $a_\alpha,b_\alpha\in\R$ such that
$\Exp(a_\alpha E_\alpha)\cdot\Exp(b_\alpha H_\alpha)=\Exp(x_\alpha
H_\alpha+y_\alpha E_\alpha)$. This equation and
$[\g{a}_\Phi,\g{g}_\Phi]=[\g{g}_{\{\alpha\}},\g{g}_{\{\beta\}}]=\{0\}$
for any $\alpha,\beta\in\Phi$, $\alpha\neq\beta$, imply
\begin{align*}
g   &= \left(\prod_{\alpha\in\Phi}\Exp(x_\alpha H_\alpha+y_\alpha
E_\alpha)\right)\Exp H = \left(\prod_{\alpha\in\Phi}\Exp(a_\alpha
E_\alpha)\Exp(b_\alpha
H_\alpha)\right)\Exp H\\
&= \Exp\left(\sum_{\alpha\in\Phi}a_\alpha
E_\alpha\right)\Exp\left(H+\sum_{\alpha\in\Phi}b_\alpha
H_\alpha\right).
\end{align*}
Hence, the equality $\Ad(\Exp H')\g{s}_{\Phi,V}=\g{s}_{\Phi,V}$ for
any $H'\in\g{a}$ and Proposition \ref{thCongruency} imply
\begin{equation*}
\Ad(g)\g{s}_{\Phi,V}=\Ad\left(\Exp\left(\sum_{\alpha\in\Phi}a_\alpha
E_\alpha\right)\right)
\Ad\left(\Exp\left(H+\sum_{\alpha\in\Phi}b_\alpha
H_\alpha\right)\right)=\g{s}_{\Phi,V,a},
\end{equation*}
where $a:\Phi\to\R$, $\alpha\mapsto a_\alpha$. This concludes the
proof.
\end{proof}

In view of Proposition \ref{thOrbits}, in order to calculate the
geometry of the orbits of $S_{\Phi,V}$, it suffices to study the
geometry of the orbit through the origin $o$ of $S_{\Phi,V,a}$,
where $\Phi$ is an orthogonal subset of $\Lambda$, $V$ is a linear
subspace of $\g{a}_\Phi$ and $a:\Phi\to\R$ is a function. Hence, we
consider
\begin{equation*}
{\g{s}_{\Phi,V,a}}=(V\oplus\g{a}^\Phi\oplus\g{n})\ominus
\left(\bigoplus_{\alpha\in\Phi}\R(a_{\alpha}H_{\alpha}+E_{\alpha})
\right),
\end{equation*}
for certain $a_\alpha\in\R$ and $E_\alpha\in\g{g}_\alpha$ with $\langle E_\alpha,E_\alpha\rangle=1$.
Note that
\begin{equation*}
{\g{s}_{\Phi,V,a}}=V\oplus\left(\bigoplus_{\alpha\in\Phi}\R
X_{\alpha}\right)\oplus\left(\bigoplus_{\alpha\in\Phi}
((\g{g}_{\alpha}\ominus\R
E_{\alpha})\oplus\g{g}_{2\alpha})\right)\oplus
\left(\bigoplus_{\lambda\in\Sigma^+ \setminus(\Phi\cup2\Phi)}
\g{g}_\lambda\right),
\end{equation*}
where $X_{\alpha}=\frac{1}{\lvert\alpha\rvert^2}H_{\alpha}
-a_{\alpha}E_{\alpha}$. We denote by $2\Phi$ the set of roots of the
form $2\alpha$ with $\alpha\in\Phi$.

Let $X,Y\in{\g{s}_{\Phi,V,a}}$ and
$\xi\in\nu_o(S_{\Phi,V,a}\cdot o)$. Using the
formula for the Levi-Civita connection with respect to
left-invariant vector fields we easily obtain
\begin{equation*}
4\langle A_\xi X,Y\rangle_{AN}=\langle[(1-\theta)\xi,X],Y\rangle.
\end{equation*}

From now on let $Z\in\g{s}_{\Phi,V,a}$.

Assume first that $\xi\in\g{a}_\Phi\ominus V$. If $X\in V$,
then $2\langle A_\xi
X,Z\rangle_{AN}=\langle[\xi,X],Z\rangle=0$. Since
$\xi\in\g{a}_\Phi$ we also have $2\langle A_\xi
X_{\alpha},Z\rangle_{AN}=\langle[\xi,X_{\alpha}],Z\rangle=
-a_{\alpha}\alpha(\xi)\langle E_{\alpha},Z\rangle=0$.
Analogously, if $X\in\g{g}_{\alpha}\ominus\R E_{\alpha}$ then
$2\langle A_\xi X,Z\rangle_{AN}=\alpha(\xi)\langle
X,Z\rangle=0$, and similarly, if $X\in\g{g}_{2\alpha}$ then
also $A_\xi X=0$. Finally, if $X\in\g{g}_\lambda$ with
$\lambda\in\Sigma^+\setminus(\Phi\cup2\Phi)$, then $2\langle
A_\xi X,Z\rangle_{AN}=\langle
[\xi,X],Z\rangle=\lambda(\xi)\langle
X,Z\rangle=2\lambda(\xi)\langle X,Z\rangle_{AN}$. In
particular, $\tr
A_\xi=\sum_{\lambda\in\Sigma^+\setminus(\Phi\cup 2\Phi)}(\dim
\g{g}_\lambda)\lambda(\xi) =2\delta(\xi)$, where as usual
$\delta=\frac{1}{2}\sum_{\lambda\in\Sigma^+}(\dim
\g{g}_\lambda)\lambda$ (see \cite[p. 329]{K96}).

Now let $\xi=a_{\alpha}H_{\alpha}+2E_{\alpha}$ for some
$\alpha\in\Phi$. The vector $\xi$ is orthogonal to
$\g{s}_{\Phi,V,a}$ with respect to
$\langle\,\cdot\,,\,\cdot\,\rangle_{AN}$ and has length
$\sqrt{2+a_\alpha^2\lvert\alpha\rvert^2}$. If $X\in V$, then it
follows that
\begin{equation*}
\langle A_{\xi} X,Z\rangle_{AN}=\frac{1}{4}\langle[(1-\theta)\xi,X],Z\rangle
=-\frac{1}{2}\alpha(X)\langle E_\alpha,Z\rangle =0.
\end{equation*}
For $\beta\in\Phi$ we calculate
\begin{equation*}
\langle A_{\xi} X_{\beta},Z\rangle_{AN}
=\left\langle -a_{\alpha}a_{\beta}\langle\alpha,\beta\rangle E_{\beta}
-\frac{\langle\alpha,\beta\rangle}{\lvert{\beta}\rvert^2}
E_{\alpha} +\frac{1}{2}a_{\beta} \langle
E_{\alpha},E_{\beta}\rangle H_\alpha,Z\right\rangle_{AN}.
\end{equation*}
If $\alpha\neq\beta$, we have as usual that $\alpha$ and $\beta$ are
orthogonal and hence $A_{\xi} X_{\beta}=0$. If $\alpha=\beta$, we
can write the above expression in terms of $\xi$ and $X_{\alpha}$ to
get
\begin{align*}
\langle A_{\xi} X_{\alpha},Z\rangle_{AN}
=\langle a_{\alpha}\lvert\alpha\rvert^2X_{\alpha}
-\frac{1}{2}\xi,Z\rangle_{AN}
=a_{\alpha}\lvert\alpha\rvert^2\langle X_{\alpha},Z\rangle_{AN}.
\end{align*}
Assume now that $X\in\g{g}_{\beta}\ominus\R E_{\beta}$ with
$\beta\in\Phi$ and $\alpha\neq\beta$. Since $\alpha$ and $\beta$ are
orthogonal we get
$[(1-\theta)\xi,X]=2(a_{\alpha}\langle\alpha,\beta\rangle X
+[E_{\alpha},X] -[\theta E_{\alpha},X])=0$, and thus $A_{\xi}X=0$.
One can  prove in a similar way that $A_{\xi}X=0$ if
$X\in\g{g}_{2\beta}$ with $\beta\in\Phi$ and $\alpha\neq\beta$.

Now we turn our attention to the subspace
$(\g{g}_{\alpha}\ominus\R E_\alpha)\oplus\g{g}_{2\alpha}$. Let
$X\in\g{g}_{\alpha}\ominus\R E_\alpha$. Clearly, $[\theta
E_{\alpha},X]\in\g{g}_0$ and $\langle[\theta
E_{\alpha},X],H\rangle=\alpha(H)\langle X,E_\alpha\rangle=0$ for
all $H\in\g{a}$. Then we get
\begin{equation*}
\langle A_{\xi}X,Z\rangle_{AN}=\frac{1}{2}\langle[a_{\alpha}H_{\alpha}+E_{\alpha}
-\theta E_{\alpha},X],Z\rangle=\langle a_{\alpha}\lvert\alpha\rvert^2
X+[E_{\alpha},X],Z\rangle_{AN}.
\end{equation*}
On the other hand, if $Y\in\g{g}_{2\alpha}$ we have
\begin{equation*}
\langle A_{\xi}Y,Z\rangle_{AN}=\frac{1}{2}\langle[a_{\alpha}H_{\alpha}+E_{\alpha}
-\theta E_{\alpha},Y],Z\rangle=\langle 2a_{\alpha}\lvert\alpha\rvert^2
Y-[\theta E_{\alpha},Y],Z\rangle_{AN}.
\end{equation*}
Since $\langle[\theta
E_{\alpha},Y],E_\alpha\rangle=
-\langle Y,[E_{\alpha},E_{\alpha}]\rangle=0$,
$A_{\xi}$ leaves the subspace
$(\g{g}_{\alpha}\ominus\R E_\alpha)\oplus\g{g}_{2\alpha}$ invariant.
Moreover, since for $Y\in\g{g}_{2\alpha}$ we have
$[E_{\alpha},[\theta E_{\alpha},Y]]=-[Y,[E_{\alpha},\theta
E_{\alpha}]]=-2\lvert\alpha\rvert^2 Y$, the linear map
$\ad(E_{\alpha}){\vert\ad(\theta
E_{\alpha})(\g{g}_{2\alpha})}:\ad(\theta
E_{\alpha})(\g{g}_{2\alpha})\to\g{g}_{2\alpha}$ is an isomorphism.
From here we get the decomposition
$\g{g}_{\alpha}=\mathop{\rm{Ker}}(\ad(E_{\alpha}){\vert
\g{g}_{\alpha}})\oplus\ad(\theta E_{\alpha})(\g{g}_{2\alpha})$.
Hence, if $X\in\mathop{\mbox{Ker}}(\ad(E_{\alpha}){\vert
\g{g}_{\alpha}})$ we get from the previous expression that
$A_{\xi}X=a_{\alpha}\lvert\alpha\rvert^2 X$, so $A_{\xi}$ restricted
to $\mathop{\rm{Ker}}(\ad(E_{\alpha}){\vert \g{g}_{\alpha}})\ominus\R E_\alpha$ is
$a_{\alpha}\lvert\alpha\rvert^21_{\mathop{\rm
Ker}(\ad(E_{\alpha}){\vert \g{g}_{\alpha}})\ominus\R E_{\alpha}}$ and the
multiplicity is $\dim \g{g}_{\alpha}-\dim \g{g}_{2\alpha}-1$. On the
other hand, for nonzero $Y\in\g{g}_{2\alpha}$ define $X=[\theta
E_{\alpha},Y]\in\g{g}_{\alpha}$. Then the previous formulas read
$A_{\xi}X=a_{\alpha}\lvert\alpha\rvert^2 X-2\lvert\alpha\rvert^2Y$
and $A_{\xi}Y=-X+2a_{\alpha}\lvert\alpha\rvert^2 Y$. The
eigenvalues of the matrix
\begin{equation*}
\left(
\begin{array}{cc}
a_{\alpha}\lvert\alpha\rvert^2   &   -1\\
-2\lvert\alpha\rvert^2         &
2a_{\alpha}\lvert\alpha\rvert^2
\end{array}\right)
\end{equation*}
are $\frac{\lvert\alpha\rvert}{2}
\left(3a_{\alpha}\lvert\alpha\rvert
\pm\sqrt{8+a_{\alpha}^2\lvert\alpha\rvert^2}\right)$.

Before continuing we need the following (recall that
$\xi=a_{\alpha}H_{\alpha}+2E_{\alpha}$)

\begin{lemma}\label{thad(1-z)xi}
$\ad((1-\theta)\xi)
\left(\g{n}\ominus\left(\bigoplus_{\gamma\in\Phi}
(\g{g}_{\gamma}\oplus\g{g}_{2\gamma})\right)\right)
\subset\g{n}\ominus\left(\bigoplus_{\gamma\in\Phi}
(\g{g}_{\gamma}\oplus\g{g}_{2\gamma})\right)$.
\end{lemma}

\begin{proof}
Let $Y\in\g{n}\ominus\left(\bigoplus_{\gamma\in\Phi}
(\g{g}_{\gamma}\oplus\g{g}_{2\gamma})\right)$. By the properties of
root systems, it is clear that $\ad((1-\theta)\xi)Y\subset\g{n}$.
For $\beta\in\Phi$ and $Z\in\g{g}_{\beta}$ we calculate
\begin{equation*}
\langle\ad((1-\theta)\xi)Y,Z\rangle
=\langle\ad((1-\theta)\xi)Z,Y\rangle
=2a_{\alpha}\langle\alpha,\beta\rangle \langle Z,Y\rangle
+\langle[E_{\alpha},Z]-[\theta E_{\alpha},Z],Y\rangle.
\end{equation*}
By assumption we have $\langle Z,Y\rangle=0$. Now, if
$\alpha\neq\beta$, $[E_{\alpha},Z]\in\g{g}_{\alpha+\beta}=0$ and
$[\theta E_{\alpha},Z]\in\g{g}_{\beta-\alpha}=0$. If $\alpha=\beta$,
$[E_{\alpha},Z]\in\g{g}_{2\alpha}$ and so
$\langle[E_{\alpha},Z],Y\rangle=0$, and $[\theta E_{\alpha},Z]
\in\g{g}_0$, and so $\langle[\theta E_{\alpha},Z],Y\rangle=0$. In
any case, $\langle\ad((1-\theta)\xi)Y,Z\rangle=0$.

If $Z\in\g{g}_{2\beta}$ a similar calculation shows
$\langle\ad((1-\theta)\xi)Y,Z\rangle =\langle[E_{\alpha},Z],Y\rangle
-\langle[\theta E_{\alpha},Z],Y\rangle$. If $\alpha\neq\beta$,
$[E_{\alpha},Z]\in\g{g}_{\alpha+2\beta}=0$ and $[\theta
E_{\alpha},Z]\in\g{g}_{2\beta-\alpha}=0$. If $\alpha=\beta$,
$[E_{\alpha},Z]\in\g{g}_{3\alpha}=0$ and $[\theta
E_{\alpha},Z]\in\g{g}_{\alpha}$ so $\langle[\theta
E_{\alpha},Z],Y\rangle=0$. In any case,
$\langle\ad((1-\theta)\xi)Y,Z\rangle=0$ and the result is proved.
\end{proof}

Now define $\phi=\Exp(\frac{\pi}{\sqrt{2}\lvert\alpha\rvert}
(1+\theta)E_{\alpha}))$ for $\alpha\in\Phi$. It is well known that
$\phi\in N_K(\g{a})$. Moreover, we have

\begin{lemma}\label{thAdphi}
$\Ad(\phi)\left(\g{n}\ominus\left(\bigoplus_{\gamma\in\Phi}
(\g{g}_{\gamma}\oplus\g{g}_{2\gamma})\right)\right)
\subset\g{n}\ominus\left(\bigoplus_{\gamma\in\Phi}
(\g{g}_{\gamma}\oplus\g{g}_{2\gamma})\right)$.
\end{lemma}

\begin{proof}
The argument given in the proof of the previous lemma can be applied
here to show that $\ad((1+\theta)\xi)
\left(\g{n}\ominus\left(\bigoplus_{\gamma\in\Phi}
(\g{g}_{\gamma}\oplus\g{g}_{2\gamma})\right)\right)
\subset\g{n}\ominus\left(\bigoplus_{\gamma\in\Phi}
(\g{g}_{\gamma}\oplus\g{g}_{2\gamma})\right)$. Since
$\Ad(\phi)=\sum_{m=1}^\infty\frac{1}{m!}
(\frac{\pi}{\sqrt{2}\lvert\alpha\rvert})^m \ad((1+\theta)\xi)^m$,
the result follows.
\end{proof}

Finally, let $X\in\g{g}_\lambda$ with
$\lambda\in\Sigma^+\setminus(\Phi\cup2\Phi)$. It follows from Lemma \ref{thad(1-z)xi} that $A_\xi$ leaves $\g{n}\ominus(\oplus_{\gamma\in\Phi}(\g{g}_\gamma\oplus\g{g}_{2\gamma}))$ invariant. Thus, if $Z\in\g{a}\oplus(\oplus_{\gamma\in\Phi}(\g{g}_\gamma\oplus\g{g}_{2\gamma}))$ then $\langle\Ad(\phi)A_\xi X,Z\rangle_{AN}=\langle A_\xi\Ad(\phi) X,Z\rangle_{AN}=0$. Assume $Z\in\g{n}\ominus(\oplus_{\gamma\in\Phi}
(\g{g}_{\gamma}\oplus\g{g}_{2\gamma}))$. Using the previous two
lemmas, $\Ad(\phi)H_{\alpha}=-H_{\alpha}$ and
$\Ad(\phi)((1-\theta)E_{\alpha})=-(1-\theta)E_{\alpha}$ we get
\begin{align*}
\langle\Ad(\phi)A_\xi X,Z\rangle_{AN}
&= \frac{1}{2}\langle\Ad(\phi)A_\xi X,Z\rangle
 = \frac{1}{2}\langle A_\xi X,\Ad(\phi^{-1})Z\rangle\\
&= \langle A_\xi X,\Ad(\phi^{-1})Z\rangle_{AN}
 = \frac{1}{4}\langle[(1-\theta)\xi,X],\Ad(\phi^{-1})Z\rangle\\
&= \frac{1}{2}\langle[\Ad(\phi)(a_\alpha H_\alpha+(1-\theta)E_\alpha),\Ad(\phi)X],Z\rangle\\
&= -\frac{1}{4}\langle[(1-\theta)\xi,\Ad(\phi)X],Z\rangle
 = -\langle A_\xi\Ad(\phi)X,Z\rangle_{AN}.
\end{align*}
In particular this implies $\tr
A_{\xi}=a_{\alpha}\lvert\alpha\rvert^2(\dim \g{g}_{\alpha}+2\dim
\g{g}_{2\alpha})$.

We summarize all these calculations in the following

\begin{proposition}
Let ${S_{\Phi,V,a}}$ be the connected subgroup of $G$ whose Lie
algebra is $\g{s}_{\Phi,V,a}$. Let us write
$X_{\alpha}=\frac{1}{\lvert\alpha\rvert^2}H_{\alpha}
-a_{\alpha}E_{\alpha}$ and denote by $A_\xi$ the shape operator of
${S_{\Phi,V,a}}\cdot o$ with respect to a normal vector
$\xi\in\nu_o(S_{\Phi,V,a}\cdot o)$. We have:
\begin{enumerate}
\item If $\xi\in\g{a}_\Phi\ominus V$, then the restriction of $A_\xi$
to the linear subspace $V\oplus\left({\bigoplus}_{\gamma\in\Phi}\R
X_{\gamma}\right)\oplus
\left({\bigoplus}_{\gamma\in\Phi}(\g{g}_{\gamma}\ominus\R
E_{\gamma})\right)\oplus
\left({\bigoplus}_{\gamma\in\Phi}\g{g}_{2\gamma}\right)$ is zero and
the restriction of $A_\xi$ to $\g{g}_\lambda$ for
$\lambda\in\Sigma^+\setminus(\Phi\cup2\Phi)$ is
$\lambda(\xi)1_{\g{g}_\lambda}$.

\item If $\xi=a_\alpha H_\alpha+2E_\alpha$ with $\langle E_\alpha,E_\alpha\rangle=1$ we have $\langle\xi,\xi\rangle_{AN}=2+a_\alpha^2\lvert\alpha\rvert^2$ and:
\begin{enumerate}
\item The restriction of $A_{\xi}$ to %the linear subspace
$V\oplus\left({\bigoplus}_{\gamma\in\Phi\setminus\{\alpha\}}\R
X_{\gamma}\right)
\oplus\left({\bigoplus}_{\gamma\in\Phi\setminus\{\alpha\}}
(\g{g}_{\gamma}\ominus\R E_{\gamma})\right)\oplus
\left(\bigoplus_{\gamma\in\Phi\setminus\{\alpha\}}
\g{g}_{2\gamma}\right)$ is zero.

\item $A_{\xi}X_{\alpha}=a_{\alpha}\lvert\alpha\rvert^2
X_{\alpha}$.

\item We can decompose $\g{g}_{\alpha}$ as
$\g{g}_{\alpha}=\mathop{\rm Ker}(\ad(E_{\alpha}){\vert
\g{g}_{\alpha}})\oplus\ad(\theta E_{\alpha})(\g{g}_{2\alpha})$. The
restriction of $A_{\xi}$ to $\mathop{\rm Ker}(\ad(E_{\alpha}){\vert
\g{g}_{\alpha}})\ominus\R E_{\alpha}$ is
$a_{\alpha}\lvert\alpha\rvert^21_{\mathop{\rm
Ker}(\ad(E_{\alpha}){\vert \g{g}_{\alpha}})\ominus\R E_{\alpha}}$
and the dimension of $\mathop{\rm Ker}(\ad(E_{\alpha}){\vert
\g{g}_{\alpha}})\ominus\R E_{\alpha}$ is $\dim \g{g}_{\alpha}-\dim
\g{g}_{2\alpha}-1$. The subspace $\ad(\theta
E_{\alpha})(\g{g}_{2\alpha})\oplus\g{g}_{2\alpha}$ is invariant
under $A_{\xi}$ and $A_{\xi}$ acts with eigenvalues
\begin{equation*}
\frac{\lvert\alpha\rvert}{2}
\left(3a_{\alpha}\lvert\alpha\rvert
\pm\sqrt{8+a_{\alpha}^2\lvert\alpha\rvert^2}\right)
\end{equation*}
whose multiplicities are $\dim \g{g}_{2\alpha}$.

\item If $\phi=\Exp\left(\frac{\pi}{\sqrt{2}\lvert\alpha\rvert}
(1+\theta)E_{\alpha})\right)\in N_K(\g{a})$ then $\g{n}\ominus
\left(\bigoplus_{\gamma\in\Phi}
(\g{g}_{\gamma}\oplus\g{g}_{2\gamma})\right)$ is invariant by
$A_{\xi}$ and $\Ad(\phi)$, and $A_{\xi}\Ad(\phi)=-\Ad(\phi)A_{\xi}$.
In particular, if $A_{\xi}X=c X$, then
$A_{\xi}\Ad(\phi)X=-c\Ad(\phi)X$.
\end{enumerate}
\end{enumerate}
\end{proposition}

The mean curvature vector $\mathcal{H}$, defined with respect to an
orthonormal basis $\{e_i\}$ as $\mathcal{H}=\sum_i{I\!I}(e_i,e_i)$,
is in our case
\begin{equation*}
\mathcal{H}=2\pi_{\g{a}_\Phi\ominus V}(H_\delta)
+\sum_{\alpha\in\Phi}
\frac{a_{\alpha}\lvert\alpha\rvert^2}{2+a_\alpha^2\lvert\alpha\rvert^2}(\dim
\g{g}_{\alpha}+2\dim
\g{g}_{2\alpha})(a_{\alpha}H_{\alpha}+2E_{\alpha}),
\end{equation*}
where as usual $\pi_{\g{a}_\Phi\ominus V}$ denotes orthogonal
projection onto $\g{a}_\Phi\ominus V$.

We recall that the horocycle foliation is induced by the group $N$,
the action of the nilpotent part of some Iwasawa decomposition. In
this case we have:

\begin{corollary}
The orbits of the horocycle foliation are isometrically congruent to
each other and their shape operator with respect to a vector
$\xi\in\g{a}$ is given by
$A_\xi=\ad(\xi)_{\mid\g{n}}=\bigoplus_{\lambda\in\Sigma^+}
\lambda(\xi)1_{\g{g}_\lambda}$.
\end{corollary}

\begin{remark}
Let $M$ be a symmetric space of rank one. Assume that
$\Lambda=\{\alpha\}$. There are up to congruency two possible
hyperpolar homogeneous foliations, namely, the horosphere foliation,
which is the same as the horocycle foliation in this case, and the
solvable foliation, where $\Phi=\{\alpha\}$. In both cases the
foliation is by homogeneous hypersurfaces.

All the leaves of the horosphere foliation are congruent. The
principal curvatures of horospheres are $\lvert\alpha\rvert$ and
$2\lvert\alpha\rvert$ with multiplicities $\dim\g{g}_\alpha$ and
$\dim\g{g}_{2\alpha}$ respectively.

Now we consider the solvable foliation, whose leaves are the orbits
of the group $S_{\{\alpha\},\{0\}}$. If $\gamma$ is a geodesic
parametrized by unit speed such $\gamma(0)=o$ and $\dot\gamma(0)$ is
orthogonal to $S_{\{\alpha\},\{0\}}\cdot o$, then the path of
$\gamma$ is a section of this hyperpolar foliation. The principal
curvatures of the orbit $S_{\{\alpha\},\{0\}}\cdot\gamma(r)$ are
\begin{equation*}
-\lvert\alpha\rvert\tanh(\lvert\alpha\rvert r),\quad
-\frac{3\lvert\alpha\rvert}{2}\tanh(\lvert\alpha\rvert r)
\pm\frac{\lvert\alpha\rvert}{2}\sqrt{4-3\tanh(\lvert\alpha\rvert r)},
\end{equation*}
with multiplicities $\dim\g{g}_{\alpha}$, $\dim\g{g}_{2\alpha}$ and
$\dim\g{g}_{2\alpha}$, respectively.

For cohomogeneity one homogeneous foliations see \cite{BT03}.
\end{remark}

%%%%%%%%%%%%%%%%%%%%%%%%%% Bibliography %%%%%%%%%%%%%%%%%%%%%%%%%%%

\end{document}